\documentclass[12pt]{article}
\usepackage{geometry}

\usepackage{amsmath,amsthm,amscd,amssymb,amsfonts,bm}
\usepackage{latexsym,lscape,epsfig,nicefrac,hyperref,cite,mathtools,todonotes}
\usepackage{mathrsfs}

\usepackage[title]{appendix}

\newcommand{\N}{\mathbb{N}}
\newcommand{\Z}{\mathbb{Z}}

\newcommand{\R}{\mathbb{R}}

\newcommand{\E}{\mathbb{E}}

\newcommand{\supp}{\text{supp}}

\newtheorem{info}{}
\newtheorem{theorem}[info]{Theorem}
\newtheorem{corr}[info]{Corollary}

\newtheorem{defin}[info]{Definition}

\newtheorem{lem}[info]{Lemma}
\newtheorem{prop}[info]{Proposition}
\newtheorem{remark}[info]{Remark}
\numberwithin{info}{section}

\numberwithin{equation}{section}
\renewcommand{\[}{\begin{equation}}
	\renewcommand{\]}{\end{equation}}

\makeatletter

\newcommand{\Lam}{\Lambda}
\newcommand{\al}{\alpha}
\newcommand{\be}{\beta}
\newcommand{\lam}{\lambda}
\newcommand{\Om}{\Omega}

\renewcommand{\P}{\mathbb{P}}
\renewcommand{\cal}{\mathcal}
\renewcommand{\frak}{\mathfrak}
\newcommand{\sq}{\sqrt}

\newcommand{\<}{\langle}
\renewcommand{\>}{\rangle}
\newcommand{\Sum}{\mathrm{\Sum}}

\newcommand{\To}{\Rightarrow}

\newcommand{\D}{\nabla}
\newcommand{\vol}{\mathrm{vol}}

\newcommand{\disteq}{\stackrel{d}{=}}
\renewcommand{\l}{\left}
\renewcommand{\r}{\right}

\renewcommand{\b}{\bm}
\newcommand{\tr}{\mathrm{tr}}

\newcommand{\ZNxi}{Z_{N,\mathrm{Sph}}}
\newcommand{\ZNxib}{Z_{N,\beta,\mathrm{Sph}}}

\begin{document}
	\title{The Free Energy of the Elastic Manifold}
	
	\author{Gérard Ben Arous, Pax Kivimae}
	
	\maketitle
	
	\begin{abstract}
		This is the first of a series of three papers about the Elastic Manifold model. This classical model proposes a rich picture due to the competition between the inherent disorder and the smoothing effect of elasticity. In this paper, we prove a Parisi formula, i.e. we compute the asymptotic quenched free energy and show it is given by the solution to a certain variational problem.
		
		This work comes after a long and distinguished line of work in the Physics literature, going back to the 1980’s (including the foundational work by Daniel Fisher \cite{fischerRG}, Marc M\'{e}zard and Giorgio Parisi \cite{mezardparisi,mezardparisi2}, and more recently by Yan Fyodorov and Pierre Le Doussal \cite{fyodorov-manifold,fyodorov-manifold-minimum}. Even though the mathematical study of Spin Glasses has seen deep progress in the recent years, after the celebrated work by Michel Talagrand \cite{talagrandIsingOG,talagrandOG}, the Elastic Manifold model has been studied from a mathematical perspective, only recently and at zero temperature. The annealed topological complexity has been computed, by the first author with Paul Bourgade and Benjamin McKenna \cite{gerardbenpaul,gerardbenpaulCompanionPaper}. Here we begin the study of this model at positive temperature by computing the quenched free energy. 
		
		We obtain our Parisi formula by first applying Laplace's method to reduce the question to a related new family of spherical Spin Glass models with an elastic interaction. The upper bound is then obtained through an interpolation argument initially developed by Francisco Guerra \cite{guerraOG} for the study of Spin Glasses. The lower bound follows by adapting the cavity method along the lines explored by Wei-Kuo Chen \cite{weikuo} and the multi-species synchronization method of Dmitry Panchenko \cite{panchenkoms}. In our next papers \cite{Paper2,Paper3} we will analyze the consequences of this Parisi formula.
	\end{abstract}
	
	\tableofcontents
	
\section{Introduction}

The study of disordered elastic systems is an area of long-standing importance, whose breadth contains a large number of well-studied models tied together by a rich set of new phenomena. The theory concerns the behavior of random manifolds, under an elastic self-interaction and subjected to a disordered potential. The behavior of the resulting manifold is thus shaped by a competition between an elastic self-interaction, which tends to smoothen and flatten it, and the disordered potential, which promotes rougher configurations. Typical examples involve random polymers, as well as the $d$-dimensional domain walls in the $(d+1)$-dimensional random field Ising model. For more background on these models, further examples, and an overview of their physical phenomenology, we suggest the review \cite{elasticReview}.

Our focus will be on the elastic manifold, a prototypical model in this class. The study of this model was initiated by Fischer \cite{fischerRG}, and extended in subsequent works \cite{balentsRG,balents1996large}, to be used as an effective model for the behavior of elastic surfaces.

Let us first describe the continuous model. The phase space of this model is the collection of functions $\b{u}:\Omega \to \R^N$, where $\Omega=[1,L]^d\subseteq \R^d$, and the Hamiltonian is given by:
\[\cal{H}(\b{u})=\frac{1}{2}\int_{\Omega} \left(\mu \|\b{u}(x)\|^2-t( \b{u}(x),\Delta \b{u}(x))\right)dx^d+\int_{\Omega}V(x,\b{u}(x))dx^d,\]
where $V(x,\b{u}(x))$ is the disorder potential, which is taken to be a centered Gaussian process on $\Omega\times \R^N$ with covariance chosen to be white in $x$ and isotropic in $u$, i.e.  given for $(x,u),(x',u')\in \Omega\times \R^N$ by
\[\E[V(x,u)V(x',u')]=\delta(x-x')B(\|u-u'\|^2),\]
where $\delta(x-x')$ denotes the Dirac $\delta$-function and $B$ is some fixed correlation function.
The model is parametrized  by two important real parameters: the mass $\mu$  and the elastic interaction strength $t$.

A deep understanding of this model came from the work by Mézard and Parisi \cite{mezardparisi}. They studied the behavior of this model in the high-dimensional (i.e., large $N$) limit and were able to compute the quenched free energy by employing the (beautiful but non rigorous) replica symmetry breaking (RSB) scheme classical in the study of spin glasses. Differentiating the free energy, they were then able to obtain results for a number of statistics of the model. This line of works includes many other important contributions \cite{engelOG,balents1996large,balentsRG}. 

In this paper, we compute the limiting quenched free energy of this model in a discrete variant, as proposed by the recent works by Fyodorov and Le Doussal \cite{fyodorov-manifold,fyodorov-manifold-minimum}. This limiting free energy is given in terms of a new variational formula. While the new variational formula given here is complex, it is also very general, applying to more general discrete surface models than the elastic manifold.

We leave the analysis of this functional to our companion paper \cite{Paper2}, where we show that we can considerably reduce the complexity of the variational formula proven here. Obtaining the results for from this simplification in the continuum limit, and in particular confirming physical predictions for certain scaling exponents, is left for our third companion paper \cite{Paper3}. Combined, we can confirm much of the celebrated picture obtained in the original work of \cite{mezardparisi}.

Finally, we obtain similar results for a related family of spherical models. This generalization is in fact required to show our results for the elastic manifold, but it is also an interesting variant per se, due to the non-trivial effect an external field term may have on the Gibbs measure in this case, the so-called pinning/depinning transition.

\subsection{Results on the Limiting Free Energy}

Here we give our main results, the first of which is a formula for the large-$N$ asymptotics of the quenched free energy of the discrete Elastic manifold model (Theorem \ref{theorem:intro:bad main euclidean}). As mentioned above, on our way to proving this result, we need to introduce a new spherical variant of the Elastic manifold model. Our computation of the asymptotic free energy for this model (Theorem \ref{theorem:intro:bad parisi spherical}), is our second main result and a major component in the proof of Theorem \ref{theorem:intro:bad main euclidean}.

Thus now we introduce both models, and state our results. We begin with the discrete elastic manifold model as studied by Fyodorov and Le Doussal \cite{fyodorov-manifold}. Fix integers $L\ge 1$ and $d\ge 0$, which will be referred to as the length and internal dimension, respectively, and positive numbers $\mu $ and $t $, referred to as the mass and interaction strength. We will denote by $\Om$ the graph $[[1,L]]^d\subseteq \Z^d$ understood in the periodic sense. Next, we let $V_N:\R^N\to \R$ be a centered Gaussian random field with isotropic covariance given for $u,v\in \R^N$ by
\[\E[V_N(u)V_N(v)]=N B\left(\|u-v\|^2_N\right).\]
where here $\|x\|^2_N=(x,x)_N$ where $(x,y)_N=N^{-1}\sum_{i=1}^{N}x_iy_i$ and $B:[0,\infty)\to [0,\infty)$ is some fixed function.

An old  result of Schoenberg \cite{schoenberg} gives a full description of the possible functions $B$.$B$ must admit a representation of the form
\[B(x)=c_0+\int_0^{\infty}\exp(-\lam^2 x)\nu(d\lam),\label{eqn:B-decomposition}\]
where $\nu$ is a finite non-negative measure on $(0,\infty)$ and $c_0\ge 0$. For convenience, we will assume that there is $\epsilon>0$ such that $\int_0^\infty \exp(\lam^2 \epsilon) \nu(d\lam)<\infty$. This implies, in particular, that $B$ extends to a smooth function on $(-\epsilon,\infty)$.

The model's Hamiltonian is then a function on point configurations $\b{u}\in (\R^N)^\Om$. Such configurations $\b{u}$ may be thought of as functions on the discrete graph $\Omega$, and for $x\in \Om$, we will use the notation $\b{u}(x)\in \R^N$ to denote the point such that $\b{u}(x)_i=\b{u}_{(i,x)}$. However, when $\b{f}$ is a function taking values in $A^{\Om}$, for some space $A$, we will use the notation $f_x(y):=\b{f}(y)(x)$. Finally, when $c\in \R$ we will use the notation $\b{c}\in \R^\Om$ to denote the constant function $\b{c}(x)=c$.

With these notations set, we define our Hamiltonian for $\b{u}\in (\R^N)^\Om$ by
\[\cal{H}_N(\b{u})=\frac{1}{2}\sum_{x,y\in \Om}(\mu I-t \Delta)_{xy}(\b{u}(x),\b{u}(y))+\sum_{x\in \Om}V_{N,x}(\b{u}(x))+\sq{N}h\sum_{x\in \Om}\b{u}_1(x),\label{eqn:def:main-model-D}\]
where here $(V_{N,x})_{x\in \Om}$ is a family of i.i.d copies of $V_N$, and $\Delta$ denotes the graph Laplacian of $\Omega$ (where we take the sign convention where $\Delta$ is negative semi-definite, following \cite{mezardparisi,fyodorov-manifold}). Then for any inverse temperature $\beta>0$, we define the associated partition function
\[Z_{N,\beta}=\int_{(\R^{N})^\Om}e^{-\beta\cal{H}_N(\b{u})}d\b{u},\]
where $d\b{u}$ is the standard Lebesgue measure on $(\R^N)^{\Om}$. Given this, we define the (normalized) free energy of the model as 
\[f_{N,\beta}:=|\Om|^{-1}N^{-1}\log Z_{N,\beta}\]

\subsubsection{Results for the Euclidean Model}

Our main result is the convergence of $f_{N,\beta}$ a.s. to a deterministic value given by a Parisi-type variational problem, which we now define, beginning with some notation. If one has $\b{u}\in \R^\Om$, the notation $\text{di}(\b{u})$ will refer to the diagonal matrix in $\R^{\Om\times \Om}$, such that $\text{di}(\b{u})_{xy}=\delta_{xy}\b{u}(x)$ for $x,y\in \Om$. When $D\in \R^{\Om\times \Om}$, we will use the shorthand $D+\b{u}:=D+\text{di}(\b{u})$. When $A$ and $B$ are matrices, we will use the notation $A>B$ to denote that is positive definite.

Now, we will need some functions associated to the discrete Laplacian $\Delta$. The existence of and claims made about these functions will be verified in Proposition \ref{prop:appendix:K and Lambda} below.
\begin{defin}
\label{defin:K and Lam for Delta}
Let $\b{K}^{-t\Delta}:(0,\infty)^{\Om}\to (0,\infty)^{\Om}$ be the function defined by the relation
\[[\left(\b{K}^{-t\Delta}(\b{u})-t\Delta\right)^{-1}]_{xx}=\b{u}(x) \text{ for } x\in \Om.\label{eqn:def:KtDelta}\]
Equivalently, $\b{K}^{-t\Delta}$ is the functional inverse of the function $\b{u}\mapsto ([\left(\b{u}-t\Delta\right)^{-1}]_{xx})_{x\in \Om}$. We also define for $\b{u}\in (0,\infty)^{\Om}$, the function
\[\Lambda^{-t\Delta}(\b{u})=\frac{1}{|\Om|}\left(\sum_{x\in \Om}K_x^{-t\Delta}(\b{u})\b{u}(x)-\log \det(\b{K}^{-t\Delta}(\b{u})-t\Delta)\right),\label{eqn:def:LambdaDelta}\]
This is the Legendre transform of the strictly concave function $\b{u}\mapsto \frac{1}{|\Om|}\log \det(\b{u}-t\Delta)$,so that in particular $|\Om|\D \Lambda^{-t\Delta}(\b{u})=\b{K}^{-t\Delta}(\b{u})$.
\end{defin}

Next, we specify the domain for the functional of our variational problem. For $S\subseteq \R$ we let $\cal{P}(S)$ denote the set of probability measures on $S$.

\begin{defin}
Fix $\b{q}\in (0,\infty)^{\Om}$ and let $q_{t}=\frac{1}{|\Om|}\sum_{x\in \Om}\b{q}(x)$. The set $\mathscr{Y}(\b{q})$ is defined as the set of pairs $(\zeta,\b{\Phi})$, where $\zeta\in \cal{P}([0,q_t])$ and $\b{\Phi}:[0,q_t]\to \prod_{x\in \Om}[0,\b{q}(x)]$ is a coordinate-wise non-decreasing function, which jointly satisfy the following conditions: For each $s\in [0,q_t]$ we have
\[\frac{1}{|\Om|}\sum_{x\in \Om}\b{q}(x)^{-1}\Phi_x(s)=q_t^{-1}s,\label{eq:condition for Psi}\]
and there there is some $0<q_*<q_t$ is such that $\zeta([0,q_*))=0$ and $\Phi_x(q_*)<\b{q}(x)$ for each $x\in \Om$.
\end{defin}

We are now ready to proceed to defining our functional. Given $(\zeta,\b{\Phi})\in\mathscr{Y}(\b{q})$, we define the continuous (coordinate-wise) non-increasing function $\b{\delta}:[0,q_t]\to [0,\infty)^{\Om}$ by \footnote{Regarding the existence $\Phi'_x$, note that as each $\Phi_x$ is non-decreasing, (\ref{eq:condition for Psi}) implies that each $\Phi_x$ is absolutely continuous. In particular, $\Phi_x'$ exists a.e. and is Lebesgue integrable.}
\[\delta_x(s)=\int_{s}^{q_t}\zeta([0,u])\Phi'_x(u)du.\] With these choices, we now define our functional:
\[ 
\begin{split}
	\cal{P}_{\beta,\b{q}}(\zeta,\b{\Phi})&=\frac{1}{2}\bigg(\log\left(\frac{2\pi}{\beta}\right)+\frac{\beta h^2}{\mu }+\Lambda^{-t\Delta}(\beta \b{\delta}(q_*))+\frac{1}{|\Om|}\sum_{x\in \Om}\bigg(-\mu \b{q}(x)\\ &+\int_{0}^{q_*} \beta K^{-t\Delta}_x(\beta\b{\delta}(q))\Phi'_x(q)dq-2\beta^2 \int_0^{q_t}\zeta([0,u])B'(2(q_x-\Phi_x(u)))\Phi'_x(u)du\bigg)\bigg).
\end{split}
\]

Our main result is then the following free energy computation.
\begin{theorem}
	\label{theorem:intro:bad main euclidean}
	We have a.s. that
	\[\lim_{N\to \infty}f_{N,\beta}= \lim_{N\to \infty}\E f_{N,\beta}=\sup_{\b{q}\in (0,\infty)^{\Om}}\left(\inf_{(\zeta,\b{\Phi})\in \mathscr{Y}(\b{q})}\cal{P}_{\beta,\b{q}}(\zeta,\b{\Phi})\right).\]

\end{theorem}

\begin{remark}
	\label{remark:beginning point doesn't matter part}
	We note that for $s\in(q_*,q_t)$ we have 
	\[\delta_x(s)=\Phi_x(q_t)-\Phi_x(s)=\b{q}(x)-\Phi_x(s).\]
	Thus for $q_*'\in (q_*,q_t)$ we see that
	\[\frac{1}{|\Om|}\sum_{x\in \Om}\int_{q_*}^{q_*'} \beta K^{-t\Delta}_x(\beta \b{\delta}(q))\Phi'_x(q)dq=\Lambda^{-t\Delta}(\beta \b{\delta}(q_*))-\Lambda^{-t\Delta}(\beta \b{\delta}(q_*')).\]
	In particular, $\cal{P}_{\beta,\b{q}}(\zeta,\b{\Phi})$ is independent of the choice of $q_*$.
\end{remark}

\subsubsection{Results for the Spherical Model}

Before discussing the proof of Theorem \ref{theorem:intro:bad main euclidean}, we will introduce a similar result for a more general family of spherical models. This result is not only a major component in the proof of Theorem \ref{theorem:intro:bad main euclidean}, but also of independent interest due to the spherical models exhibiting a more non-trivial dependence on the external field.

The model here is similar to the above model, except that instead of taking an isotropic field on $\R^N$ in the presence of a quadratic potential, we instead consider isotropic fields restricted to the sphere. Such isotropic fields are more well known, as they form the family of spherical spin glass models, which has seen extensive study.

In particular, we fix, for each $N\ge 1$, a choice of centered Gaussian random field $H_N:S_N\to \R$ with isotropic covariance given for $\sigma,\tau\in S_N$ by
\[\E[H_N(\sigma)H_N(\tau)]=N\xi((\sigma,\tau)_N),\]
where $\xi:[-1,1]\to \R$ is some fixed function. As before, it is also a result of Schoenberg \cite{schoenbergSpheres} that for such a family to exist for all $N\ge 1$, it is necessary and sufficient that $\xi$ is an analytic function on $[-1,1]$ with positive coefficients. That is, there are coefficients $\beta_p\ge 0$ such that
\[\xi(x)=\sum_{p=0}^{\infty}\beta_p^2 x^p,\label{eqn:xi-decomposition}\]
and such that $\xi(1)<\infty$. We will assume as before that for sufficiently small $\epsilon>0$, that $\xi(1+\epsilon)<\infty$, so that $\xi$ extends to an analytic function on $(-1-\epsilon,1+\epsilon)$.

Let us now take some finite set $\Om$, and fix a positive semi-definite matrix $D\in \R^{\Om\times \Om}$. For each $x\in \Om$, fix a mixing function $\xi_x$ in (\ref{eqn:xi-decomposition}), and let $H_{N,x}$ be its associated Hamiltonian (chosen independently for each $x\in \Om$). Finally, we choose an external field vector $\b{h}\in \R^{\Om}$. We define the spherical version of our Hamiltonian by taking, for $\b{u}\in S_N^{\Om}$
\[\cal{H}_{N,\mathrm{Sph}}(\b{u})=\frac{1}{2}\sum_{x,y\in \Om} D_{xy}(\b{u}(x),\b{u}(y))+\sum_{x\in \Om}H_{N,x}(\b{u}(x))+\sqrt{N}\sum_{x\in \Om}\b{h}(x)\b{u}_1(x).\label{eqn:def:intro-spherical-model}\]
Associated with this, we may define the partition function
\[\ZNxib=\int_{S_N^{\Om}}e^{-\beta \cal{H}_{N,\mathrm{Sph}}(\b{u})}\omega(d\b{u}),\]
where here $\omega(d\b{u})$ the surface measure on $S_N^{\Om}$ induced by the inclusion $S_N^{\Om}\subseteq (\R^N)^{\Om}$. We similarly define the models normalized free energy as
\[f_{N,\beta,\mathrm{Sph}}=|\Om|^{-1}N^{-1}\log \ZNxib.\]

To give the Parisi-functional in this case, we will need a generalization of Definition \ref{defin:K and Lam for Delta}.

\begin{defin}
	\label{defin:K and Lam for general}
	Let $K^D:\{\b{u}\in \R^{\Om}:D+\b{u}>0\}\to (0,\infty)^{\Om}$ be the function defined by
	\[[\left(\b{u}+\b{K}^D(\b{u})\right)^{-1}]_{xx}=\b{u}(x) \text{ for } x\in \Om.\label{eqn:def:KD}\]
	We also define for $\b{u}\in \{\b{u}\in \R^{\Om}:D+\b{u}>0\}$, the function
	\[\Lambda^{D}(\b{u})=\frac{1}{|\Om|}\left(\sum_{x\in \Om}K_x^{D}(\b{u})\b{u}(x)-\log \det(D+\b{K}^{D}(\b{u}))\right).\label{eqn:def:Lambda}\]
\end{defin}

If we denote $\mathscr{Y}:=\mathscr{Y}(\b{1})$, then for $(\zeta,\b{\Phi})\in \mathscr{Y}$ and any $q_*$ such that $\zeta([0,q_*])=1$ and $\Phi_x(q_*)<1$ for all $x\in \Om$, we may define the functional
\[
\begin{split}
	\cal{B}_\beta(\zeta,\b{\Phi})=&\frac{1}{2}\bigg(\log\left(\frac{2\pi}{\beta}\right)+\Lambda^D(\beta\b{\delta}(q_*))+\frac{1}{|\Om|}\sum_{x\in \Om}\bigg(\int_{0}^{q_*}\beta K_x^D(\beta \b{\delta}(q))\Phi'_x(q)dq\\
	&+\int_0^{1}\beta^2\zeta([0,u])\xi_x'(\Phi_x(q))\Phi'_x(q)dq+\sum_{y\in \Om}\beta^2 [D+\b{K}^D(\beta\b{\delta}(0))^{-1}]_{xy}\b{h}(x)\b{h}(y)\bigg)\bigg).\label{eqn:def:hat-A}
\end{split}
\]

As in Remark \ref{remark:beginning point doesn't matter part}, it may be checked that this functional is independent of the choice of $q_*$. Our main result for this class is then the following computation for the quenched free energy.

\begin{theorem}
	\label{theorem:intro:bad parisi spherical}
	We have a.s. that
	\[\lim_{N\to \infty}f_{N,\beta,\mathrm{Sph}}=\lim_{N\to \infty}\E f_{N,\beta,\mathrm{Sph}}= \inf_{(\zeta, \b{\Phi})\in \mathscr{Y}}\cal{B}_\beta(\zeta,\b{\Phi}).\]
\end{theorem}

\subsection{Outline of the Proof}

We now outline the major steps in our proofs, as well as the general structure of this paper. In Section \ref{section:non-transitive main theorem}, we give the proof of with Theorem \ref{theorem:intro:bad main euclidean}, assuming Theorem \ref{theorem:intro:bad parisi spherical}. To do this, we first reduce to the case of $\beta=1$ and $h=0$ by reparametrization. Then, we view the partition function $Z_{N,\beta}$ as an integral over the contributions of products of spheres of fixed radii using the coarea formula. The Euclidean model, restricted to a product of spheres, may then be related to an inhomogeneous spherical model by rescaling the radius. This model may then be treated by Theorem \ref{theorem:intro:bad parisi spherical}. In particular, we may understand the expected exponential order of this contribution by employing Theorem \ref{theorem:intro:bad parisi spherical}. This task is completed in Corollary \ref{corr:formula for spherical slice}, which essentially shows that the contribution is precisely 
\[\inf_{(\zeta,\b{\Phi})\in \mathscr{Y}(\b{q})}\cal{P}_{\beta,\b{q}}(\zeta,\b{\Phi}).\]
With this result, Theorem \ref{theorem:intro:bad main euclidean} heuristically follows by Laplace's method.
However, as each of these contributions are random, replacing them with their expected exponential order is non-trivial. Instead, we show a continuity estimate which shows that the surface-area contribution over a given product of spheres is not significantly changed by slightly thickening each sphere. This allows us to work with the sum of contributions over products of thickened spheres instead. With some simple decay estimates to reduce this to a finite sum, we may then control their fluctuations by standard Gaussian concentration inequalities. Taking the thickness to zero then gives Theorem \ref{theorem:intro:bad main euclidean}.

We now move on to discuss the proof of Theorem \ref{theorem:intro:bad parisi spherical}. The first step, is to show we may express the asymptotic free energy in terms of a different functional (see Theorem 
\ref{theorem:spherical parisi continuum}). This new functional is somewhat artificial, and even in the one-site case does not coincide with the functional obtained by Crisanti and Sommers \cite{crisantisommersOG}. Such auxiliary functionals however have appeared both in the rigorous computations of the free energy in the one-site case by Talagrand \cite{talagrandOG} and in the related multi-species spherical model by \cite{erik,erikCS}, though if it has a physical meaning is still unclear. Proving that these expressions are equivalent (i.e that Theorem \ref{theorem:spherical parisi continuum} implies Theorem \ref{theorem:intro:bad parisi spherical}) is the focus of Section 3.

Now we are left with Theorem \ref{theorem:spherical parisi continuum}. On a broad scale, the mechanical basis for our proof begins with a combination of the multi-species synchronization method introduced in the computation of free energy in the convex multi-species spin glass model by Panchenko \cite{panchenkoms}, supplemented by a number of apparatuses used in the computation of the free energy in the unipartite spherical spin glass models \cite{talagrandOG,talagrandGuerraPositivity,panchenko-Ultrametricity,weikuo}.

Specifically, we are able to obtain an upper bound by employing a Guerra-style interpolation argument, and our lower bound employs the extended cavity method, as introduced by Chen \cite{weikuo}, as well as the Aizenman-Sims-Star scheme \cite{ass1,ass2}. Both of these require us to first add a slight perturbation term to the Hamiltonian, small enough to not affect the free energy, but strong enough to force a number of strong properties for the Gibbs measure. For both the upper and lower bound, this allows us to replace the free energy with a difference defined in terms of certain ``synchronized" Ruelle probability cascades. This reduces the problem to instead computing asymptotics for a certain functional over such functionals, which is where effects of the confining potential $\mu I-t \Delta$ come into play, as computing such functionals on a non-product measure adds a number of technicalities. 

The bulk of this analysis is contained in Sections \ref{section:RPC-intro} and \ref{section:RPC-gg-iden}, which also recall the majority of the technological results and constructions we use from this theory.	Roughly, the first section, Section \ref{section:RPC-intro}, involves the study of Ruelle probability cascades and the interpretation of the Parisi functional in terms of them, as well as some further approximation and rigidity results. The second, Section \ref{section:RPC-gg-iden}, focuses more on the multi-species Ghirlanda-Guerra identities, their relation to ``synchronized" Ruelle probability cascades, and the perturbation of the Hamiltonian.

To the reader not versed in these methods, we suggest the essentially self-contained and incredibly well-written monograph of Panchenko \cite{panchenko}, which we will also use to set the majority of our notation. We will also crucially use the synchronization method of \cite{panchenkoms}. We will also reference often the work \cite{erik}, which is quite thoroughly written. Unfortunately, their multi-site model is essentially orthogonal to ours, so we can only use some technical lemmas. 

With these tools developed, we finally complete the proof of Theorem \ref{theorem:intro:bad parisi spherical}, which is the subject of Section \ref{section:lowerbound:cavity}. As discussed above, this is heavily reliant on the tools in Sections \ref{section:RPC-intro} and \ref{section:RPC-gg-iden}. However, with these tools, the upper bound follows immediately from a Guerra-style interpolation and an application of Talagrand's positivity principle \cite{talagrandGuerraPositivity}. The lower bound is more involved, as we need to use the extended cavity method of Chen \cite{weikuo}, combined with the multi-species synchronization approach of Panchenko \cite{panchenkoms}.

Finally, we provide three appendices. Appendix A provides some basic results about Gaussian processes used throughout the text. Appendix B provides some properties of the functionals $\cal{A}$ and $\cal{B}$. Appendix C provides some important technical results about the functions $\b{K}^{D}$ and $\Lambda^D$.

\subsection{Related Works}

We first review the results for the one-site case (i.e. when $|\Om|=1$), where much more is known. Beginning with the spherical model (\ref{eqn:def:intro-spherical-model}), an expression for the quenched free energy, which henceforth we will simply refer to as the free energy for simplicity, was derived on a physical basis by Crisanti and Sommers \cite{crisantisommersOG}, employing the replica-symmetry breaking scheme used by Parisi \cite{parisiOG} to treat the Ising spin glass model. A rigorous derivation of these formulas in the even case (i.e. when $\xi$ is an even function) was obtained by Talagrand \cite{talagrandOG, talagrandIsingOG}, cleverly employing a version of the interpolation method used by Guerra to obtain the upper bound in the Ising case \cite{guerraOG}. The general Ising case was later obtained by Panchenko \cite{panchenkoUnipartite}, with the general spherical case soon after being obtained by Chen \cite{weikuo}.

Moving onto the elastic model, an asymptotic formula for the free energy was essentially conjectured by M{\'e}zard and Parisi \cite{mezardparisi,mezardparisi2}, and was further studied in the one-site case by Engel \cite{engelOG}. The one-site computation was then made rigorous in the note of Klimovsky \cite{planeonesiteOG}, by employing the computation of the spherical case.

We also mention that a number of computations for the free energy of different multi-site models have recently been obtained. Many of these rely on the foundational work of Panchenko \cite{panchenkoms}, who computed the free energy of the multi-species Ising spin glass model in the convex case. Following this, Bates and Sohn \cite{erik,erikCS} recently computed the free energy for the multi-species spherical spin glass model, again in the convex case. Finally, we mention the work of Ko \cite{justin}, who computed the free energy in the spherical case with contained overlaps.

Moving back to the case of the elastic manifold, recent work of Xu and Zeng \cite{xuzengelasticmanifold} established a formula for the ground-state energy in the replica symmetric regime. These coincide with the general formula for the ground-state energy we obtain in our companion works \cite{Paper2, Paper3}. However, unlike the previously mentioned works, the method of \cite{xuzengelasticmanifold} does not depend on a Guerra-style interpolation. 

Instead, their result is based on the (annealed) complexity, which is roughly the logarithm of the expected number of critical points, possibly subject to some conditions. For the elastic manifold, the complexity was first computed by Fyodorov and Le Doussal \cite{fyodorov-manifold}. This builds on early computations of the complexity in the one-site case by Fyodorov \cite{Fyo04}, which was followed by Fyodorov and Williams who computed the complexity of only local minima \cite{FW07}. The rigor of the results coming from works have mostly been confirmed in the work of Bourgade, McKenna, and the first author \cite{gerardbenpaul}, using a general framework they developed to treat complexity computations in \cite{gerardbenpaulCompanionPaper}.

Typically, the complexity is positive, so that one expects an exponentially large number of critical points. However, typically when the system is subject to a suitably strong external field or confining potential, the complexity may vanish, and more so the expected number of critical points or local minima may tend to one. This possibility was first noticed by Cugliandolo and Dean \cite{topologytrivialFirst} in a simple case of the one-site spherical model, and has been coined "topological trivialization" by Fyodorov and Le Doussal \cite{topologytrivialization2} who studied the behavior of in more detail.

If one restricts the count of critical points to those whose energy is near a given value, then the computation of the ground-state energy amounts to refining the computation of the complexity to vanishing order. This approach was pioneered by Fyodorov in the case of pure one-site spherical models \cite{topology3}, and indeed is the method used by \cite{xuzengelasticmanifold} to obtain their formula. This approach has since been generalized to mixed one-site spherical models by Belius, \v{C}ern\'{y}, Nakajima and Schmidt \cite{topology4}, and finally to the case of multi-species spherical models by Huang and Selke \cite{topologyms}.

Returning more to related work on complexity, the work \cite{FW07} actually studies a more general model than (\ref{eqn:def:main-model-D}) where the confining term $\frac{\mu}{2}\|x\|^2$ is related by a more general term $NU\l(\frac{1}{2}\|x\|^2_N\r)$ where $U$ is an increasing and convex function. Another common generalization is to instead weaken the condition of isotropy of $V_N$ to assume that $V_N$ only possesses isotropic increments. These, and similar models, were studied extensively by Fyodorov et. al. in \cite{FS07, FB08,fyodorovepointinabox}, many of whose results were later confirmed and extended by Auffinger and Zeng \cite{tucazeng1,tucazeng2}.

Much more is known about the complexity of the spherical case. In particular, the complexity in this case was computed in works by Auffinger, \v{C}ern\'{y}, and one of the authors \cite{tucapure,tucamixed}. These results match those obtained earlier in physics literature \cite{complexityOG,complexityminimaOG}, for which they correspond to the zero-temperature case of a more general computation for the complexity of TAP states. 

It is at this point that one may ask if computations of the ground state energy can be obtained outside of the topological trivialization regime. For this, it is important to introduce the distinction between the annealed complexity, and the quenched complexity, which is given by the typical exponential order of the number of critical points rather than its average. In general, it is known that these two quantities cannot coincide in general. For example, they are known to differ at sufficiently low energy for $2$RSB mixed spherical spin glass models \cite{chensen}. 

It is a remarkable result then due to Subag \cite{subag} that in the case of the pure $p$-spin glass model, annealed complexity of local minima concentrates around its average in any level-set above the ground-state energy. This was later extended by Subag and Zeitouni \cite{subagofer} to general critical points in the pure model, under mild assumptions on $p$, and further to low-lying local minima for a certain class of $1$RSB mixed spin glass models by Subag, Zeitouni, and the first author in \cite{subag1RSB}. Not only does this lead to a computation of the ground-state energy, but such results have led to a more complete understanding of the geometry of the low-temperature Gibbs measure, as shown in follow-up works by Subag and Zeitouni \cite{pspin-second-application1, pspin-second-application2}. In addition, these have played an important role in studying the associated Langevin dynamics, as in the work of Gheissari and Jagannath \cite{pspin-second-application3}. In addition, the work of Jagannath and the first author \cite{aukoshgerardshattering} discusses further ramifications of this study to concepts such as shattering and metastability. Finally, we mention a new novel approach to the model’s Langevin dynamics by Gheissari and Jagannath and the first author \cite{boundingflows}.

However, outside of these special cases, such concentration results are scarce, with the only full extension being to the non-gradient generalization of the spherical pure $p$-spin model by the second \cite{kivimae-non-gradient}, some partial results on the spherical bipartite model \cite{kivimae-bipartite}, both by the second author, and a general result for the total local minima of the mixed spin glass model by Belius and Schmidt \cite{david}, though with no control on their energy values.

\subsection{Notation}

Here we will set some notation and conventions, beginning with the thermodynamic notation. With a choice of topological space $\Sigma$, a measurable function $H:\Sigma\to \R$, and positive Borel measure $m$, we will define the partition function associated to $(H,m)$ as
\[Z(H,m)=\int_{\Sigma}\exp(H(x))m(dx),\]
when the integral on the right is finite. This definition encompasses all the objects referred to as partition functions above. For example, the partition function $Z_N$ coincides with $Z(-\cal{H}_N,d\b{u})$. When $Z(H,m)$ is finite, we define the Gibbs measure associated to $(H,m)$ as the probability measure on $\Sigma$, $G_{(H,m)}$, given by
\[G_{(H,m)}(dx)=\frac{\exp(H(x))m(dx)}{Z(H,m)}.\]
For a measurable function $f:\Sigma \to \R$, we will often use the notation
\[\<f(x)\>_{H,m}=G_{(H,m)}(f):=\int_\Sigma f(x)G_{(H,m)}(dx).\]
Associated as well, for any $n\ge 1$, is the (replicated) Gibbs measure on $G_{(H,m)}^{\otimes n}$ on $\Sigma^{n}$. When one instead has a function $f:\Sigma^n\to \R$ we will also use the mildly abusive notation
\[\<f(x_1,\dots,x_n)\>_{H,m}=G_{(H,m)}^{\otimes n}(f)=\int_{\Sigma^n} f(x^1,\dots, x^n)\prod_{i=1}^nG_{(H,m)}(dx^i),\]
and in the special case of $n=2$ we will often write
\[\<f(x,x')\>_{H,m}=G_{(H,m)}^{\otimes 2}(f)=\int_{\Sigma^2}f(x,x')G_{(H,m)}(dx)G_{(H,m)}(dx').\]
In addition, we will often use the notation $\<*\>_{(H,m)}$ in place of $G_{(H,m)}$.

Next, we define some subsets that will be useful throughout the work. First, for $0\le q_1<q_2\le \infty$ we define the annulus
\[S_N(q_1,q_2)=\left\{u\in \R^N: q_1\le \|u\|^2_N\le q_2\right\},\]
and for $q>\epsilon>0$ we define the thickened sphere
\[S_N^{\epsilon}(q)=S_N(q,q+\epsilon).\]

Next, for $q>0$, let $S_N(q)\subseteq \R^N$ denote the sphere of radius $\sqrt{Nq}$, and for $\b{q}\in (0,\infty)^{\Om}$ let $S_N(\b{q})=\prod_{x\in \Om}S_N(\b{q}(x))\subseteq (\R^{N})^{\Om}$.
We define $S_N^{\epsilon}(\b{q})$ and $S_N(\b{q}_0,\b{q}_1)$ similarly.

\pagebreak 
	
	\section{From the Spherical to the Euclidean Model\label{section:non-transitive main theorem}}
	
	In this section we will prove Theorem \ref{theorem:intro:bad main euclidean}, our result for the Euclidean model, assuming our result for the spherical model, Theorem \ref{theorem:intro:bad parisi spherical}. The argument follows a few basic steps, beginning with two elementary reductions.
	
	We first reduce to the case of $\beta=1$ by absorbing the factor into the other model parameters. If we temporarily write the Parisi-function in terms of the auxiliary parameters $(h,\mu,t,B)$ as $\cal{P}_{\beta,q}(*;h,\mu,t,B)$, and similarly denote the partition function as $Z_{N,\beta}(h,\mu,t,B)$. Then it is easily checked that
	\[Z_{N,\beta}(h,\mu,t,B)=Z_{N,1}(\beta h,\beta\mu,\beta t,\beta^2 B),\; \cal{P}_{\beta,\b{q}}(\zeta,\b{\Phi};h,\mu,t,B)=\cal{P}_{1,\b{q}}(\zeta,\b{\Phi};\beta h,\beta\mu,\beta t, \beta^2 B).\]
	In particular, to establish Theorem \ref{theorem:intro:bad main euclidean}, it suffices to treat the case of $\beta=1$. A similar argument shows that it suffices to treat the case of $\beta=1$ to show Theorem \ref{theorem:intro:bad parisi spherical}. Thus we will assume that $\beta=1$ henceforth unless otherwise stated. We will also use the notations $Z_N:=Z_{N,1}$, $f_N:=f_{N,1}$, etc. for simplicity.
	
	The next step is to reduce to the case of $h=0$ and is also a fairly simple manipulation, involving taking a fixed translation of target Euclidean space in the model. This does not affect the distribution of the random component of the Hamiltonian, as it is isotropic. However, it changes the linear term through the deterministic quadratic piece, which may be used to absorb the external field term. In particular, we have the following lemma.
	\begin{lem}
		\label{lem:h doesn't matter specific}
		Let us use the notation $\cal{H}_{N,h}$ to indicate the dependence of the function $\cal{H}_N$ given by (\ref{eqn:def:main-model-D}) on the choice $h$. Then we have the following distributional equality of functions on $\R^N$
		\[\cal{H}_{N,h}(\b{u})\disteq \cal{H}_{N,0}\left(\b{u}+\sqrt{N}\frac{h}{\mu}\b{e}_{1}\right)-N|\Om|\frac{h^2}{2\mu},\label{eqn:h doesn't matter function shift}\]
		where $\b{e}_{1}\in (\R^{N})^{\Om}$ is the vector such that $[\b{e}_{1}(y)]_j=\delta_{j1}$ for $y\in \Om$ and $1\le j\le N$.\\
		In particular, if we use the notation $Z_N(h)$ to indicate the dependence of the partition function on $h$, we have the following distributional equality
		\[|\Om|^{-1}N^{-1}\log Z_N(h)\disteq |\Om|^{-1}N^{-1}\log Z_N(0)+\frac{h^2}{2\mu}. \label{eqn:h doesn't matter partition function}\]
	\end{lem}
	\begin{proof}
		We observe that the constant vector $\b{1}\in \R^{\Om}$ is an eigenvector of $\mu I-t \Delta$ with eigenvalue $\mu$, so that
		\[\sum_{y\in \Om}(\mu I-t\Delta)_{xy}=\sum_{y\in \Om}(\mu I-t\Delta)_{xy}\b{1}(y)=\mu \b{1}(x)=\mu.\]
		In particular, it is easily verified that
		\[
		\begin{split}
			\frac{1}{2}\sum_{x,y\in \Om}(\mu I-t\Delta)_{xy}\left(\b{u}(x)+\sqrt{N}\frac{h}{\mu}\b{e}_{1},\b{u}(y)+\sqrt{N}\frac{h}{\mu}\b{e}_{1}\right)=\\
			\frac{1}{2}\sum_{x,y\in \Om}(\mu I-t\Delta )_{xy}(\b{u}(x),\b{u}(y))+\frac{1}{2}N|\Om|\frac{h^2}{\mu}+\sqrt{N}h\sum_{x\in \Om}\b{u}_1(x).
		\end{split}
		\]
		We recognize the latter two terms as the deterministic part of $\cal{H}_{N,h}(\b{u})$ so if we let
		\[\widetilde{\cal{H}}_N(\b{u})=\frac{1}{2}\sum_{x,y\in \Om}(\mu I-t\Delta)_{xy}(\b{u}(x),\b{u}(y))+\sum_{x\in \Om}V_{N,x}\l(\b{u}(x)-\sqrt{N}\frac{h}{\mu}\b{e}_{1}\r),\]
		then we see that
		\[\widetilde{\cal{H}}_{N}\l(\b{u}+\sqrt{N}\frac{h}{\mu}\b{e}_{1}\r)=N|\Om|\frac{h^2}{2\mu}+\cal{H}_{N,h}(\b{u}).\]
		Now, as each $V_{N,x}$ is isotropic, we see that as a function on $\b{u}$, $\widetilde{\cal{H}}_N(\b{u})$ coincides in law with $\cal{H}_{N,0}(\b{u})$, which completes the proof.
	\end{proof}
	
	\begin{remark}
		It is crucial to note that this manipulation is very specific to the Euclidean model in the case of a quadratic potential. Indeed, if one works either in the case of the spherical model, or if one were to take a non-quadratic confining potential, such a simplification need not hold. In particular, in this case one expects the external field to meaningfully affect the statics of the model, as is well known in the one-site case.
	\end{remark}
	
	With this Lemma, we see it suffices to show Theorem \ref{theorem:intro:bad main euclidean} in the case of $h=0$ and $\beta=1$. Next we analyze the contribution to the partition function $Z_N$ coming from a product of spheres of different radii. Rescaling, we may use Theorem \ref{theorem:intro:bad parisi spherical} to compute this, and then we may relate the result to $\cal{P}_{\b{q}}$, where the $\b{q}$ denotes the vector of squared-radii of the spheres.
	
	To begin, we establish the latter claim with the  following result relating $\cal{P}_{\b{q}}(\zeta,\b{\Phi})$ to $\cal{B}(\zeta',\b{\Phi}')$ for a certain choice of parameters $(\b{\xi},\b{h},\zeta',\b{\Phi}')$.
	
	\begin{lem}
		\label{lem:outline:identification of A and P}
		Fix $\b{q}\in (0,\infty)^{\Om}$. Let us choose $D_{\b{q}}=-t\text{di}(\sq{\b{q}}) \Delta\text{di}(\sq{\b{q}})$ and $\b{h}=\b{0}$. For $q>0$, let $B_q(r)=B(2q(1-r))$. It is easily verified that $B_q$ defines a mixing function. Now define, for each $x\in \Om$, $\b{B}_{\b{q}}(x)=B_{\b{q}(x)}$. Next consider $\cal{H}_{N,\mathrm{Sph}}$ as in (\ref{eqn:def:intro-spherical-model}) for this choice of $(\b{B}_{\b{q}},D_{\b{q}},\b{h})$, and denote the associated functional as $\cal{B}_{\b{q}}$. 
		\\
		Then for $(\zeta^{\b{q}},\b{\Phi}^{\b{q}})$ in the domain of $\cal{B}_{{\b{q}}}$, we may define $(\zeta,\b{\Phi})$ by $\Phi_x(s)=\b{q}(x)\Phi_x^{\b{q}}(q_t^{-1}s)$ and $\zeta^{\b{q}}([0,s])=\zeta([0,q_t^{-1}s])$. This gives a bijection between the domain of $\cal{B}_{{\b{q}}}$ and the domain of $\cal{P}_{\b{q}}$, and moreover we have
		\[\cal{P}_{\b{q}}(\zeta,\b{\Phi})=\frac{h^2}{2\mu}+\frac{1}{2|\Om|}\sum_{x\in \Om}\left(-\mu \b{q}(x)+\log(\b{q}(x))\right)+\cal{B}_{\b{q}}(\zeta^{\b{q}},\b{\Phi}^{\b{q}}).\]
	\end{lem}
	\begin{proof}
		Let $(\b{K}^{\b{q}},\Lambda^{\b{q}})$ denote the functions corresponding to $D_{\b{q}}$, while reserving the notation $(\b{K},\Lambda)$ for the choice $D=- t \Delta$. We first note that for $x\in \Om$
		\[K^{\b{q}}_{x}(u)=\b{q}(x)K_x(\b{q}\b{u}).\]
		From this, we obtain that
		\[
		\begin{split}
			\sum_{x\in \Om}\log(\b{q}(x))+|\Om|\Lambda^{\b{q}}(\b{q}^{-1}\b{u})=&\sum_{x\in \Om}\left(\log(\b{q}(x))+K_x^{\b{q}}(\b{q}^{-1}\b{u})\b{q}(x)^{-1}\b{u}(x)\right)\\
			&-\log \det(\b{K}^{\b{q}}(\b{q}^{-1}\b{u})-\text{di}(\sq{\b{q}})t\Delta\text{di}(\sq{\b{q}}))\\
			=&\sum_{x\in \Om}K_x(\b{u})\b{u}(x)-\log \det(\b{K}(\b{u})-t\Delta)\\
			=&|\Om|\Lam(\b{u}).
		\end{split}
		\]
		Next, if we denote by $\b{\delta}^{\b{q}}$ the function corresponding to $\cal{B}_{\b{q}}$, while using $\b{\delta}$ for the corresponding function for $\cal{P}_{\b{q}}$, we see that $\delta_x^{\b{q}}(u)=\b{q}(x)^{-1}\delta_x(q_tu)$. From this, we may conclude that
		\[\int_{0}^{s}K_x^{\b{q}}(\b{\delta}^{\b{q}}(u))(\Phi_x^{\b{q}})'(u)du=\int_{0}^{q_ts}K_x(\b{\delta}(u))\Phi'_x(u)du,\]
		and more simply that
		\[\int_0^{1}\zeta^{\b{q}}([0,u])\b{q}(x)B_{x,\b{q}(x)}'(\Phi_x^{\b{q}}(q))(\Phi_x^{\b{q}})'(q)dq=-2\int_0^{q_t}\zeta([0,u])B'(2(\b{q}(x)-\Phi_x(u)))\Phi'_x(u)du.\]
		Altogether this gives the final result by matching terms.
	\end{proof}
	
	Then we define the restriction of the partition function to $S_N(\b{q})$ by 
	\[Z_N(\b{q})=\int_{S_N(\b{q})}\exp(-\cal{H}_N(\b{u}))\omega(d\b{u}),\]
	where $\omega(d\b{u})$ is the surface measure induced by the inclusion $S_N(\b{q})\subseteq (\R^{N})^{\Om}$. We are then able to obtain the following result immediately from Theorem \ref{theorem:intro:bad parisi spherical} and Lemma \ref{lem:outline:identification of A and P}
	
	\begin{corr}
		\label{corr:formula for spherical slice}
		Assume that $h=0$. For any choice of $\b{q}\in (0,\infty)^{\Om}$, we have that
		\[\lim_{N\to \infty}|\Om|^{-1}N^{-1}\E \log Z_N(\b{q})=\inf_{(\zeta,\b{\Phi})\in \mathscr{Y}(\b{q})}\cal{P}_{\b{q}}(\zeta,\b{\Phi}).\]
	\end{corr}
	\begin{proof}
		We will freely use the notation of Lemma \ref{lem:outline:identification of A and P} in this proof. Observe that by a change of variables,
		\[Z_N(\b{q})=\left(\prod_{x\in \Om}\b{q}(x)^{(N-1)/2}\right)\int_{S_N^{\Om}}\exp(-\cal{H}_{N}(\b{q}\b{u})) \omega(d\b{u}).\label{eqn:ignore-18267}\]
		Recalling the form of $\cal{H}_N$ from (\ref{eqn:def:main-model-D}), we may write
		\[
		\begin{split}
			\cal{H}_{N}(\b{q}\b{u})=& \frac{1}{2}\sum_{x,y\in \Om}(\mu I-t \Delta)_{x,y}\sqrt{\b{q}(x)\b{q}(y)}(\b{u}(x),\b{u}(y))+\sum_{x\in \Om}V_{N,x}(\sq{\b{q}(x)}\b{u}(x))\\
			&=\frac{N\mu}{2} \sum_{x\in \Om}\b{q}(x)-\frac{t}{2}\sum_{x,y\in \Om} \Delta_{x,y}\sqrt{\b{q}(x)\b{q}(y)}(\b{u}(x),\b{u}(y))+\sum_{x\in \Om}V_{N,x}(\sq{\b{q}(x)}\b{u}(x)).
		\end{split}
		\]
		We note that restricted to $u\in S_N$, the function $V_{N,x}(\sq{\b{q}(x)}u)$ is an isotropic function with correlation function $N B(2q(1-(u,u')_N))$. In particular, we see that shifted Hamiltonian, $\cal{H}_{N}(\b{q}\b{u})-\frac{N\mu}{2} \sum_{x\in \Om}\b{q}(x)$, restricted to $S_N^{\Om}$ coincides in law with the spherical model (\ref{eqn:def:intro-spherical-model}) corresponding to choice of parameters $(\b{B}_{\b{q}},D_{\b{q}},\b{h})$. Thus by Theorem \ref{theorem:intro:bad parisi spherical} and (\ref{eqn:ignore-18267}) we see that
		\[\lim_{N\to \infty}|\Om|^{-1}N^{-1}\E \log Z_N(\b{q})=\frac{1}{2|\Om|}\sum_{x\in \Om}\left(-\mu \b{q}(x)+\log(\b{q}(x))\right)+\inf_{(\zeta,\b{\Phi})\in \mathscr{Y}}\cal{B}_{\b{q}}(\zeta,\b{\Phi}).\]
		Using the bijection of Lemma \ref{lem:outline:identification of A and P} then completes the proof.
	\end{proof}
	
	With these prerequisites established, Theorem \ref{theorem:intro:bad main euclidean} heuristically follows by using the co-area formula to express $Z_N$ as an integral over the contributions $Z_N(\b{q})=\exp(\log Z_N(\b{q}))$ and then applying Laplace's method and our asymptotic formula for $N^{-1}\E \log Z_N(\b{q})$. However, directly using this method is problematic as we would need to uniformly control fluctuations of $\log(Z_N(\b{q}))$ around $\E \log(Z_N(\b{q}))$, which is difficult to do pointwise.
	
	So our next step will be to modify Corollary \ref{corr:formula for spherical slice} to a version which will apply to a product of thin annuli. For this, define for $A\subseteq (\R^N)^{\Om}$ the integral
	\[Z_N(A)=\int_{A}\exp(-\cal{H}_N(\b{u}))d\b{u},\]
	and then also $Z_N^{\epsilon}(\b{q})=Z_N(S_N^{\epsilon}(\b{q}))$. Our next step will be to uniformly bound difference between $Z_N^{\epsilon}(\b{q})$ and $Z_N(\b{q})$ when $\b{q}\in (0,\infty)^{\Om}$ is fixed and $\epsilon>0$ is small.
	
	\begin{lem}
		\label{lem:epsilon of room around slices}
		Let us fix a small $0<m<1$. Than we have that
		\[\lim_{\epsilon\to 0}\limsup_{N\to \infty}\l(\sup_{\b{q}\in [m,m^{-1}]^{\Om}}\l(N^{-1}|\E \log Z_N^{\epsilon}(\b{q})-\E \log Z_N(\b{q})|\r)\r)=0.\]
	\end{lem}
	\begin{proof}
		This will follow by applying Proposition \ref{prop:Guerra-disintegration-new-appendix}, as for fixed $\b{q}$, the difference of partitions functions is a special case of the difference considered in that lemma. Specifically, in this case when take $(S,\mu)$ to be the unit mass on a point (which we then remove from the notation), and have
		\[f_x(r^2,(r')^2,\rho)=B\left(r^2+(r')^2-2\rho\right), \;\;\; g(\b{u})=-\frac{t}{2}\sum_{x,y\in \Om}\Delta_{xy}(\b{u}(x),\b{u}(y)).\]
		We also assume that $\epsilon<m/2$. A direct computation shows that
		\[\|\D f_x(r^2,(r')^2,\rho)\|=2\sqrt{\left(r^2+(r')^2+1\right)}(-B')\left(r^2+(r')^2-2\rho\right),\]
		so that
		\[\|\|\D f_x\|\|_{L^\infty}\le 2\sqrt{2(\b{q}(x)+\epsilon)+1}|B'(0)|\le 2\sqrt{2m^{-1}+2\epsilon+1}|B'(0)|.\]
		Furthermore, by the Cauchy-Schwarz inequality we have that
		\[
		\|\D g(\b{u})\|^2=\sum_{x\in \Om}\bigg\|\sum_{y\in \Om}t\Delta_{xy}\b{u}(y)\bigg\|^2\le t^2\|\Delta\|_F^2|\Om|N\left(m^{-1}+\epsilon\right).\]
		Inputting these computations into Proposition \ref{prop:Guerra-disintegration-new-appendix} and taking limits then gives the desired result.
	\end{proof}
	
	Finally, we will need a result that will allow us to neglect the subset of $\b{q}$ where the minimum or maximum value of $\b{q}(x)$ is too small or large respectively. This shows that the contributions of all values of $\b{q}$ outside of a compact subset of $(0,\infty)^{\Om}$ are negligible.
	
	\begin{lem}
		\label{lem:you can ignore big subsets}
		Let us assume that $\b{h}=0$ and define the subsets
		\[
		\begin{split}
			A_{min}^N(q)&=\left\{\b{u}\in (\R^N)^{\Om}:\min_{x\in \Om}\|\b{u}(x)\|_N\le q\right\},\\
			A_{max}^N(q)&=\left\{\b{u}\in (\R^N)^{\Om}:\max_{x\in \Om}\|\b{u}(x)\|_N\ge q\right\}.
		\end{split}	
		\]
		Then we have that
		\[
		\limsup_{q\to 0}\limsup_{N\to\infty}N^{-1}\E\log Z_N(A_{min}^N(q))\le \lim_{q\to 0}\limsup_{N\to\infty}N^{-1}\log \E Z_N(A_{min}^N(q))=-\infty,
		\]
		\[
		\limsup_{q\to \infty}\limsup_{N\to\infty}N^{-1}\E \log Z_N(A_{max}^N(q))\le \lim_{q\to \infty}\limsup_{N\to\infty}N^{-1}\log \E Z_N(A_{max}^N(q))-\infty.
		\]
	\end{lem}
	\begin{proof}
		For any subset $A\subseteq (\R^N)^{\Om}$ of non-zero measure, we have by Jensen's inequality that
		\[\E \log Z_N(A)\le \log \E Z_N(A)=N|\Om|B(0)+\log Z_N^0(A),\]
		where $Z_N^0(A)$ the deterministic quantity
		\[Z_N^0(A)=\int_{A} \exp\left(-\frac{1}{2}\sum_{x,y\in \Om}(\mu I-t\Delta)_{xy}(\b{u}(x),\b{u}(y))\right)d\b{u}.\]
		In particular, it suffices to show the corresponding statements for $\log Z_N^0(A_{min}^N(q))$ and $\log Z_N^0(A_{max}^N(q))$.
		For this, let us denote for a subset $I\subseteq [0,\infty)$ and choice $x\in \Om$, the subset
		\[A_{x,N}(I)=\{\b{u}\in (\R^{N})^{\Om}:\|\b{u}(x)\|_N\in I\}.\]
		Combining this with the9 observations
		\[A_{min}^N(q)=\bigcup_{x\in \Om}A_{x,N}([0,q]),\;\;\; \;A_{max}^N(q)=\bigcup_{x\in \Om}A_{x,N}([q,\infty]),\]
		we see it suffices to fix $x\in \Om$ and show the corresponding claims for $\log Z_N^0(([0,q]))$ and $\log Z_N^0(A_{x,N}([q,\infty]))$. As $\Delta$ is negative semi-definite, we see that
		\[Z_N^0(A_{x,N}(I))\le \int_{A_{x,N}(I)} \exp\left(-\frac{\mu}{2}\sum_{x\in \Om}\|\b{u}(x)\|^2\right)d\b{u}.\]
		The right-hand side is easily simplified to
		\[\left(\frac{2\pi}{\mu}\right)^{N(|\Om|-1)/2}\vol(S^{N-1}(\sqrt{N}))\int_{I}r^{N-1}\exp\left(-N\frac{\mu r^2}{2}\right)dr.\]
		By Stirling's formula we have that
		\[\lim_{N\to \infty}N^{-1}\log(\vol(S^{N-1}(\sqrt{N})))=\log(\sq{2\pi e}).\]
		Finally, by Laplace's method we have that for open $I\subseteq (0,\infty)$ we have that
		\[\limsup_{N\to\infty}N^{-1}\log\left(\int_I r^{N-1}\exp\left(-N\frac{\mu r^2}{2}\right)dr\right)=\sup_{r\in I}\left(\log(r)-\frac{\mu r^2}{2}\right).\]
		In particular, the desired bounds follow from the fact that $\log(r)-\frac{\mu r^2}{2}\to -\infty$ when either $r\to 0$ or $r\to \infty$.
	\end{proof}
	\begin{remark}
		\label{remark:decline in restriction over q}
		This proof also shows a related statement, needed elsewhere. Namely, there is some small $c>0$, such that for any large enough $r>0$
		\[N^{-1}\log \E Z_N(A^N_{max}(r))\le -cr^2.\]
	\end{remark}
	
	With these prerequisites completed, we are ready to begin the proof of Theorem \ref{theorem:intro:bad main euclidean}. This will consist of using Lemma \ref{lem:you can ignore big subsets} to reduce our consideration to a compact subset of possible $\b{q}$, then choosing an $\epsilon>0$, we may cover this with a finite number of the $Z_N^\epsilon(\b{q})$. These may be understood by using Lemma \ref{lem:epsilon of room around slices} to relate them to $Z_N(\b{q})$, where we may use Corollary \ref{corr:free energy of formula with restricted radii}
	
	\begin{proof}[Proof of Theorem \ref{theorem:intro:bad main euclidean}]
		By Lemma \ref{lem:h doesn't matter specific}, it suffices to assume that $h=0$. Now fix $q\in (0,\infty)^{\Om}$. Then combining Corollary \ref{corr:free energy of formula with restricted radii} and Lemma \ref{lem:epsilon of room around slices} with Jensen's inequality we see that
		\[\inf_{(\zeta,\b{\Phi})\in \mathscr{Y}(\b{q})}\cal{P}_{\b{q}}(\zeta,\b{\Phi})=\lim_{\epsilon\to 0}\liminf_{N\to \infty}|\Om|^{-1}N^{-1}\E \log Z_N^{\epsilon}(\b{q})\le \liminf_{N\to \infty}|\Om|^{-1}N^{-1}\E \log Z_N.\]
		In particular, this establishes the lower bound in expectation
		\[\sup_{\b{q}\in (0,\infty)^{\Om}}\l(\inf_{(\zeta,\b{\Phi})\in \mathscr{Y}(\b{q})}\cal{P}_{\b{q}}(\zeta,\b{\Phi})\r)\le \liminf_{N\to \infty}|\Om|^{-1}N^{-1}\E \log Z_N.\]
		Before proceeding to the upper bound we prove a basic result about the concentration of $|\Om|^{-1}N^{-1}\log Z_N$. We note that by Jensen's inequality
		\[
		\log \E Z_N=
		\log \l(\int_{(\R^{N})}\exp\l(-\frac{1}{2}\sum_{x,y\in \Om}(\mu I-t\Delta)_{xy}(\b{u}(x),\b{u}(y))\r)d\b{u}\r)+N|\Om|B(0).
		\]
		As $\Delta$ is negative semi-definite, this easily shows that $\log \E Z_N$ is uniformly bounded above in $N$. Thus by Corollary \ref{cor:Borell-TIS-free-energy-noncompact} we have for any $\alpha>0$ that
		\[\P\left(\left||\Om|^{-1}N^{-1}\log Z_N-|\Om|^{-1}N^{-1}\E \log Z_N\right|>\alpha\right)\le 2\exp\left(-\frac{N|\Om|\alpha^2}{4B(0)}\right).\label{eqn:ignore-1736}\]
		Now we proceed to show the upper bound. Define  \[F^N_{m}:=\{\b{u}\in (\R^N)^{\Om}:\|\b{u}\|_N\in [m,m^{-1}]^{\Om}\}=\left(A^N_{min}(m)\cup A^N_{max}(m^{-1})\right)^{c}.\] By Lemma \ref{lem:you can ignore big subsets}, for sufficiently small $m>0$ we have that
		\[\limsup_{N\to \infty}|\Om|^{-1}N^{-1}\E \log Z_N=\limsup_{N\to \infty}|\Om|^{-1}N^{-1}\E \log Z_N(F^N_{m}).\label{eqn:ignore-1777}\]
		Now fix $\epsilon>0$. By Lemma \ref{lem:epsilon of room around slices} we may choose small enough $\delta>0$ such that
		\[\limsup_{N\to \infty}\sup_{\b{q}\in [\sqrt{m}/2,2\sqrt{m^{-1}}]^{\Om}}\bigg(|\Om|^{-1}N^{-1}|\E \log Z_N^{\delta}(\b{q})-\E \log Z_N(\b{q})|\bigg)\le \epsilon.\]
		
		Now as $[m,m^{-1}]^{\Om}$ is compact we may find $\b{q}_1,\cdots,\b{q}_{n-1},\b{q}_n\in [\sqrt{m}/2,2\sqrt{m^{-1}}]^{\Om}$ such that \[[m,m^{-1}]^{\Om}\subseteq \bigcup_{1\le i\le n}\prod_{x\in \Om}[\sqrt{\b{q}_i(x)},\sqrt{\b{q}_i(x)+\delta}],\]
		so that in particular, $F^N_{\delta}\subseteq\bigcup_{1\le i\le n}S_N^{\delta}(\b{q}_i)$. Let us denote the events
		\[\cal{E}_N^i=\{N^{-1}\log Z_N^{\delta}(\b{q}_i)\ge  N^{-1}\E \log Z_N^{\delta}(\b{q}_i)+\epsilon\},\;\;\; \cal{E}_N=\bigcup_{1\le i\le n}\left(\cal{E}_N^i\right)^c.\]
		We see that
		\[
		\begin{split}
			\limsup_{N\to \infty}\E[I_{\cal{E}_N}|\Om|^{-1}N^{-1}\log Z_N(F^N_{m})]\le \epsilon+\max_{1\le i\le n} \limsup_{N\to \infty}|\Om|^{-1}N^{-1}\E \log Z_N^{\delta}(\b{q}_i)\le \\
			2\epsilon+\max_{1\le i\le n} |\Om|^{-1}N^{-1}\limsup_{N\to \infty}\E \log Z_N(\b{q}_i)\le 2\epsilon+\sup_{\b{q}\in (0,\infty)^{\Om}}\left(\inf_{\zeta,\b{\Phi}\in \mathscr{Y}(\b{q})}\cal{P}_{\b{q}}(\zeta,\b{\Phi})\right).\label{eqn:ignore-1151}
		\end{split}
		\]
		We note that by Proposition \ref{prop:Borell-TIS-free-energy}, we have for each $i$, that 
		\[\P(\cal{E}_N^i)\le 2\exp\left(-\frac{N\epsilon^2}{4|\Om|B(0)}\right).\]
		Thus by the Borel-Cantelli Lemma we see that a.s. 
		\[\limsup_{N\to \infty}|\Om|^{-1}N^{-1}\log Z_N(F^N_{m})\le 2\epsilon+ \sup_{\b{q}\in (0,\infty)^{\Om}}\left(\inf_{\zeta,\b{\Phi}\in \mathscr{Y}(\b{q})}\cal{P}_{\b{q}}(\zeta,\b{\Phi})\right).\]
		Using (\ref{eqn:ignore-1736}), (\ref{eqn:ignore-1777}) and the observation that $|\Om|^{-1}N^{-1}\E \log Z_N\le |\Om|^{-1}N^{-1}\log \E Z_N$ is uniformly bounded in $N$, we see that
		\[\limsup_{N\to \infty}|\Om|^{-1}N^{-1}\E \log Z_N(F^N_{m})\le 2\epsilon+\sup_{\b{q}\in (0,\infty)^{\Om}}\left(\inf_{\zeta,\b{\Phi}\in \mathscr{Y}(\b{q})}\cal{P}_{\b{q}}(\zeta,\b{\Phi})\right).\]
		Thus we have established the desired convergence in expectation. The statement about a.s. convergence follows immediately from (\ref{eqn:ignore-1736}) and another application of the Borel-Cantelli Lemma.
	\end{proof}
	We complete this section by also noting two corollaries of this proof that will be useful in our companion works \cite{Paper2,Paper3}. The first is that this proof actually establishes a more general formula for the free energy of the partition function over sets of restricted radii in the case where $h=0$. In particular, our proof immediately gives the following result.
	\begin{corr}
		\label{corr:free energy of formula with restricted radii}
		Assume that $h=0$. Then for any open subset $U\subseteq [0,\infty)^{\Om}$ such that $U\cap (0,\infty)^{\Om}\neq \varnothing$, we have that
		\[
		\lim_{N\to \infty}|\Om|^{-1}N^{-1}\E \log Z_N(\{\b{u}\in (\R^{N})^{\Om}:\|\b{u}\|^2_N\in U\})=\sup_{\b{q}\in U}\left(\inf_{(\zeta,\b{\Phi})\in \mathscr{Y}(\b{q})}\cal{P}_{\b{q}}(\zeta,\b{\Phi})\right).
		\]
	\end{corr}
	This corollary combined with Lemma \ref{lem:you can ignore big subsets} also gives the following result.
	\begin{corr}
		\label{corr:outline:ignore-small-and-large-q}
		There is sufficiently small $m>0$ such that
		\[\sup_{\b{q}\in (0,\infty)^{\Om}}\left(\inf_{(\zeta,\b{\Phi})\in \mathscr{Y}(\b{q})}\cal{P}_{\b{q}}(\zeta,\b{\Phi})\right)=\sup_{\b{q}\in [m,m^{-1}]^{\Om}}\left(\inf_{(\zeta,\b{\Phi})\in \mathscr{Y}(\b{q})}\cal{P}_{\b{q}}(\zeta,\b{\Phi})\right).\]
	\end{corr}
	
	\pagebreak
	
	\section{A First Step to a Crisanti-Sommers Formula \label{section:CS=Parisi}}
	
	In this section, we introduce a new formula for the asymptotic free energy of the spherical variant of the elastic manifold, given by Theorem \ref{theorem:spherical parisi continuum}. The proof of this new formula will be the goal of the subsequent sections of this work. 
	
	The focus of this section is to prove that this alternate Parisi formula implies our Theorem \ref{theorem:intro:bad parisi spherical}, which generalizes to the elastic manifold context the well-know Crisanti-Sommers formula \cite{crisantisommersOG} about the one-site model (i.e spherical spin glasses). As we saw in Section \ref{section:non-transitive main theorem}, this implies our main result about the asymptotic free energy for the Euclidean model of the elastic manifold (i.e. Theorem \ref{theorem:intro:bad main euclidean}).
	
	In order to introduce this variant, we now need to introduce an alternative version of the Parisi functional, which will be denoted as $\cal{A}$. To begin, for a choice of $(\zeta,\b{\Phi})\in \mathscr{Y}$, we define the function $\b{d}:[0,1]\to [0,\infty)^{\Om}$ by
	\[\b{d}(u)=\int_u^1 \zeta([0,s]) \xi''_x(\Phi_x(s))\Phi'_x(s)ds.\]
	We then choose $\b{b}\in \R^{\Om}$ such that $D+\b{b}-\b{d}(0)>0$. With these choices, we define the functional
	\[
	\begin{split}
		\cal{A}(\zeta,\b{\Phi},\b{b})=\frac{1}{2|\Om|}\bigg[&|\Om|\log(2\pi)+\sum_{x\in \Om}\int_0^1 [(D+\b{b}-\b{d}(s))^{-1}]_{xx}\xi''_x(\Phi_x(s))\Phi_x'(s)ds\\ 
		&+\sum_{x\in \Om}\left(\b{b}(x)-\int_0^1 \zeta([0,s])\Phi_x(s)\xi_x''(\Phi_x(s))\Phi_x'(s)ds\right)\\
		&+\sum_{x,y\in \Om}[(D+\b{b}-\b{d}(0))^{-1}]_{xy}\l(\b{h}(x)\b{h}(y)+\delta_{xy}\xi_x'(0)\r)\bigg],
	\end{split}
	\label{eqn:def:A continuum}
	\]
	We then minimize over $\b{b}$ to define the functional
	\[\cal{A}(\zeta,\b{\Phi})=\inf_{\b{b}\in \R^{\Om }:\;D+\b{b}-\b{d}(0)>0}\cal{A}(\zeta,\b{\Phi},\b{b}).\label{eqn:def:A true def}\]
	
	While this functional is defined on all of $(\zeta,\b{\Phi})\in \mathscr{Y}$, for technical reasons, we restrict this functional a subset of finite measures,
	\[\mathscr{Y}_{fin}=\{(\zeta,\b{\Phi})\in \mathscr{Y}:\supp(\zeta) \text{ is a finite set}\}.\]
	We now give our new computation for the asymptotic free energy in terms of $\cal{A}$.
	
	\begin{theorem}[A New Parisi Formula]
		\label{theorem:spherical parisi continuum}
		\[
		\lim_{N\to \infty}|\Om|^{-1}N^{-1}\E \log Z_{N,\mathrm{Sph}}=\inf_{(\zeta,\b{\Phi})\in \mathscr{Y}_{fin}}\cal{A}(\zeta,\b{\Phi}).
		\]
	\end{theorem}
	
	We mention that using the functional $\cal{A}$ is important in our Parisi formula, as $\cal{A}$ is what appears more naturally in our analysis of the free energy. However $\cal{B}$ is both more simple to analyze and also the functional which appears by applying the replica-trick. As discussed above, this matches a similar issue in the original proof of the even one-site case by Talagrand \cite{talagrandOG}. One may also see the work \cite{erikCS} which establishes a similar result in the case of multi-species spin glasses. 
	
	\begin{remark}
		We believe that the restriction to finite measures in Theorem \ref{theorem:spherical parisi continuum} should not affect the infimum of $\cal{A}$, or more specifically, that \[\inf_{(\zeta,\b{\Phi})\in \mathscr{Y}_{fin}}\cal{A}(\zeta,\b{\Phi})=\inf_{(\zeta,\b{\Phi})\in \mathscr{Y}}\cal{A}(\zeta,\b{\Phi}).\]
		However, as this result is not needed here, we have not pursued this issue for brevity. We do however, show the analogous result for the functional $\cal{B}$ in Appendix \ref{section:appendix:B appendix}.
	\end{remark}
	
	The proof of Theorem \ref{theorem:spherical parisi continuum} is the focus of all subsequent sections in this paper. The focus of the remainder of this section is to prove Theorem \ref{theorem:intro:bad parisi spherical} using Theorem \ref{theorem:spherical parisi continuum}. The key to this is the following result, which shows that these functionals have the same infimum over $\mathscr{Y}_{fin}$.
	
	\begin{prop}
		\label{prop: easy Parisi=CS-minor}
		We have that 
		\[\inf_{(\zeta,\b{\Phi})\in \mathscr{Y}_{fin}}\cal{B}(\zeta,\b{\Phi})=\inf_{(\zeta,\b{\Phi})\in \mathscr{Y}_{fin} }\cal{A}(\zeta,\b{\Phi}).\]
	\end{prop}
	
	Given this result, we may now easily obtain Theorem \ref{theorem:intro:bad parisi spherical} from Theorem \ref{theorem:spherical parisi continuum}.
	\begin{proof}[Proof of Theorem \ref{theorem:intro:bad parisi spherical}]
		By Theorem \ref{theorem:spherical parisi continuum} and Propositions \ref{prop: easy Parisi=CS-minor} and \ref{prop:A and B continuity} we have that
		\[\lim_{N\to \infty}|\Om|^{-1}N^{-1}\E \log Z_{N,\mathrm{Sph}}=\inf_{(\zeta,\b{\Phi})\in \mathscr{Y}_{fin}}\cal{A}(\zeta,\b{\Phi})=\inf_{(\zeta,\b{\Phi})\in \mathscr{Y}_{fin}}\cal{B}(\zeta,\b{\Phi})=\inf_{(\zeta,\b{\Phi})\in \mathscr{Y}}\cal{B}(\zeta,\b{\Phi}).\]
		The a.s. statement then follows from this, Proposition \ref{prop:Borell-TIS-free-energy} and the Borel-Cantelli lemma.
	\end{proof}
	To prepare for the proof of Proposition \ref{prop: easy Parisi=CS-minor}, we must study $\mathscr{Y}_{fin}$. To begin, let us choose $r\ge 1$, as well as a sequence of real numbers 
	\[0=m_0 < m_1<\cdots <m_{r-1}< m_r=1,\label{eqn:def:x-sequence-Tal}\]
	and a sequence $(\b{s}^k(x))_{k=1}^{r}$, for each $x\in \Om$, such that
	\[0=\b{s}^0(x)\le \b{s}^1(x)\le \cdots \le \b{s}^r(x)< \b{s}^{r+1}(x)=1\label{eqn:def:Q-sequence-Tal}.\]
	
	We will use the notations $\vec{m}=(m_i)_{0\le i\le r-1}$ and $\vec{\b{s}}=(\b{s}^i)_{1\le i\le r}$, and similarly for other such sequences. In the next remark, we summarize how we may associate a pair $(\zeta,\b{\Phi})\in \mathscr{Y}_{fin}$ to the pair $(\vec{\b{s}},\vec{m})$.
	
	\begin{remark}
		\label{remark:associated to measure: talagrand}
		Let us take sequences $\vec{m}$ and $\vec{\b{s}}$ as in (\ref{eqn:def:x-sequence-Tal}) and (\ref{eqn:def:Q-sequence-Tal}) of length parameter $r$. By defining 
		\[s^k=\frac{1}{|\Om|}\sum_{x\in \Om}\b{s}^k(x)\label{def:q-seq-norm-tal}\] 
		we obtain a sequence $\vec{s}:=(s^k)_{k=1}^{r}$ of the form
		\[0=s^0\le s^1\le \dots \le s^r< s^{r+1}=1.\]
		We may associate with $(\vec{m},\vec{s})$ a discrete probability measure on $[0,1]$, $\zeta$, by
		\[\zeta=\sum_{0\le k\le r}(m_k-m_{k-1})\delta_{s^k},\]
		where $\delta_s$ denotes the Dirac $\delta$-function supported at $s$. We also may associate a function $\b{\Phi}:[0,1]\to [0,1]^{\Om}$ by letting $\b{\Phi}(s^k)=\b{s}^k$, then extending the function to all of $[0,1]$ by linear interpolation. It is easily checked that this satisfies (\ref{eq:condition for Psi}).  Moreover, we see that $(\zeta,\b{\Phi})\in \mathscr{Y}_{fin}$, and more specifically that
		\[(\Phi_x)_*(\zeta)=\sum_{0\le k\le r}(m_k-m_{k-1})\delta_{\b{s}^k}.\]
	\end{remark}
	
	We note two things about the process given in Remark \ref{remark:associated to measure: talagrand}. The first is that when one has some $\b{s}^k=\b{s}^{k+1}$ for some $k$, the measure no longer depends on $m_k$, so that the process is not injective. However, restricted to the set of choices where all $\b{s}^k$ are distinct, the process is indeed injective. In particular, this construction overparameterized its image, though only mildly.
	
	The second and more important fact is that this process does not surject onto $\mathscr{Y}_{fin}$. Indeed, to obtain a surjection, we would have to include all possible $\b{\Phi}$ which take the values $\b{\Phi}(s^k)=\b{s}^k$, instead of simply taking them to be a linear interpolation. However, in Appendix \ref{section:appendix:B appendix}, it is checked that if two choices of $(\zeta,\b{\Phi}),(\zeta,\b{\Phi}')\in \mathscr{Y}$ are such that $\b{\Phi}(s)=\b{\Phi}'(s)$ for $s\in \supp(\zeta)$, then one has that $\cal{A}(\zeta,\b{\Phi})=\cal{A}(\zeta,\b{\Phi}')$ (\textbf{Proposition \ref{prop:general: for B, parisi doesn't depend away from support}}) and $\cal{B}(\zeta,\b{\Phi})=\cal{B}(\zeta,\b{\Phi}')$ (Proposition \ref{prop:general:parisi doesn't depend away from support}). So, as far as these functionals concerned, we can restrict to this class of finite measures.
	
	In particular, let us define $\cal{B}(\vec{\b{s}},\vec{m}):=\cal{B}(\zeta,\b{\Phi})$, where $(\zeta,\b{\Phi})\in \mathscr{Y}_{fin}$ is the pair associated to $(\vec{\b{s}},\vec{m})$ by Remark \ref{remark:associated to measure: talagrand}. We define $\cal{A}(\vec{\b{s}},\vec{m})$ identically. Then the above discussion demonstrates the following:
	\[
	\inf_{\vec{\b{s}},\vec{m}}\cal{B}(\vec{\b{s}},\vec{m})=\inf_{(\zeta,\b{\Phi})\in \mathscr{Y}_{fin}}\cal{B}(\zeta,\b{\Phi}) \text{ and } \inf_{\vec{\b{s}},\vec{m}}\cal{A}(\vec{\b{s}},\vec{m})=\inf_{(\zeta,\b{\Phi})\in \mathscr{Y}_{fin}}\cal{A}(\zeta,\b{\Phi}).
	\label{eqn:lem:finite to parametrized}
	\]
	
	Our next step will be to evaluate $\cal{B}(\vec{\b{s}},\vec{m})$ and $\cal{A}(\vec{\b{s}},\vec{m})$. To evaluate $\cal{B}$ we define a sequence for $x\in \Om$ and $1\le l\le r$ by
	\[\delta^{l}_x=\sum_{l\le k\le l}m_k(\b{s}^{k+1}(x)-\b{s}^k(x)),\;\; \delta^{r+1}_x=0.\]
	Routine computation then gives that
	\[
	\begin{split}
		\cal{B}(\vec{\b{s}},\vec{m})=\frac{1}{2|\Om|}\bigg[&|\Om|\log(2\pi)+|\Om|\Lambda^D(\b{\delta}^r)-\sum_{1\le k<r}\frac{|\Om|}{m_k}\bigg(\Lam^D(\b{\delta}^{k+1})-\Lam^D(\b{\delta}^{k})\bigg)\\
		&+\sum_{x\in \Om}\left(-K_x^D(\b{\delta}^1)\b{s}^1(x)+\sum_{1\le k<r}m_k(\xi_x(\b{s}^{k+1}(x))-\xi_x(\b{s}^k(x)))\right)\\
		&+\sum_{x,y\in \Om}[(D+\b{K}^D(\b{\delta}^1))^{-1}]_{xy}\b{h}(x)\b{h}(y) \bigg].\label{eqn:def:leg-hardcore}
	\end{split}
	\]
	
	Next we give an expression for $\cal{A}$. Given such choices we may inductively define for $x\in \Om$ and $1\le l\le r$
	\[\b{d}^{l}(x)=\sum_{l\le k\le r}m_{k}(\xi'_x(\b{s}^{k+1}(x))-\xi'_x(\b{s}^{k}(x))),\;\; \b{d}^{r+1}(x)=0.\]
	Finally let us take a choice of some $\b{b}\in \R^{\Om}$ such that $D+\b{b}-\b{d}^1>0$. Associated with these choices, we may compute that
	\[
	\begin{split}
		\cal{A}(\vec{\b{s}},\vec{m},\b{b})=\frac{1}{2|\Om|}\bigg[&|\Om|\log(2\pi)+\sum_{1\le k\le r}\frac{1}{m_k}\log \frac{\det(D+\b{b}-\b{d}^{k+1})}{\det(D+\b{b}-\b{d}^k)}-\log \det(D+\b{b})\\ 
		&+\sum_{x\in \Om}\left(\b{b}(x)-\sum_{1\le  k<r}m_k(\theta_x(\b{s}^{k+1}(x))-\theta_x(\b{s}^k(x)))\right)\\
		&+\sum_{x,y\in \Om}[(D+\b{b}-\b{d}^1)^{-1}]_{xy}\l(\b{h}(x)\b{h}(y)+\delta_{xy}\xi_x'(\b{s}^1(x))\r)\bigg],
	\end{split}
	\label{eqn:def:A in terms of s}
	\]
	where here $\theta_x(r)=\xi_x'(r)-\xi_x(r)+\xi_x(0)$. Then we finally have that
	\[\cal{A}(\vec{\b{s}},\vec{m})=\inf_{\b{b}\in \R^{\Om}: \; D+\b{b}-\b{d}^1>0}\cal{A}(\vec{\b{s}},\vec{m},\b{b}).\]
	
	The main technical result we use to establish Proposition \ref{prop: easy Parisi=CS-minor} is the following.
	
	\begin{prop}
		\label{prop:Parisi=CS-minor}
		Let us assume that for each $x\in \Om$, we have that $\xi''_x(0),\xi'_x(0)\neq 0$. Then for each $r\ge 1$ and fixed choice of $\vec{m}$, we have that
		\[\inf_{\vec{\b{s}}}\cal{B}(\vec{\b{s}},\vec{m})=\inf_{\vec{\b{s}}}\cal{A}(\vec{\b{s}},\vec{m}),\]
		where here $\vec{\b{s}}$ is taken over all choices of (\ref{eqn:def:Q-sequence-Tal}).
	\end{prop}
	
	An important difference between Propositions \ref{prop: easy Parisi=CS-minor} and \ref{prop:Parisi=CS-minor}, beyond the parametrization, is that we do not need to minimize over $\vec{m}$, let alone the choice of $r\ge 1$. In particular, Proposition \ref{prop:Parisi=CS-minor} gives a much stronger result on the minimization than Proposition \ref{prop: easy Parisi=CS-minor}, under only a slight restriction on $\b{\xi}$. For example, keeping $r\ge 1$ fixed, but taking the infimum over $\vec{m}$, we see that Proposition \ref{prop:Parisi=CS-minor} yields the following corollary. 
	
	\begin{corr}
		\label{corr:Parisi=CS-minor kRSB}
		Let us assume that for each $x\in \Om$, we have that $\xi''_x(0),\xi'_x(0)\neq 0$. For $r\ge 1$, let us define
		\[\mathscr{Y}_{r}:=\{(\zeta,\b{\Phi})\in \mathscr{Y}_{fin}:\supp(\b{\Phi}_*(\zeta)) \text{ consists of at-most } r\text{ points}\}.\]
		Then we have that
		\[\inf_{(\zeta,\b{\Phi})\in \mathscr{Y}_r}\cal{B}(\zeta,\b{\Phi})=\inf_{(\zeta,\b{\Phi})\in \mathscr{Y}_r}\cal{A}(\zeta,\b{\Phi}).\]
	\end{corr}
	
	In particular, using some final continuity results from the appendices, we may obtain Proposition \ref{prop: easy Parisi=CS-minor} from Proposition \ref{prop:Parisi=CS-minor}.
	
	\begin{proof}[Proof of Proposition \ref{prop: easy Parisi=CS-minor}]
		Minimization of Proposition \ref{prop:Parisi=CS-minor} over $\vec{m}$, combined with (\ref{eqn:lem:finite to parametrized}), demonstrates that if we have that $\xi'_x(0),\xi_x''(0)\neq 0$ for each $x\in \Om$ then
		\[\inf_{(\zeta,\b{\Phi})\in \mathscr{Y}_{fin}}\cal{B}(\zeta,\b{\Phi})=\inf_{(\zeta,\b{\Phi})\in \mathscr{Y}_{fin} }\cal{A}(\zeta,\b{\Phi}).\label{eqn:ignore-0434873}\]
		What is left is to show this is true if we remove the condition on $\b{\xi}$. However, by Proposition \ref{prop:overview:continuity of parisi in xi} the right hand side is continuous in $\b{\xi}$. Moreover, by applying Corollary \ref{corr:overview:continuity of free energy in xi} and Theorem \ref{theorem:spherical parisi continuum} we see that the same holds for $\cal{A}$. In particular, the general statement follows from (\ref{eqn:ignore-0434873}) and a density argument in $\b{\xi}$.
	\end{proof}
	
	The remainder of this section is devoted to proving Proposition \ref{prop:Parisi=CS-minor}. The proof follows the standard framework of variational calculus. We compute the derivatives of each functional in $\vec{\b{s}}$ (and for $\cal{A}$, also in $\b{b}$). Assuming these derivatives vanish for hold some $(\vec{\b{s}},\vec{m},\b{b})$, then we show that $\cal{A}(\vec{\b{s}},\vec{m},\b{b})=\cal{B}(\vec{\b{s}},\vec{m})$. After this, we verify that each functional indeed has a global minimizer for which all derivatives vanish. This fact is non-trivial, as the functionals' shared domain is neither compact nor open, and in fact turns out to be quite difficult.
	
	For simplicity, we fix $D$ and use the abusive notation $K:=K^D$ and $\Lambda:=\Lambda^D$ for this section alone. We begin with our result which compares the functionals, roughly assuming a set of minimization equations.
	
	\begin{lem}
		\label{lem:CS-section:main-lemma}
		Assume that $(\vec{\b{s}},\vec{m},\b{b})$ satisfies, for all $x\in \Om$ and $1\le l< r$,
		\[
		m_{l}\left(\xi'_x(\b{s}^{l+1}(x))-\xi'_x(\b{s}^l(x))\right)=K_x(\b{\delta}^{l+1})-K_x(\b{\delta}^{l}),\label{eqn:lem:min-cs-1}
		\]
		\[
		[(D+\b{b}-\b{d}^r)^{-1}]_{xx}=\b{\delta}^{r}(x)=1-\b{s}^r(x).\label{eqn:lem:min-cs-b}
		\]
		Then $\cal{A}(\vec{\b{s}},\vec{m},\b{b})=\cal{B}(\vec{\b{s}},\vec{m})$.
	\end{lem}
	\begin{proof}
		We may write (\ref{eqn:lem:min-cs-1}) and (\ref{eqn:lem:min-cs-b}) as
		\[\b{d}^{l+1}(x)-\b{d}^{l}(x)=K_x(\b{\delta}^{l+1})-K_x(\b{\delta}^{l}) \text{  and  } \b{b}(x)-\b{d}^r(x)=K_x(\b{\delta}^k).\]
		From this and (\ref{eqn:lem:min-cs-b}), we inductively conclude, for each $1\le l \le r$ and $x\in \Om$ that
		\[\b{b}(x)-\b{d}^l(x)=K_x(\b{\delta}^l) \text{ and } [(D+\b{b}-\b{d}^l)^{-1}]_{xx}=\b{\delta}^l(x).\]
		Employing this and canceling the term involving only $\log \det(D+\b{b})$ we see that
		\[
		\begin{split}
			\sum_{1\le k\le r}\frac{1}{m_k}\log \frac{\det(D+\b{b}-\b{d}^{k+1})}{\det(D+\b{b}-\b{d}^k)}-\log \det(D+\b{b})=\\
			-\log \det(D+\b{K}(\b{\delta}^r))+\sum_{1\le k<r}\frac{1}{m_k}\bigg(\log \det(D+\b{K}(\b{\delta}^{k+1}))-\log \det(D+\b{K}(\b{\delta}^k))\bigg).\label{eqn:ignore-1334}
		\end{split}
		\]
		Using (\ref{eqn:lem:min-cs-1}) and noting that  $\b{\delta}^l(x)-\b{\delta}^{l+1}(x)=m_l(\b{s}^{l+1}(x)-\b{s}^l(x))$, we see that
		\[
		\begin{split}
			\sum_{1\le k<r}m_k(\b{s}^{k+1}(x)\xi_x'(\b{s}^{k+1}(x))-\b{s}^k(x)\xi_x'(\b{s}^k(x)))=\\
			\sum_{1\le k< r}(\b{\delta}^{k}(x)-\b{\delta}^{k+1}(x))\xi_x'(\b{s}^{k+1}(x))+\b{s}^k(x)m_k(\xi_x'(\b{s}^{k+1}(x))-\xi_x'(\b{s}^k(x)))=\\
			\sum_{1\le k<r}(\b{\delta}^{k}(x)-\b{\delta}^{k+1}(x))\xi_x'(\b{s}^{k+1}(x))+\sum_{1\le k<r}\b{s}^k(x)\left(-K_x(\b{\delta}^k)+K_x(\b{\delta}^{k+1})\right).\label{eqn:ignore-1011}
		\end{split}
		\]
		We may write the first summation on the right as
		\[
		\begin{split}
			&\b{\delta}^{1}(x)\xi'_x(\b{s}^1(x))-\b{\delta}^{r}(x)\xi'_x(\b{s}^r(x))+\sum_{1\le k<r}\b{\delta}^{k}(x)\left(\xi_x'(\b{s}^{k+1}(x))-\xi_x'(\b{s}^k(x))\right)=\\
			&\b{\delta}^{1}(x)\xi'_x(\b{s}^1(x))-\b{\delta}^{r}(x)\xi'_x(\b{s}^r(x))-\sum_{1\le k<r}\frac{\b{\delta}^{k}(x)}{m_k}\left(K_x(\b{\delta}^{k+1})-K_x(\b{\delta}^k)\right).
		\end{split}
		\]
		Similarly, we may write the second summation on the right as
		\[
		\begin{split}
			&\b{s}^r(x)K_x(\b{\delta}^r)-\b{s}^1(x)K_x(\b{\delta}^1)+\sum_{1\le k<r}K_x(\b{\delta}^{k+1})(-\b{s}^{k+1}(x)+\b{s}^k(x))=\\
			&\b{s}^r(x)K_x(\b{\delta}^r)-\b{s}^1(x)K_x(\b{\delta}^1)-\sum_{1\le k<r}\frac{K_x(\b{\delta}^{k+1})}{m_k}(\b{\delta}^{k+1}(x)-\b{\delta}^k(x)).
		\end{split}
		\]
		In particular we may write the quantity in (\ref{eqn:ignore-1011}) as
		\[
		\begin{split}
			\b{\delta}^{1}(x)\xi'_x(\b{s}^1(x))-\b{\delta}^{r}(x)\xi'_x(\b{s}^r(x))+\b{s}^r(x)K_x(\b{\delta}^r)-\b{s}^1(x)K_x(\b{\delta}^1)\\
			+\sum_{1\le k<r}\frac{1}{m_k}\left(\b{\delta}^{k}(x) K_x(\b{\delta}^k)-\b{\delta}^{k+1}(x)K_x(\b{\delta}^{k+1})\right).
		\end{split}
		\]
		Combining this with (\ref{eqn:ignore-1334}) and recalling the definition of $\Lambda$ (\ref{eqn:def:Lambda}) we see that,
		\[
		\begin{split}
			|\Om|\Lambda(\b{\delta}^r)&-|\Om|\sum_{1\le k<r}\frac{1}{m_k}\bigg(\Lam(\b{\delta}^{k+1})-\Lam(\b{\delta}^{k})\bigg)
			=\\
			&-\sum_{x\in \Om}\bigg(
			\sum_{1\le k<r}m_k(\b{s}^{k+1}(x)\xi_x'(\b{s}^{k+1}(x))-\b{s}^k(x)\xi_x'(\b{s}^k(x)))\bigg)\\
			&+\sum_{x\in \Om}\bigg(K_x(\b{\delta}^r)+\b{\delta}^{1}(x)\xi'_x(\b{s}^1(x))-\b{\delta}^{r}(x)\xi'_x(\b{s}^r(x))-\b{s}^1(x)K_x(\b{\delta}^1)\bigg)\\
			&+\sum_{1\le k\le r}\frac{1}{m_k}\log \frac{\det(D+\b{b}-\b{d}^{k+1})}{\det(D+\b{b}-\b{d}^k)}-\log \det(D+\b{b}).\label{eqn:ignore-1052}
		\end{split}
		\]
		In particular, canceling a common summation and recalling that $[(D+\b{b}-\b{d}^1)^{-1}]_{xx}=\b{\delta}^1(x)$, we see that
		\[2|\Om|\left(\cal{A}(\vec{\b{s}},\vec{m},\b{b})-\cal{B}(\vec{\b{s}},\vec{m})\right)=\sum_{x\in \Om}\bigg(\b{b}(x)+\b{\delta}^r(x)\xi_x'(\b{s}^r(x))-K_x(\b{\delta}^r)\bigg).\label{eqn:ignore-455}\]
		Using the inverse of (\ref{eqn:lem:min-cs-b}) we see that
		\[\b{b}(x)-(1-\b{s}^r(x))\xi'_x(\b{s}^r(x))=K_x(\b{\delta}^r),\]
		which shows all the terms on the right-hand of (\ref{eqn:ignore-455}) vanishes, completing the proof.
	\end{proof}
	
	While we described the equations in the above lemma as minimization equations, we now further elaborate on them. The first set of equations essentially emerge by setting all derivatives of either the map $\vec{\b{s}}\mapsto \cal{A}(\vec{s},\vec{m})$ or $\vec{\b{s}}\mapsto \cal{A}(\vec{s},\vec{m})$ to zero. The second set of equations follow by simultaneously setting the derivatives of the map $\b{b}\mapsto \cal{A}(\vec{s},\vec{m},\b{b})$ to zero.
	
	However, even if we assume each functional has a minimizer, it is not clear that these derivatives vanish, and it is still difficult to obtain the equations of Lemma \ref{lem:CS-section:main-lemma} directly from the derivatives. Our next lemma establishes these claims for $\cal{B}$.

	\begin{lem}
		\label{lem:CS-section:CS-min}
		Fix $r\ge 1$ and $\vec{m}$. Then if $\vec{\b{s}}$ is a minimizer of $\vec{\b{s}}\mapsto \cal{B}(\vec{\b{s}},\vec{m})$, for each $x\in \Om$ and $1\le l<r$, (\ref{eqn:lem:min-cs-1}) holds.
	\end{lem}
	\begin{proof}[Proof of Lemma \ref{lem:CS-section:CS-min}]
		We first rewrite
		\[
		\begin{split}
			2\cal{B}&(\vec{\b{s}},\vec{m})=\log(2\pi)+\frac{1}{|\Om|}\sum_{x\in \Om}\left(-K_x(\b{\delta}^1)\b{s}^1(x)+\xi_x(1)+\sum_{1\le k\le r}(m_k-m_{k-1})\xi_x(\b{s}^k(x))\right)\\
			&+\frac{1}{|\Om|}\sum_{x,y\in \Om}[(D+\b{K}(\b{\delta}^1))^{-1}]_{xy}\b{h}(x)\b{h}(y)-\frac{1}{m_1}\Lambda(\b{\delta}^1)+\sum_{2\le k\le r}\left(\frac{1}{m_{k-1}}-\frac{1}{m_k}\right) \Lambda(\b{\delta}^k).
		\end{split}
		\]
		We compute that 
		\[|\Om|\frac{d\Lambda(\b{\delta}^k)}{d\b{s}^n(x)}=K_x(\b{\delta}^k(x))\frac{d\b{\delta}^k(x)}{d\b{s}^n(x)}.\]
		We also note that
		\[\frac{d \b{\delta}^{l}(x)}{d\b{s}^n(x)}=0 \text{ if } l>n,\;\;\; \frac{d\b{\delta}^{l}(x)}{d\b{s}^n(x)}=-\left(m_{n-1}-m_n\right) \text{ if } l<n,\;\;\; \frac{d \b{\delta}^{n}(x)}{d\b{s}^n(x)}=-m_n.\]
		Using these formulas as well as the simple manipulation
		\[(m_n-m_{n-1})^{-1}m_n\left(\frac{1}{m_{n-1}}-\frac{1}{m_n}\right)=\frac{1}{m_{n-1}},\label{eqn:ignore-1449}\]
		we see that if we denote the derivative in the $x$-th coordinate by $\D_x$, than for $r\ge n>1$
		\[
		\begin{split}
			2|\Om|(m_n&-m_{n-1})^{-1}\frac{d}{d\b{s}^n(x)}\cal{B}(\vec{\b{s}},\vec{m})=\sum_{y\in \Om}\D_x K_y(\b{\delta}^1)\b{s}^1(y)-\xi'_x(\b{s}^n(x))\\
			&+\sum_{w,y,z\in \Om}[(D+\b{K}(\b{\delta}^1))^{-1}]_{wy}[(D+\b{K}(\b{\delta}^1))^{-1}]_{wz}\D_x K_w(\b{\delta}^1)\b{h}(y)\b{h}(z)\\
			&-\frac{1}{m_1}K_x(\b{\delta}^1)
			-\sum_{2\le k<n}\left(\frac{1}{m_{k-1}}-\frac{1}{m_k}\right) K_x(\b{\delta}^k)+\frac{1}{m_{n-1}}K_x(\b{\delta}^n).	\label{eqn:lemma-ignore-1}
		\end{split}
		\]
		Next, we obtain that 
		\[
		\begin{split}
			2|\Om|m^{-1}_{1}\frac{d}{d\b{s}^1(x)}\cal{B}(\vec{\b{s}},&\vec{m})=-\xi_x'(\b{s}^1(x))+\sum_{y\in \Om}\D_x K_y(\b{\delta}^1)\b{s}^1(y)\\
			&+\sum_{w,y,z\in \Om}[(D+\b{K}(\b{\delta}^1))^{-1}]_{wy}[(D+\b{K}(\b{\delta}^1))^{-1}]_{wz}\D_x K_w(\b{\delta}^1)\b{h}(y)\b{h}(z). \label{eqn:lemma-ignore-2}
		\end{split}
		\]
		Let us assume for the moment that both (\ref{eqn:lemma-ignore-1}) and (\ref{eqn:lemma-ignore-2}) vanish. Then equating these expressions and canceling like terms shows that for any $1\le n\le r$ and $x\in \Om$
		\[\xi'_x(\b{s}^n(x))=\xi_x'(\b{s}^1(x))-\frac{1}{m_1}K_x(\b{\delta}^1)-\sum_{2\le k<n}\left(\frac{1}{m_{k-1}}-\frac{1}{m_k}\right) K_x(\b{\delta}^k)+\frac{1}{m_{n-1}}K_x(\b{\delta}^n).\]	
		Subtraction of these relations from each other yields (\ref{eqn:lem:min-cs-1}).
		
		Next we show that if $\vec{\b{s}}$ is a minimizer of $\vec{\b{s}}\mapsto \cal{B}(\vec{\b{s}},\vec{m})$, then both expressions in (\ref{eqn:lemma-ignore-1}) and (\ref{eqn:lemma-ignore-2}) vanish, or equivalently that all derivatives in $\b{s}^k(x)$ vanish. If a minimizer satisfies $0<\b{s}^1(x)<\dots<\b{s}^{r-1}(x)<\b{s}^r(x)<1$ for each $x\in \Om$, this is obvious, with the problems occurring when there are equalities, so that we are on the boundary of functionals domain. On the other hand, if say $\b{s}^{n-1}(x)=\b{s}^n(x)$ for some $n>1$, we see that $\b{K}(\b{\delta}^n)=\b{K}(\b{\delta}^{n-1})$ so one can check that
		\[(m_{n-1}-m_n)^{-1}\frac{d}{d\b{s}^n(x)}\cal{B}(\vec{\b{s}},\vec{m})=(m_{n-2}-m_{n-1})^{-1}\frac{d}{d\b{s}^{n-1}(x)}\cal{B}(\vec{\b{s}},\vec{m}).\label{eqn:ignore-1641}\]
		Applying this inductively to a maximal chain of equalities $\b{s}^k(x)=\b{s}^{k+1}(x)=\dots =\b{s}^{l}(x)$, we see that the derivative in $\b{s}^k(x)$ coincides with the derivative in $\b{s}^l(x)$ up to a positive constant. In particular, this shows derivatives of all must vanish unless either $k=0$ or $l=r+1$. However if $l=r+1$ the minimizer must have that $\b{s}^{r+1}(x)=1$ and so is not in the domain, so we may assume that $k=0$ and $l<r+1$. Then (\ref{eqn:ignore-1641}) shows in particular that $\b{s}^1(x)=0$ and that $\frac{d}{d\b{s}^1(x)}\cal{B}(\vec{\b{s}},\vec{m})\ge 0$, and finally that to conclude all derivatives vanish it suffices to show that $\frac{d}{d\b{s}^1(x)}\cal{B}(\vec{\b{s}},\vec{m})=0$.
		
		Now note that the vector $(\frac{d}{d\b{s}^1(x)}\cal{B}(\vec{\b{s}},\vec{m}))_{x\in \Om}$ has all non-negative entries. On the other hand, note that the  first and last terms on the right-hand side of (\ref{eqn:lemma-ignore-2}) are strictly negative, so that
		moving these terms to the left-hand side, we conclude the second term is non-negative. In particular that $(\sum_{y\in \Om}\D_x K_y(\b{\delta}^1)\b{s}^1(y))_{x\in \Om}$ has non-negative entries. However the matrix $\D_x K_y(\b{\delta}^1)$ is negative-definite (see Proposition \ref{prop:appendix:K and Lambda}), so we conclude that all entries of $(\sum_{y\in \Om}\D_x K_y(\b{\delta}^1)\b{s}^1(y))_{x\in \Om}$ are infact zero. Again using (\ref{eqn:lemma-ignore-2}) we conclude that the entries of $(\frac{d}{d\b{s}^1(x)}\cal{B}(\vec{\b{s}},\vec{m}))_{x\in \Om}$ are non-positive, so combined with the non-negativity shown above, we have completed the proof that $\frac{d}{d\b{s}^1(x)}\cal{B}(\vec{\b{s}},\vec{m})=0$ for all $x\in \Om$.
	\end{proof}
	
	Next, we obtain a result for minimizers of $\cal{A}$.
	
	\begin{lem}
		\label{lem:CS-section:Parisi-min}
		Fix $r\ge 1$ and $\vec{m}$, and assume that $\xi'_x(0),\xi''_x(0)\neq 0$ for each $x\in \Om$. Then if $(\vec{\b{s}},\b{b})$ is a minimizer of the $(\vec{\b{s}},\b{b})\mapsto \cal{A}(\vec{\b{s}},\vec{m},\b{b})$ one has for each $x\in \Om$ and $1\le l\le r$, that (\ref{eqn:lem:min-cs-1}) and (\ref{eqn:lem:min-cs-b}) hold.
	\end{lem}
	\begin{proof}
		We first rewrite
		\[
		\begin{split}
			2|\Om|\cal{A}(\vec{\b{s}},\vec{m},\b{b})=&|\Om|\log(2\pi)-\frac{1}{m_1}\log \det(D+\b{b}-\b{d}^1)\\
			&+\sum_{2\le k\le r}\left(\frac{1}{m_{k-1}}-\frac{1}{m_k}\right)\log \det(D+\b{b}-\b{d}^{k})\\ 
			&+\sum_{x\in \Om}\left(\b{b}(x)+\theta_x(1)+\sum_{1\le  k\le r}(m_k-m_{k-1})\theta_x(\b{s}^k(x))\right)\\
			&+\sum_{x,y\in \Om}[(D+\b{b}-\b{d}^1)^{-1}]_{xy}\left(\b{h}(x)\b{h}(y)+\xi_x'(\b{s}^1(x))\right).
		\end{split}
		\]
		We note that
		\[\frac{d \b{d}^l(x)}{d\b{s}^n(x)}=0 \text{ if } l>n,\;\;\; \frac{d\b{d}^l(x)}{d\b{s}^n(x)}=\left(m_{n-1}-m_n\right) \xi_x''(\b{s}^n(x)) \text{ if } l<n,\;\;\; \frac{d \b{d}^n(x)}{d\b{s}^n(x)}=-m_n\xi_x''(\b{s}^n(x)).\]
		Then using that $\theta_x'(s)=s\xi''_x(s)$ and (\ref{eqn:ignore-1449}) we see that for $1<n\le r$
		\[
		\begin{split}
			\xi_x''(\b{s}^n(x))^{-1}&2|\Om|(m_n-m_{n-1})^{-1}\frac{d}{d\b{s}^n(x)}\cal{A}(\vec{\b{s}},\vec{m},\b{b})=\\
			&-\frac{1}{m_1}[(D+\b{b}-\b{d}^1)^{-1}]_{xx}+\sum_{2\le k<n}\left(\frac{1}{m_{k-1}}-\frac{1}{m_k}\right)[(D+\b{b}-\b{d}^{k})^{-1}]_{xx}\\
			&+\frac{1}{m_{n-1}}[(D+\b{b}-\b{d}^{n})^{-1}]_{xx}+\b{s}^n(x)\\
			&-\sum_{y,z\in \Om}[(D+\b{b}-\b{d}^1)^{-1}]_{yx}[(D+\b{b}-\b{d}^1)^{-1}]_{xz} (\xi_y'(\b{s}^1(y))\delta_{yz}+\b{h}(y)\b{h}(z)).\label{eqn:ignore-1939}
		\end{split}
		\]
		The $n=1$ case gives instead that
		\[
		\begin{split}
			\xi_x''(\b{s}^1(x))^{-1}2|\Om|m_{1}^{-1}\frac{d}{d\b{s}^1(x)}\cal{A}(\vec{\b{s}},\vec{m},\b{b})=-\sum_{y,z\in \Om}[(D+\b{b}-\b{d}^1)^{-1}]_{yx}[(D+\b{b}-\b{d}^1)^{-1}]_{xz}\\ \times
			(\xi_y'(\b{s}^1(y))\delta_{yz}+\b{h}(y)\b{h}(z))+\b{s}^1(x).
		\end{split}
		\]
		Let us assume for the moment that we have shown for each $1\le n\le r$ and $x\in \Om$
		\[
		\begin{split}
			0=&-\frac{1}{m_1}[(D+\b{b}-\b{d}^1)^{-1}]_{xx}
			+\sum_{2\le k<n}\left(\frac{1}{m_{k-1}}-\frac{1}{m_k}\right)[(D+\b{b}-\b{d}^{k})^{-1}]_{xx}
			\\
			&+\frac{1}{m_{n-1}}[(D+\b{b}-\b{d}^{n})^{-1}]_{xx}
			+\b{s}^n(x)\\
			&-\sum_{y,z\in \Om}[(D+\b{b}-\b{d}^1)^{-1}]_{yx}[(D+\b{b}-\b{d}^1)^{-1}]_{xz}(\xi_y'(\b{s}^1(y))\delta_{yz}+\b{h}(y)\b{h}(z)). \label{eqn:ignore-1929}
		\end{split}
		\] 
		Subtracting these equations from each other we obtain that for $r>n\ge 1$
		\[\b{s}^{n+1}(x)-\b{s}^n(x)=\frac{1}{m_n}\bigg([(D+\b{b}-\b{d}^{n})^{-1}]_{xx}-[(D+\b{b}-\b{d}^{n+1})^{-1}]_{xx}\bigg).\label{eqn:ignore-804}\]
		Minimization over $\b{b}(x)$ yields (\ref{eqn:ignore-b-min}), and subtracting and (\ref{eqn:ignore-1929}) with $n=r$ yields (\ref{eqn:lem:min-cs-b}). Combined with (\ref{eqn:ignore-804}) we obtain that for any $1\le k\le r$
		\[[(D+\b{b}-\b{d}^k)^{-1}]_{xx}=\b{\delta}^k(x).\]
		Applying the inverse function $\b{K}$ to this relation, and then taking differences, yields (\ref{eqn:lem:min-cs-1}).
		
		What is left is to verify that (\ref{eqn:ignore-1929}) holds. By an argument similar to that in Lemma \ref{lem:CS-section:CS-min}, it suffices to assume that $\b{s}^1(x)=0$ for some $x\in \Om$ and show that $\frac{d}{d\b{s}^1(x)}\cal{A}(\vec{\b{s}},\vec{m},\b{b})=0$. However, if $\b{s}^1(x)=0$, then we have that
		\[
		\begin{split}
			\xi_x''(0)^{-1}2|\Om|m_{1}^{-1}\frac{d}{d\b{s}^1(x)}\cal{A}(\vec{\b{s}},\vec{m},\b{b})&=\\
			-\sum_{y,z\in \Om}[(D+\b{b}-\b{d}^1)^{-1}]_{yx}[(D+\b{b}-\b{d}^1)^{-1}]_{xz}(\xi_y'(\b{s}^1(y))\delta_{yz}+\b{h}(y)\b{h}(z))&\le \\
			-[(D+\b{b}-\b{d}^1)^{-1}]_{xx}[(D+\b{b}-\b{d}^1)^{-1}]_{xx}\xi_x'(0)&<0.
		\end{split}
		\]
		This contradicts the minimization assumption.
	\end{proof}
	
	Now finally we give a lemma which establishes the existence of minimizers.
	
	\begin{lem}
		\label{lemma:CS:minimzers exist}
		Fix $\vec{m}$ as above, and assume that $\xi''_x(0)\neq 0$ and $\xi'_x(0)\neq 0$ for each $x\in \Om$. Then both $(\vec{\b{s}},\b{b})\to \cal{A}(\vec{\b{s}},\vec{m},\b{b})$ and $\vec{\b{s}}\to \cal{B}(\vec{\b{s}},\vec{m},\b{b})$ possess a minimizer.
	\end{lem}
	
	Before proceeding to the proof of Lemma \ref{lemma:CS:minimzers exist} we complete the proof of Proposition \ref{prop:Parisi=CS-minor}.
	
	\begin{proof}[Proof of Proposition \ref{prop:Parisi=CS-minor}]
		Lemmas \ref{lem:CS-section:main-lemma}, \ref{lem:CS-section:Parisi-min}, and \ref{lemma:CS:minimzers exist} imply that 
		\[\inf_{\vec{\b{s}}}\cal{B}(\vec{\b{s}},\vec{m})\le \inf_{\vec{\b{s}},\b{b}}\cal{A}(\vec{\b{s}},\vec{m},\b{b})= \inf_{\vec{\b{s}}}\cal{A}(\vec{\b{s}},\vec{m}).\]
		For the reverse inequality, note that Lemmas \ref{lem:CS-section:CS-min} and \ref{lemma:CS:minimzers exist} implies that there is $\vec{\b{s}}^*$ such $\inf_{\vec{\b{s}}}\cal{B}(\vec{\b{s}},\vec{m})=\cal{B}(\vec{\b{s}}^*,\vec{m})$ and such that (\ref{eqn:lem:min-cs-1}) holds. However, observe that once $(\vec{\b{s}}^*,\vec{m})$ is fixed, one may apply $\b{K}$ to find a unique $\b{b}^*$ such that $(\vec{\b{s}}^*,\vec{m},\b{b}^*)$ satisfies (\ref{eqn:lem:min-cs-1}). Thus by Lemma \ref{lem:CS-section:main-lemma}
		\[\inf_{\vec{\b{s}}}\cal{B}(\vec{\b{s}},\vec{m})=\cal{B}(\vec{\b{s}}^*,\vec{m})=\cal{A}(\vec{\b{s}}^*,\vec{m},\b{b}^*)\ge \inf_{\vec{\b{s}}} \cal{A}(\vec{\b{s}},\vec{m}).\]
	\end{proof}
	
	All that is left is to prove Lemma \ref{lemma:CS:minimzers exist}. We begin with a technical result we will use to show the existence of minimizers for $\cal{B}$.
	\begin{lem}
		\label{lem:CS-section:CS no points near 1}
		For a choice of $(\zeta,\b{\Phi})\in \mathscr{Y}$, and $x\in \Om$, we define the function
		\[
		\begin{split}
			\cal{G}_x(q):=&\sum_{x\in \Om}\left(\int_0^q \D_y K_x(\b{\delta}(u))\Phi'_x(u)du\right)+\xi'_x(\Phi_x(q))\\
			&+\sum_{y,z\in \Om}[(D+\b{K}(\b{\delta}(0)))^{-1}]_{yx}[(D+\b{K}(\b{\delta}(0)))^{-1}]_{zx}\b{h}(y)\b{h}(z).
		\end{split}
		\]
		Then there exists some small $c>0$, such that if we denote  $q_M=\sup (\supp (\zeta))$ and assume that $(\zeta,\b{\Phi})$ is such that $\cal{G}_x(q_M)\ge 0$ for each $x\in \Om$, then we have that
		\[\supp(\b{\Phi}_*(\zeta))\in [0,1-c]^{\Om}.\]
	\end{lem}
	\begin{proof}
		For the moment, let us choose some $(\zeta,\b{\Phi})$ in $\mathscr{Y}$. We observe the trivial identity
		\[\delta_x(q)=\int_q^1\zeta([0,u]) \Phi'_x(u)du\le \int_q^1\Phi'_x(u)du=1-\Phi_x(q)\le 1.\]
		By applying Theorem \ref{theorem:appendix:K lemma hard}, there is $\epsilon:=\epsilon(D)>0$ such that
		\[\begin{split}
			\sum_{y\in \Om}\int_0^{q_M}\D_yK_x(\b{\delta}(q))\Phi'_x(q)dq\le -\epsilon\int_0^{q_M}\delta_x(q)^{-2}\Phi'_x(q)dq+\epsilon^{-1}\sum_{y\in \Om\setminus\{x\}}\int_0^{q_M}d\Phi_y(q)\le\\
			-\epsilon\int_0^{q_M}(1-\Phi_x(q))^{-2}\Phi'_x(q)dq+\epsilon^{-1}(|\Om|-1)=\epsilon\left(1-\frac{1}{1-\Phi_x(q_M)}\right)+\epsilon^{-1}(|\Om|-1).
		\end{split}\]
		We also note that
		\[
		\begin{split}
			\sum_{y,z\in \Om}[(D+\b{K}(\b{\delta}(0)))^{-1}]_{yx}[(D+\b{K}(\b{\delta}(0)))^{-1}]_{zx}\b{h}(y)\b{h}(z)\le \\
			\left(s_1((D+\b{K}^D(\b{\delta}(0)))^{-1})\right)^2\|\b{h}\|^2\le 
			\left(|\Om|\tr((D+\b{K}^D(\b{\delta}(0)))^{-1})\right)^2\|\b{h}\|^2=\\
			\left(\sum_{x\in \Om}\delta_x(0)\right)^2\|\b{h}\|^2\le |\Om|^2\|\b{h}\|^2.
		\end{split}
		\]
		In particular,
		\[\cal{G}_x(q_M)\le \epsilon\left(1-\frac{1}{1-\Phi_x(q_M)}\right)+\xi'_x(1)+\epsilon^{-1}(|\Om|-1)+|\Om|^2\|\b{h}\|^2.\]
		We note that the final term on the right-hand side grows to $-\infty$ as $\Phi_x(q_M)\to 1$, so it is clear there is some fixed $c_x:=c_x(D,\xi_x'(1),\|\b{h}\|)>0$ so that if $\cal{G}_x(q_M)\ge 0$ we must have that $\Phi_x(q_M)< 1-c_x$. Taking $c=\min_{x\in \Om}c_x$ completes the proof.
	\end{proof}
	
	With this, we can now prove the existence of minimizers for $\cal{B}$.
	\begin{proof}[Proof of Lemma \ref{lemma:CS:minimzers exist} for $\cal{B}$]
		Let us denote, for $0<\epsilon<1$, the space of sequences $\vec{\b{s}}$ such that 
		\[0=\b{s}^0(x)\le \b{s}^1(x)\le \cdots \le \b{s}^r(x)\le 1-\epsilon,\]
		as $\mathscr{S}_\epsilon$. It is clear that $\mathscr{S}_\epsilon$ is compact, so that $\vec{\b{s}}\mapsto \cal{B}(\vec{\b{s}},\vec{m})$ possesses a minimizer over $\mathscr{S}_\epsilon$. As all possible $\vec{\b{s}}$ lie in some $\mathscr{S}_{\epsilon}$, we see that if choose some small $c>0$, and show that for any $0<\epsilon<c$, the minimizer of $\cal{A}(\vec{\b{s}},\vec{m})$ over $\mathscr{S}_{\epsilon}$ actually lies in $\mathscr{S}_{c}$, then
		\[\inf_{\vec{\b{s}}\in \mathscr{S}_c}\cal{B}(\vec{\b{s}},\vec{m})=\inf_{\vec{\b{s}}}\cal{B}(\vec{\b{s}},\vec{m}),\]
		so in particular, the minimizer of $\cal{B}(\vec{\b{s}},\vec{m})$ over $\mathscr{S}_{c}$ is actually a global minimizer. We will prove this claim for $c>0$ satisfying Lemma \ref{lem:CS-section:CS no points near 1}. By this Lemma, we need only show that any minimizer over $\mathscr{S}_{\epsilon}$, say $(\vec{\b{s}},\vec{m})$, satisfies $\cal{G}_x(s^r)\ge 0$ for all $x\in \Om$ if $0<\epsilon<c$.
		
		However, by routine computation and (\ref{eqn:lemma-ignore-1}) we see that
		\[-\cal{G}_x(s^r)=2|\Om|(1-m_{r-1})^{-1}\frac{d}{d\b{s}^r(x)}\cal{B}(\vec{\b{s}},\vec{m}).\]
		However, employing the work around (\ref{eqn:ignore-1641}), we see that minimization over $\mathscr{S}_{\epsilon}$ implies that $\frac{d}{d\b{s}^r(x)}\cal{B}(\vec{\b{s}},\vec{m})\le 0$. Thus the previous relation immediately implies that $\cal{G}_x(s^r)\ge 0$, as desired.
	\end{proof}
	
	Now we focus on the minimizers of $\cal{A}$. To demonstrate the existence of minimizers of $\cal{A}$, it is convenient here, and crucial in subsequent sections, to mildly extend the domain of the functional. The key is that while the functional $\cal{B}$ requires us to avoid pairs $(\zeta,\b{\Phi})$ such that $(\Phi_x)_*(\zeta)$ is supported at $1$ in order to be well-defined, the functional $\cal{A}$ needs no such restriction. 
	
	In particular, let us define the following extension of $\mathscr{Y}$. Let $\zeta$ be a probability measure on $[0,1]$, and let $\b{\Phi}:[0,1]\to [0,1]^{\Om}$ be a non-decreasing function such that (\ref{eq:condition for Psi}) holds with $\b{q}=\b{1}$. We denote the set of such pairs as $\mathscr{Y}^0$. We observe that $\mathscr{Y}\subseteq \mathscr{Y}^0$, with the difference being that $\mathscr{Y}^0$ places no restraints on the support. The expression for $\cal{A}$ given in (\ref{eqn:def:A continuum}) immediately extends $\cal{A}$ to the set $\mathscr{Y}^0$. We define $\mathscr{Y}^0_{fin}$ as the subset of $\mathscr{Y}^0$ such that $\b{\Phi}_*(\zeta)$ has finite support.
	
	In generality, we are now ready to study the minimization of $\cal{A}$ in terms of $\b{b}$.
	
	\begin{lem}
		\label{lem:CS: minimizer in b exists}
		Let us fix a choice of $(\zeta,\b{\Phi})\in \mathscr{Y}^0$, and assume that $\xi''_x(1)\neq 0$ for each $x\in \Om$. Then the function $\b{b}\mapsto \cal{A}(\zeta,\b{\Phi},\b{b})$ is strictly convex and possesses a unique minimizer on its domain if either $\min_{x} d_x(0)>0$ or $\xi'_x(0)\neq 0$ for all $x\in \Om$.
		
		Moreover, the minimizing $\b{b}$ given as the unique solution to the critical equation
		\[
		\begin{split}
			&-\frac{1}{m_1}[(D+\b{b}-\b{d}^1)^{-1}]_{xx}
			+\sum_{2\le k\le r}\left(\frac{1}{m_{k-1}}-\frac{1}{m_k}\right)[(D+\b{b}-\b{d}^{k})^{-1}]_{xx}\\
			&-\sum_{y,z\in \Om}[(D+\b{b}-\b{d}^1)^{-1}]_{yx}[(D+\b{b}-\b{d}^1)^{-1}]_{xz}(\xi_y'(\b{s}^1(y))\delta_{yz}+\b{h}(y)\b{h}(z))+1=0.\label{eqn:ignore-b-min}
		\end{split}
		\]
	\end{lem}
	
	\begin{proof}
		To begin we analyze the terms in the continuum expression (\ref{eqn:def:A continuum}) for $\cal{A}$. We note all terms on the right-hand side are clearly convex, while the second term in the first line is clearly strictly convex as $\xi''_x(r)>0$ for all $r>0$ by the assumption that $\xi''_x(1)\neq 0$. This establishes strict convexity, and so to establish the existence of a unique minimizer, we only need to show that we may restrict the infimum to a compact subset of the domain
		\[\{\b{b}\in \R^{\Om}:D+\b{b}-\b{d}(0)>0\}.\]
		
		For this, note that it is clear that the first and third lines on the right-hand side of (\ref{eqn:def:A continuum}) are non-negative. Moreover, as $\zeta([0,s])\le 1$ and $\Phi_x(s)\le 1$, we have that
		\[
		\int_0^1 \zeta([0,s])\Phi_x(s)\xi_x''(\Phi_x(s))\Phi_x'(s)ds\le \int_0^1 \xi_x''(\Phi_x(s))\Phi_x'(s)ds= \xi_x'(1)-\xi'(0)\le \xi_x'(1).
		\]
		In particular, we obtain the lower bound
		\[
		\cal{A}(\zeta,\b{\Phi},\b{b})\ge \frac{1}{2|\Om|}\left(\sum_{x\in \Om}\b{b}(x)-\sum_{x\in \Om}\xi_x'(1)\right).
		\]
		Noting that we must have that $\b{b}(x)>-s_{1}(D)$, this bound makes it clear that for some large $C>0$ we may restrict the infimum to the subset 
		\[\{\b{b}\in [-s_1(D),C]^{\Om}:D+\b{b}-\b{d}(0)>0\}.\]
		Next, we note that
		\[
		\begin{split}
			\sum_{x\in \Om}\int_0^1 [(D+\b{b}-\b{d}(s))^{-1}]_{xx}\xi''_x(\Phi_x(s))\Phi_x'(s)ds&\ge\\ \sum_{x\in \Om}\int_0^1 [(D+\b{b}-\b{d}(s))^{-1}]_{xx}\zeta([0,s])\xi''_x(\Phi_x(s))\Phi_x'(s)ds& =\\
			-\sum_{x\in \Om}\int_0^1 [(D+\b{b}-\b{d}(s))^{-1}]_{xx}\b{d}'(s)ds &= \\
			\log \det(D+\b{b})-\log\det(D+\b{b}-\b{d}(0)).&
		\end{split}
		\]
		Using this and previous observations, we see that
		\[\cal{A}(\zeta,\b{\Phi},\b{b})\ge \frac{1}{2|\Om|}\left(\log \det(D+\b{b})-\log\det(D+\b{b}-\b{d}(0))-|\Om|s_1(D)-\sum_{x\in \Om}\xi_x'(1)\right).
		\]
		If we then assume that $\epsilon:=\min_{x\in \Om}d_x(0)>0$, then it is clear that $\log \det(D+\b{b})\ge |\Om|\log \epsilon$. In particular, up to a constant $C>0$, we have that
		\[\cal{A}(\zeta,\b{\Phi},\b{b})\ge -C-\frac{1}{2|\Om|}\log\det(D+\b{b}-\b{d}(0)).
		\]
		This clearly shows that for some small $c>0$, we may further restrict the infimum to the domain
		\[\{\b{b}\in [-s_1(D),C]^{\Om}:D+\b{b}-\b{d}(0)\ge cI\},\]
		which completes the proof in this case. If we instead assume that $\xi'_x(0)>0$ for all $x\in \Om$, then we instead use the third term as a lower bound to obtain that
		\[
		\cal{A}(\zeta,\b{\Phi},\b{b})\ge -C+\frac{1}{2|\Om|}\sum_{x\in \Om}[(D+\b{b}-\b{d}(0))^{-1}]_{xx}\xi_x'(0).
		\]
		This term again allows us to restrict the domain of our infimum to 
		\[\{\b{b}\in [-s_1(D),C]^{\Om}:D+\b{b}-\b{d}(0)\ge cI\},\]
		showing a unique minimizer exists in either case. Lastly, the final claim then follows by a routine computation of $2|\Om|\frac{d}{d\b{b}(x)}\cal{A}(\vec{\b{s}},\vec{m},\b{b})$.
	\end{proof}
	
	We are now ready to prove the second half of Lemma \ref{lemma:CS:minimzers exist} and complete the section.
	
	\begin{proof}[Proof of Lemma \ref{lemma:CS:minimzers exist} for $\cal{A}$]
		Observe that if we weaken the restriction on $\vec{\b{s}}$ to simply
		\[0=\b{s}^0(x)\le \b{s}^1(x)\le \cdots \le \b{s}^r(x)\le  \b{s}^{r+1}(x)=1.\label{eqn:CS:s-sequence general}\]
		then the process of Remark \ref{remark:associated to measure: talagrand} still associates to a pair $(\vec{\b{s}},\vec{m})$, a measure $(\zeta,\b{\Phi})\in \mathscr{Y}^0$. Note that with this extensions, the set of possible $\vec{\b{s}}$ is now compact. Thus using Lemma \ref{lem:CS: minimizer in b exists}, and its bounds, it easy to show that the map $\vec{\b{s}}\mapsto \cal{A}(\vec{\b{s}},\vec{m})$ is continous and thus have a minimizer, and using Lemma \ref{lem:CS: minimizer in b exists} again we may conclude the existence of a minimizing pair $(\vec{\b{s}},\b{b})$. This only differs from what we want to show in that $\vec{\b{s}}$ may satisfy the weakened requirement (\ref{eqn:CS:s-sequence general}), so to complete the proof we only need to show that our minimizer satisfies $\b{s}^r(x)<1$ for all $x\in \Om$.
		
		Now note the computations of the derivatives from Lemma \ref{lem:CS-section:Parisi-min} still hold. In particular, combining (\ref{eqn:ignore-b-min}) with (\ref{eqn:lemma-ignore-1}) in the case of $n=r$, we conclude that
		\[2|\Om|(1-m_{r-1})^{-1}\frac{d}{d\b{s}^r(x)}\cal{A}(\vec{\b{s}},\vec{m},\b{b})=\sum_{x\in \Om}[(D+\b{b}-\b{d}^r)^{-1}]_{xx}+\b{s}^r(x)-1.\label{eqn:ignore-fninfg}\]
		If $\b{s}^r(x)=1$, the minimization condition clearly implies that 
		\[\frac{d}{d\b{s}^r(x)}\cal{A}(\vec{\b{s}},\vec{m},\b{b})\le 0.\]
		However, in this case (\ref{eqn:ignore-fninfg}) implies that
		\[2|\Om|(1-m_{r-1})^{-1}\frac{d}{d\b{s}^r(x)}\cal{A}(\vec{\b{s}},\vec{m},\b{b})=\sum_{x\in \Om}[(D+\b{b}-\b{d}^r)^{-1}]_{xx}>0,\]
		which gives a contradiction.
	\end{proof}
	
	\pagebreak
	
	\section{A Useful Approximation of the Parisi Functional\label{section:RPC-intro}}
	
	In this section, we will develop some results needed to establish Theorem \ref{theorem:spherical parisi continuum}. The most important result is Proposition \ref{prop:rpc-intro:infimum identification of B in limit}, which essentially shows that the  functional $\cal{A}$ can be expressed as a limit of certain functionals $\cal{A}_M$ as $M\to \infty$. This is important in our proof, as for example, the upper bound we use to obtain Theorem \ref{theorem:spherical parisi continuum} is of the form
	\[\E f_{N,\mathrm{Sph}}\le \inf_{(\zeta,\b{\Phi})}\cal{A}_N(\zeta,\b{\Phi})+o(1),\]
	and so only becomes tight as $N\to \infty$.
	
	However, the functionals $\cal{A}_M$ are complicated, and require us to recall a number of important objects and constructions from the study of unipartite spin glass models, and in particular the family of Ruelle probability cascades (RPCs). These objects were originally introduced by Ruelle \cite{ruelleOG}, and will play a key role in our derivation. As such, this section also serves to set notation and recall some basic properties. We will try to follow the notation of the book by Panchenko \cite{panchenko} closely, and reference this book for our required results on these objects when possible.
	
	Thus we have broken up this section into three parts. First, we introduce a modification of the class of finite measures used in the previous section (see Remarks \ref{remark:associated to measure} and \ref{remark:associated measure and functions}), which is an important technicality. Next, we recall the RPCs, their associated random measure, and some basic properties. Finally, we introduce some auxiliary functions, and evaluate them on the random measures coming from an RPC to define the functional $\cal{A}_M$. We finish the subsection by introducing and proving Proposition \ref{prop:rpc-intro:infimum identification of B in limit}, the main result of this section, which essentially shows that $\lim_{M\to \infty}\inf \cal{A}_M=\inf \cal{A}$.
	
	\subsection{A Reparameterization of the Set of Finite Measures}
	
	Here we address an important, though tedious, change to the way we use sequences to describe finite pairs $(\zeta,\b{\Phi})$. The need for a second discretization procedure is common, and matches the difference between the procedure of \cite{talagrandOG} and \cite{panchenko}.
	
	In particular, recall that we considered pairs $(\vec{\b{s}},\vec{m})$ which obeyed the inequalities (\ref{eqn:def:Q-sequence-Tal}) and (\ref{eqn:def:x-sequence-Tal}). Via Remark \ref{remark:associated to measure: talagrand}, we may associate with this a pair $(\zeta,\b{\Phi})\in \mathscr{Y}_{fin}$. This parametrization follows the one used in the original work on a one-site case by Talagrand \cite{talagrandOG}.
	
	However, when working with RPCs, it is convenient to restrict ourselves to measures whose support contain both $0$ and $1$ (see  Section 3.4, and especially Remark 4.1, of \cite{panchenko} for a discussion of this). 
	
	However, there are a number of reasons we could not just use this modification in the previous section. The first issue is that while functional $\cal{A}$ may be defined for such measures (see the discussion immediately prior to Lemma \ref{lem:CS: minimizer in b exists}), the functional $\cal{B}$ cannot. The second issue is that as we never minimized over the weights of the measure, only the locations of points in the support, in the previous section. So forcing $0$ (let alone $1$) to be included in the support of the measure would break our argument, and require us to minimize over the weights.
	
	With this explained, we now give our parametrization. To start, take $r\ge 1$, as well as a choice of
	\[0=t_{-1}<t_0 <\cdots < t_r=1.\label{eqn:def:x-sequence-Pan}\]
	For each $x\in \Om$, we additionally choose a sequence $(\b{q}^k(x))_{k=1}^{r}$ such that
	\[0=\b{q}^0(x)\le \b{q}^1(x)\le \cdots \le\b{q}^{r}(x)=1\label{eqn:def:Q-sequence-Pan}.\]
	We use the notation $(\vec{\b{q}},\vec{t})$ as before. We will also be careful to only use the notation similar to $(\vec{\b{q}},\vec{t})$ henceforth, to limit confusion with pairs $(\vec{\b{s}},\vec{m})$ which instead follow (\ref{eqn:def:x-sequence-Tal}) and (\ref{eqn:def:Q-sequence-Tal})
	
	We now describe how to construct a pair $(\zeta,\b{\Phi})\in \mathscr{Y}^0_{fin}$ from these sequences. For convenience, we will split this into two remarks. First, note that we associate a sequence $\vec{q}$ from $\vec{\b{q}}$ by taking
	\[q^k=\frac{1}{|\Om|}\sum_{x\in \Om}\b{q}^k(x).\label{def:q-seq-norm}\]
	This sequence satisfies
	\[0=q^0\le q^1\le \cdots \le q^r=1,\label{eqn:def:q-sequence-Pan-av}\]
	Given the $(\vec{q},\vec{t})$ pair, we may associate a probability measure on $[0,1]$.
	
	\begin{remark}
		\label{remark:associated to measure}
		Let us take sequences $\vec{t}$ and $\vec{q}$ as in (\ref{eqn:def:x-sequence-Pan}) and (\ref{eqn:def:q-sequence-Pan-av}) of length parameter $r$. We may associate to this a discrete probability measure on $\zeta$ on $[0,1]$ given by
		\[\zeta=\sum_{0\le k\le r}(t_k-t_{k-1})\delta_{q^k},\]
		where $\delta_q$ denotes the Dirac $\delta$-function supported at $q$. We note that conversely, any probability measure on $[0,1]$, whose support contains $\{0,1\}$, and consists of $(r+1)$-points or less, may be obtained from such a pair $(\vec{t},\vec{q})$.
		
		Moreover, if a probability measure is such that its support contains $\{0,1\}$ and it has at most $(r+1)$-points in its support, then there is a such pair $(\vec{t},\vec{q})$ which generates it. If it has exactly $(r+1)$-points in its support, this representation is unique.
	\end{remark}
	
	Now that we have associated generated a measure $\zeta$ given $(\vec{t},\vec{q})$, we will generate the pair $(\zeta,\b{\Phi})$ from the additional information in $\vec{\b{q}}$.
	
	\begin{remark}
		\label{remark:associated measure and functions}
		Let us take sequences $\vec{t}$ and $\vec{\b{q}}$ as in (\ref{eqn:def:x-sequence-Pan}) and (\ref{eqn:def:Q-sequence-Pan}). To the data $(\vec{t},\vec{q})$ we may take the measure $\zeta$ associated by Remark \ref{remark:associated to measure}. We define functions $\b{\Phi}:[0,1]\to [0,1]^{\Om}$ by fixing the values $\b{\Phi}(q^k)=\b{q}^k$, and otherwise by linear interpolation. For each $x\in \Om$, we observe that the push-forward measures $(\Phi_x)_*(\zeta)$ are now also measures on $[0,1]$, which are of the form of Remark \ref{remark:associated to measure}, except associated to the pair of sequences $(\vec{t},\vec{\b{q}(x)})$. Moreover, we have that
		\[(\b{\Phi})_*(\zeta)=\sum_{0\le k\le r}(t_k-t_{k-1})\delta_{\b{q}^k}.\]
		Finally, we observe that $(\zeta,\b{\Phi})\in \mathscr{Y}_0$.
	\end{remark}
	
	Now with our mechanism to generate pairs $(\zeta,\b{\Phi})\in \mathscr{Y}^0_{fin}$ from pairs $(\vec{\b{q}},\vec{t})$, our next task will be to give the computation of $\cal{A}(\vec{\b{q}},\vec{t})$.
	
	We define for $x\in \Om$ and $1\le l\le r$
	\[\b{d}^{l}(x)=\sum_{l\le k<r}t_{k}(\xi'_x(\b{q}^{k+1}(x))-\xi'_x(\b{q}^{k}(x))),\;\; \b{d}^{r}(x)=0.\]
	Next, we take a choice of some $\b{b}\in \R^{\Om}$, such that $D+\b{b}-\b{d}^1>0$. Associated with these choices we have that
	\[
	\begin{split}
		\cal{A}(\vec{\b{q}},\vec{t},\b{b})=\frac{1}{2|\Om|}\bigg(&|\Om|\log(2\pi)+\sum_{0\le k<r}\frac{1}{t_k}\log \frac{\det(D+\b{b}-\b{d}^{k+1})}{\det(D+\b{b}-\b{d}^k)}-\log \det(D+\b{b})\\ 
		&+\sum_{x\in \Om}\l(\b{b}(x)-\sum_{0\le  k< r}t_k(\theta_x(\b{q}^{k+1}(x))-\theta_x(\b{q}^k(x)))\r)\\
		&+\sum_{x,y\in \Om}[(D+\b{b}-\b{d}^1)^{-1}]_{xy}(\b{h}(x)\b{h}(y)+\delta_{xy}\xi_x'(0))\bigg),
	\end{split}
	\]
	and then further define
	\[\cal{A}(\vec{\b{q}},\vec{t})=\inf_{\b{b}:D+\b{b}-\b{d}^1>0}\cal{A}(\vec{\b{q}},\vec{t},\b{b}).\]
	
	For comparison, the evaluation of $\cal{A}(\vec{\b{s}},\vec{m},\b{b})$ is given in (\ref{eqn:def:A in terms of s}). Now while the measures arising from both conventions lead to disjoint compose disjoint subsets of $\mathscr{Y}^0_{fin}$, they are both dense (at least up to the equivalence given in Proposition \ref{prop:general: for B, parisi doesn't depend away from support}). As such, one may easily obtain the following result.
	\begin{lem}
		\label{lem: all minimums are equivalent}
		\[\inf_{\vec{\b{q}},\vec{t}} \cal{A}(\vec{\b{q}},\vec{t})=\inf_{\vec{\b{s}},\vec{m}} \cal{A}(\vec{\b{s}},\vec{m})=\inf_{(\zeta,\b{\Phi})\in \mathscr{Y}_{fin}} \cal{A}(\zeta,\b{\Phi})=\inf_{(\zeta,\b{\Phi})\in \mathscr{Y}_{fin}^0} \cal{A}(\zeta,\b{\Phi}).\]
	\end{lem}
	
	\subsection{Properties of Ruelle Probability Cascades}
	
	We will now recall some basic information about RPCs and their associated measures. For a precise construction of these objects and derivation of the stated results, see Chapter 2 of \cite{panchenko}, whose notation we try to closely follow. To begin, take $r\ge 1$, and fix a choice of sequence $\vec{t}$ as in (\ref{eqn:def:x-sequence-Pan}). Given this, there is an associated random probability measure on $\N^r$, whose weights we will denote $(v_{\alpha})_{\alpha\in \N^r}$, known as the associated RPC.
	
	Now fix an infinite dimensional separable real Hilbert space $H$. Given an additional sequence
	there is a further random measure on $H$, which is a.s. supported on some set of non-random collection vectors $(h_{\alpha})_{\alpha\in \N^r}$. These vectors have a key property.
	For $\alpha,\beta\in \N^r$, we will denote by $\alpha \wedge \beta$ the index of the least common ancestor: 
	\[\al\wedge \be=\inf \l(\{r\}\cup \{i:\al_i\neq \be_i\}\r).\]
	Then the vectors are such that $(h_{\al},h_{\be})_H=q^{\al \wedge \be}$. $G$ is then given by the random probability measure, $G$, supported on the points $h_{\alpha}\in H$, with weights $G(\{h_{\alpha}\})=v_{\alpha}$.
	
	We will now begin recalling the properties of $G$ we will require, beginning first with how to recover some of the data $(\vec{t},\vec{q})$ from it. Let us denote the probability measure associated with our fixed $(\vec{t},\vec{q})$ via Remark \ref{remark:associated to measure} as $\zeta$. The first property we need (see (2.82) of \cite{panchenko}) is that the law of $(v,v')_H$, where $(v,v')\in H^{\times 2}$ is sampled from $\E G^{\otimes 2}$, is given by $\zeta$. That is for any continuous function $f:\R\to \R$
	\[\E \<f((v,v')_H)\>=\int f(q)\zeta(dq),\label{eqn:average is zeta}\]
	where $\<*\>$ denotes the average of $(v,v')$ under $G^{\otimes 2}$.
	
	The next property we will need is the existence of some centered Gaussian processes defined on $\{h_{\alpha}\}_{\alpha\in \N^r}$. Let $\phi:[0,1]\to \R$ be any non-decreasing function. Then there is a centered Gaussian process, $g_{\phi}$, defined on the points $\{h_{\alpha}\}$, such that for $\al,\be\in \N^r$,
	\[\E [g_{\phi}(h_{\alpha})g_{\phi}(h_{\beta})]=\phi(q^{\al\wedge \be}).\label{eqn:property of RPC:gaussians}\]
	
	\begin{remark}
		\label{remark:alpha-h alpha summation}
		We observe that if we denote by $\<f(v)\>$ the average over $v$ sampled from $G$, then defining $f(\alpha):=f(h_\alpha)$, we have that
		\[\<f(v)\>=\sum_{\alpha\in \N^r} v_{\alpha}f(\alpha).\]	
		We at times employ the latter method of writing this average when we believe one is more clear.
	\end{remark}
	
	We now state the final property we need for $G$. Let us fix some random variable $\omega$ taking values in some Borel space $T$ (for example $\R^n$), as well as some measurable function $F:T^r\to \R$. Let us furthermore fix a family of i.i.d copies of $\omega$, $(\omega_{\beta})_{\beta\in \mathscr{A}}$. We will use the abuse of notation by letting, for $\alpha\in \N^r$,  
	\[F(h_{\alpha})=F(\alpha):=F((\omega_{\beta})_{\beta\in p(\alpha)}).\] 
	We will need to evaluate expressions of the form
	\[\E \log \<\exp F(v)\>=\E \log \left(\sum_{\alpha\in \N^r}v_{\alpha}\exp(F(\alpha))\right).\]
	The final result we need is a recursive formula to compute this. Define $F_r=F$, and for $0\le k<r$, we define $F_k:T^{k}\to \R$ by the iterative relation
	\[F_k(z_0,\dots, z_k)=\frac{1}{t_k}\log \E \exp t_k F_{k+1}(z_0,\dots, z_k, \omega).\label{eqn:ignore-1293923}\]
	It then follows from Theorem 2.9 of \cite{panchenko} that 
	\[\E \log \left(\sum_{\alpha\in \N^r}v_{\alpha}\exp F(\alpha)\right)=F_0.\label{eqn:property of RPC:recursion}\]
	
	Finally, we will need to recall an important invariance-property which RPC's satisfy, which follows as a special case of Theorem 2.11 of \cite{panchenko}. Namely, with $F$ as above and any $f:[-1,1]\to \R$ we have that
	\[\E \left[\frac{\sum_{\alpha,\alpha'\in \N^r}f(q^{\alpha \wedge \alpha '})v_\alpha v_{\alpha'} \exp(F(\alpha)+F(\alpha'))}{\left(\sum_{\alpha\in \N^r}v_\alpha \exp(F(\alpha))\right)^2}\right]=\int f(q)\zeta(dq).\label{eqn:invariance of RPC}\]
	Note that when $F=0$ this coincides with (\ref{eqn:average is zeta}) above.
	
	\subsection{An Approximate Parisi Functional \label{subsection:Abstract Functional}}
	
	We now define a family of functionals defined on a certain class of random probability measures. We will primarily evaluate these functionals on the Gibbs measure associated to our Hamiltonians (or more precisely, their perturbed variants) as well as on a family of RPCs which we will show approximate our Gibbs measure. For now however, it will be advantageous to adopt a general framework.
	
	Let $\mu$ be a random probability measure on $H$. We will assume that a.s. the support of $\mu$ lies in a fixed subset of the unit ball of $H$. Next, let $\b{R}:H\times H\to [-1,1]^{\Om}$ be some symmetric continuous function such that a.s. $\b{R}(u,u)(x)=1$ for any $x\in \Om$ and $u\in H$ sampled from $\mu$.
	
	Now first assume that, a.s. in $\mu$, there is a sequence, $(Z_x)_{x\in \Om}$, of independent centered Gaussian processes defined on a set containing the support of $\mu$, such that for $u,u'\in H$,
	\[\E Z_x(u)Z_x(u')=\xi_x(\b{R}_x(u,u')).\label{def:RPC:Z_x}\]
	Moreover, assume there is an additional independent centered Gaussian process, $Y$, also defined on a set containing the support of $\mu$, such that
	\[\E Y(u)Y(u')=\sum_{x\in \Om}\theta_x(\b{R}_x(u,u')).\label{def:RPC:Y}\]
	Both of these processes are assumed to be independent of $\mu$.
	
	Now fix $M\ge 1$, and let $m$ be a finite (deterministic) positive measure on $S_M$, and let $\vec{Z}_x(u)=(Z_x^i(u))_{1\le i\le M}$ be a sequence of i.i.d copies of $Z_x(u)$. We consider
	\[\Gamma^M_{1,m,\b{\xi}}(\mu,\b{R})=|\Om|^{-1}M^{-1}\E \log \left\<\int_{S_M^{\Om}}\exp \left(\sum_{x\in \Om}(\vec{Z}_x(U),\b{v})\right)m(d\b{v})\right \>_{\mu},\]
	where here $U$ is sampled from $\mu$, and $\<*\>_{\mu}=\E[*|\mu,\vec{Z}_x(U)]$ denotes the expectation taken only in $U$. Similarly we introduce,
	\[\Gamma^M_{2,m,\b{\xi}}(\mu,\b{R})=|\Om|^{-1}M^{-1}\E \log \<\exp( \sqrt{M} Y(U))\>_{\mu},\]
	and finally our functional
	\[\Gamma^M_{m,\b{\xi}}(\mu,\b{R})=\Gamma^M_{1,m,\b{\xi}}(\mu,\b{R})-\Gamma^M_{2,m,\b{\xi}}(\mu,\b{R}).\]
	
	We need a result that shows we may approximate $\Gamma^M_{m,\b{\xi}}$ in terms of a continuous function of  a number of i.i.d samples (i.e. replicas) from $\mu$.  This result is an easy modification of Theorem 1.3 of \cite{panchenko} (see also Proposition 2.6 of \cite{erik}) and its proof is omitted.
	\begin{lem}
		\label{lem:perturbation:approximation of Parisi formula}
		Choose $\epsilon>0$ and $M\ge 1$. Then we may find $\ell\ge 1$ and a function $F_{\epsilon}:([-1,1]^{\Om})^{\ell\times \ell}\to \R$ such that for all $(\mu,\b{R})$ in the domain of $\Gamma^M_{m,\b{\xi}}$,
		\[|\Gamma^M_{m,\b{\xi}}(\mu,\b{R})-\E \<F_{\epsilon}([\b{R}(U_i,U_j)]_{1\le i,j\le \ell})\>_{\mu}|\le \epsilon,\]
		where $U_1,\dots U_\ell$ are i.i.d samples from $\mu$, and $\<*\>_{\mu}$ denotes the expectation taken only with respect to these.
	\end{lem}

	With all of these preliminaries established, we will now evaluate the above functional on the class of RPCs, on a certain class of functions. We will then relate this sequence of functionals to the Parisi functional $\cal{A}$ considered above.
	
	Let $(\zeta,\b{\Phi})$ be the pair associated to $(\vec{t},\vec{\b{q}})$ by Remark \ref{remark:associated measure and functions}. Next we define $\b{R}:H\times H\to [-1,1]^{\Om}$ by
	\[\b{R}^{\b{\Phi}}(v,v')=\b{\Phi}(\min(\max((v,v')_H,0),1)).\label{eqn:rpc-1:R^Phi}\]
	Let us further let the measure $m$ on $S_M^{\Om}$ being given
	\[\omega_{M,D,\b{h}}(d\b{u})=\exp\left(-\frac{1}{2}\sum_{x,y\in \Om}D_{xy}(\b{u}(x),\b{u}(y))+\sq{M}\sum_{x\in \Om}\b{h}(x) \b{u}_1(x)\right)\omega(d\b{u}).\label{def:omega-measure-D}\]
	Let $G_{\zeta}$ be the RPC considered in the previous subsection. Then for any $M\ge 1$, we define
	\[\cal{A}_M(\vec{\b{q}},\vec{t})=\Gamma^M_{\omega_{M,D,\b{h}},\b{\xi}}(G_{\zeta},\b{R}^{\b{\Phi}}).\label{def:RPC:B-Gamma}\]
	
	Now that we have finally introduced $\cal{A}_M$, we may state the main result of this section.
	
	\begin{prop}
		\label{prop:rpc-intro:infimum identification of B in limit}
		\[\lim_{M\to \infty}\inf_{\vec{\b{q}},\vec{m}}\cal{A}_M(\vec{\b{q}},\vec{t})=\inf_{\vec{\b{q}},\vec{m}}\cal{A}(\vec{\b{q}},\vec{t}).\]
	\end{prop}
	
	To show this we first show a point-wise form of this convergence. Morally, this proof consists of showing that we may replace the integration over $S_M$ in $\cal{A}_M$ with integration over $\R^M$ as $M\to \infty$ by introducing a Lagrange multiplier (which is the $\b{b}$ introduced in the definition of $\cal{A}$). Then we may simply apply the inductive formula (\ref{eqn:ignore-1293923}) to simplify this integration. 
	
	\begin{lem}
		\label{lem:rpc-intro:identification of B in limit}
		We have that
		\[\lim_{M\to \infty}\cal{A}_M(\vec{\b{q}},\vec{t})=\cal{A}(\vec{\b{q}},\vec{t}).\]
	\end{lem}
	
	We will now prepare for the proof of Lemma \ref{lem:rpc-intro:identification of B in limit}. As such, we will now fix the choice of $(\omega_{M,D,\b{h}},\b{\xi},G_{\zeta},\b{R}^{\b{\Phi}})$, and denote $\Gamma^M:=\Gamma^M_{\omega_{M,D,\b{h}},\b{\xi}}(G_{\zeta},\b{R}^{\b{\Phi}})$ and similarly for $\Gamma^M_i$. We may evaluate the term $\Gamma^M_2$ quite simply through the usage of (\ref{eqn:property of RPC:recursion}) to obtain
	\[
	\begin{split}
		\Gamma^M_{2}&=|\Om|^{-1}M^{-1}\E \log\left(\sum_{\al\in \N^r}v_\al e^{\sq{M}Y(\alpha)}\right)\\
		&=\frac{1}{2|\Om|}\sum_{x\in \Om}\sum_{0\le k<r}t_k(\theta_x(\b{q}^{k+1}(x))-\theta_x(\b{q}^{k}(x))).\label{eqn:ignore-713}
	\end{split}
	\]
	To obtain the term $\Gamma^M_{1}$, we will employ a large-derivation argument to replace these spherical integrals with a series of Gaussian integrals, which are $M$-independent, and to which we may effectively employ (\ref{eqn:property of RPC:recursion}).
	
	To begin, for $\b{b}\in \R^{\Om}$, we define for $\b{v}\in \R^{\Om}$, the function
	\[
	\begin{split}
		Y^{\b{b}}(\b{v})=|\Om|^{-1}\E \log \bigg(\sum_{\al\in \N^r}v_\al\int_{\R^\Om} \exp \bigg(&\sum_{x\in \Om}(Z_x(\alpha)+\b{v}(x))\b{u}(x)\\
		&-\frac{1}{2}\sum_{x,x'\in \Om}(D_{xx'}+\b{b}(x)\delta_{xx'})\b{u}(x)\b{u}(x')\bigg)d\b{u}\bigg),\label{eqn:def:Yb-first}
	\end{split}
	\]
	when such an expression exists. We may obtain an formula for this expression by again employing (\ref{eqn:property of RPC:recursion}), similarly to the proof of Lemma 3.1 in \cite{panchenko} (see also the proof of Corollary 3.1 of \cite{justin}). 
	Namely, let us form, for each $1\le k\le r$, independent centered Gaussian vectors $\b{z}^k\in \R^{\Om}$, with covariance
	\[\E[\b{z}^k(x)\b{z}^k(x')]=\delta_{x,x'}\left(\xi_x'(\b{q}^k(x))-\xi'_x(\b{q}^{k-1}(x))\right),\;\;\;\E[\b{z}^0(x)\b{z}^0(x')]=\delta_{x,x'}\xi_x'(0).\]
	Let us then define for $\b{v}\in \R^{\Om}$
	\[
	\begin{split}
		Y_r^{\b{b}}(\b{v})&=|\Om|^{-1}\log\int_{\R^{\Om}} \exp \left(\sum_{x\in \Om}\b{v}(x)\b{u}(x)-\frac{1}{2}\sum_{x,x'\in \Om}(D_{xx'}+\b{b}(x)\delta_{xx'})\b{u}(x)\b{u}(x')\right)d\b{u}\\
		&=\frac{1}{2}\log(2\pi)-\frac{1}{2|\Om|}\log \det(D+\b{b})+\frac{1}{2|\Om|}\sum_{x,x'\in \Om}[(D+\b{b})^{-1}]_{xx'}\b{v}(x)\b{v}(x').
	\end{split}	
	\]
	Then if we recursively define
	\[Y_\ell^{\b{b}}(\b{v})=\frac{1}{t_\ell}\log \E \exp(t_\ell Y_{\ell+1}^{\b{b}}(\b{v}+\b{z}^{\ell+1})),\]
	then we have that $Y^{\b{b}}(\b{v})=Y^{\b{b}}_0(\b{v})$. Evaluating this recursion can also be done exactly, as the first term, as well as each recurrence, is simply a Gaussian integral. In particular, employing Lemma 3.4 of \cite{justin} recursively (see also their Corollary 3.1) we obtain that
	\[
	\begin{split}
		Y^{\b{b}}(\b{v})=\frac{1}{2|\Om|}&\bigg(|\Om|\log(2\pi)-\log \det(D+\b{b})+\sum_{0\le k\le r-1}\frac{1}{t_k}\log\frac{\det(D+\b{b}-\b{d}^{k+1})}{\det(D+\b{b}-\b{d}^{k})}\\
		&+\sum_{x,y\in \Om}[(D+\b{b}+\b{d}^1)^{-1}]_{xy}\left(\b{v}(x)\b{v}(y)+\delta_{x,y}\xi'_x(0)\right)\bigg).\label{eqn:def:upper bound:explicity-y-def}
	\end{split}
	\]
	Finally, let us define the function
	\[W(\b{b})=Y^{\b{b}}(\b{h})+\frac{1}{2|\Om|}\sum_{x\in \Om}\b{b}(x).\]
	It is clear then that this is defined if and only if $D+\b{b}+\b{d}^1>0$. By using (\ref{eqn:ignore-713}), we see that Lemma \ref{lem:rpc-intro:identification of B in limit} equivalent to the following result.
	\begin{lem}
		\label{lem:upper bound:replacement by b}
		We have that
		\[\lim_{M\to \infty}\Gamma^M_1=\inf_{\b{b}}W(\b{b}),\]
		where here $\b{b}$ ranges over the set  $\{\b{b}\in \R^{\Om}:D+\b{b}+\b{d}^1>0\}$.
	\end{lem}
	
	To show this result we will first define a version of $\Gamma^M_1$ which is defined on a thickened version of $S_N$, and which incorporates the choice of $\b{b}$:
	\[
	\begin{split}
		\Gamma^M_{1,\epsilon}(\b{b})=M^{-1}|\Om|^{-1}\E \log\bigg(\sum_{\al\in \N^r}v_\al &\int_{S_M(1-\epsilon,1+\epsilon)^{\Om}} \exp \bigg(\sum_{x\in \Om}(\vec{Z}_x(\alpha),\b{u}(x))\\
		&-\frac{1}{2}\sum_{x,x'\in \Om}(D_{xx'}+\b{b}(x)\delta_{xx'})(\b{u}(x),\b{u}(x'))\bigg)d\b{u}\bigg).
	\end{split}
	\]
	
	To compare this to $\Gamma^M_{1}$ we will need the following result.
	
	\begin{lem}
		\label{lem:thickening of WN}
		For any choice $\b{b}\in \R^{\Om}$ in the domain of $W(\b{b})$, we have that
		\[\lim_{\epsilon\to 0}\limsup_{M\to \infty}\bigg|\Gamma^M_{1,\epsilon}(\b{b})-\Gamma^M_1-\frac{1}{2|\Om|}\sum_{x\in \Om}\b{b}(x)\bigg|=0.\]
	\end{lem}
	\begin{proof}
		This will essentially be an application of Proposition \ref{prop:Guerra-disintegration-new-appendix}, which we will use the notation of. First note it is clear that this lemma applies when $(S,\mu)$ consists of a random Borel measure $\mu$ on some space $S$, which is a.s. finite. We note that $S_M(1-\epsilon,1+\epsilon)^{\Om}=S_M^{2\epsilon}(\b{1}-\b{\epsilon})$, and take the choice $(S,\mu)=(\N^r,(v_\alpha)_{\alpha\in \N^r})$. We take the centered Gaussian functions $V_x(\b{u},\alpha):=(\vec{Z}_x(\alpha),\b{u}(x))$, which on $(S_M(1-\epsilon,1+\epsilon)^{\Om})^\Om\times \N^r$ have the covariance structure
		\[\E[(\vec{Z}_x(\alpha),\b{u}(x))(\vec{Z}_x(\alpha'),\b{u}'(x))]=\xi'(\b{q}^{\alpha\wedge\alpha'}(x))(\b{u}(x),\b{u}'(x)).\]
		Finally, we take the fixed function
		\[g(\b{u})=-\frac{1}{2}\sum_{x,y\in \Om}(D_{xy}+\b{b}(x)\delta_{xy})(\b{u}(x),\b{u}(y))+\sq{M}\sum_{x\in \Om}\b{h}(x)\b{u}_1(x).\label{eqn:def:upper bound:g-def}\]
		Then noting that for $\b{u}\in S_M^{\Om}$, we may write 
		\[\frac{1}{2}\sum_{x\in \Om}\b{b}(x)=\frac{1}{2}\sum_{x\in \Om}\b{b}(x)\|\b{u}(x)\|^2_M,\]
		we may incorporate this into the integral of $\Gamma^M_1$ to conclude that 
		\[|\Om|\bigg|\Gamma^M_{1,\epsilon}(\b{b})-\Gamma^M_1-\frac{1}{2}\sum_{x\in \Om}\b{b}(x)\bigg|,\]
		coincides with the difference considered in Proposition \ref{prop:Guerra-disintegration-new-appendix} with $N=M$ and the above choices, so we are left to control the terms in this Lemma's upper bound. For this, note that the first term is given by $\sqrt{3}(2\epsilon)\sum_{x\in \Om}\sup_{\alpha,\alpha'}\b{q}^{\alpha\wedge\alpha'}(x)$, which is negligible, and similarly for the terms in the second line, so we are only left to bound $M^{-1/2}\|\|\D g\|\|_{L^\infty}\sqrt{|\Om|(2\epsilon)}$.
		For this, note that using the Cauchy-Schwarz inequality repeatedly we see that
		\[
		\begin{split}
			\|\D g(\b{u})\|^2 &\le \sum_{x\in \Om}\left(\bigg\|\sum_{y\in \Om}\left(D_{xy}+\b{b}(y)\delta_{xy}\right)\b{u}(y)\bigg\|^2+M|\b{h}(x)|^2\right)  \\
			&\le \left(\sum_{x,y\in \Om}D_{xy}^2+\sum_{x\in \Om}\b{b}(y)^2\right)\left(\sum_{y\in \Om}\|\b{u}(y)\|^2\right)\le \left(\|D\|_F^2+\|\b{b}\|^2\right)M|\Om|,
		\end{split}
		\]
		which shows that the remaining term is negligible in the limit as well.
	\end{proof}
	
	We note that by extending the integral of $\Gamma^M_{1,\epsilon}$ from $S_M(1-\epsilon,1+\epsilon)^\Om$ to $(\R^M)^{\Om}$, we obtain that $\Gamma^M_{1,\epsilon}(\b{b})\le Y^{\b{b}}(\b{h})$, which with the previous lemma establishes the upper bound of Lemma \ref{lem:upper bound:replacement by b}. To demonstrate the lower bound we will need the following result.
	
	\begin{proof}[Proof of Lemma \ref{lem:upper bound:replacement by b}]
		As noted above, the upper bound follows from Lemma \ref{lem:thickening of WN}, so if we denote by $\b{b}^*$ the minimizer obtained from Lemma \ref{lem:CS: minimizer in b exists}, it suffices to show that
		\[\limsup_{M\to \infty}\Gamma_1^M\ge W(\b{b}^*).\]
		It is clear that $W(\b{b})$ is differentiable, so for each $x\in \Om$, we have that $\partial_{\b{b}(x)}W(\b{b}^*(x))=0$. Now if we define, for subsets $A\subseteq (\R^M)^{\Om}$,
		\[Y^{\b{b},M}(A)=M^{-1}|\Om|^{-1}\E \log \bigg(\sum_{\al\in \N^r}v_\al\int_{A} \exp \bigg(\sum_{x\in \Om}(\vec{Z}_x(\alpha),\b{u}(x))
		+g_{\b{b}}(\b{u})\bigg)d\b{u}\bigg),\]
		with $g_{\b{b}}(\b{u})=g(\b{u})$ is the function defined in (\ref{eqn:def:upper bound:g-def}), though including the explicit $\b{b}$ dependence. We note that $Y^{\b{b},M}((\R^{M})^{\Om})=Y^{\b{b}}(\b{h})$. Moreover, we see that $\Gamma^{M}_{1,\epsilon}(\b{b})=Y^{\b{b},M}(S_M(1-\epsilon,1+\epsilon)^{\Om})$, so again using Lemma \ref{lem:thickening of WN} we see that it suffices to show that
		\[\liminf_{M\to \infty}Y^{\b{b}^*,M}((S_M(1-\epsilon,1+\epsilon))^{\Om})\ge \lim_{M\to\infty}Y^{\b{b}^*,M}((\R^{M})^{\Om})=Y^{\b{b}^*}(\b{h}).\]
		For this, we partition the set $(\R^{M})^{\Om}$ as
		\[(\R^{M})^{\Om}=(S_M(1-\epsilon,1+\epsilon))^{\Om}\cup \bigcup_{x\in \Om}V_x^-\cup V_x^+,\]
		where here
		\[V_x^+=\{\b{u}\in (\R^{M})^{\Om}:\|\b{u}(x)\|^2_M\ge 1+\epsilon\},\;\;V_x^-=\{\b{u}\in (\R^{M})^{\Om}:\|\b{u}(x)\|^2_M\le 1-\epsilon\}.\]
		It suffices to show that for each $x\in \Om$ and choice of sign $\pm$,
		\[\limsup_{M\to \infty}Y^{\b{b}^*,M}(V_x^{\pm})<Y^{\b{b}^*}(\b{h}).\label{eqn:ignore-1236}\]
		We will show this in the case of $V_x^+$, with the case of $V_x^-$ being similar. We first define $\b{b}_t$ by letting $\b{b}_t(x)=\b{b}^*-t$ and $\b{b}_t(y)=\b{b}^*(y)$ for $y\in \Om\setminus \{x\}$, and note that
		\[g_{\b{b}_t}(\b{u})-g_{\b{b}^*}(\b{u})=\frac{t}{2}\|\b{u}(x)\|^2.\]
		Then using the general integral inequality
		\[\log\int_X e^{f(x)+g(x)}\mu(dx)\le \log\int_X e^{f(x)}\mu(dx)+\sup_{x\in X}g(x),\]
		the observation that $\|\b{u}(x)\|^2_M\ge 1+\epsilon$ on $V^+_x$, we see that for $t>0$
		\[Y^{\b{b}^*,M}(V^+_x)\le  Y^{\b{b}_t}(V^+_x)-\frac{t}{2|\Om|}(1+\epsilon)\le Y^{\b{b}_t}(\b{h})-\frac{t}{2|\Om|}(1+\epsilon).\label{eqn:ignore-1244}\]
		Using that $\b{b}^*$ is a critical point, we see that
		\[\frac{d}{dt}\left(Y^{\b{b}_t}(\b{h})-\frac{t}{2|\Om|}(1+\epsilon)\right)\bigg|_{t=0}=-\partial_{\b{b}(x)}W(\b{b}^*)-\frac{\epsilon}{2|\Om|}=-\frac{\epsilon}{2|\Om|}<0.\]
		Thus in particular, for sufficiently small $t>0$, the right-hand side of (\ref{eqn:ignore-1244}) is decreasing. As this coincides with $Y^{\b{b}^*}(\b{h})$ when $t=0$, the inequality (\ref{eqn:ignore-1244}) shows that (\ref{eqn:ignore-1236}).
	\end{proof}
	Given this pointwise convergence result, we now prove a uniform continuity result for $\cal{A}_M(\vec{\b{q}},\vec{t})$ in terms of $(G_{\zeta},\b{R}^{\b{\Phi}})$. For this, we will write \[\Gamma^M(\zeta,\b{\Phi}):=\Gamma^M_{\omega_{M,D,\b{h}},\b{\xi}}(G_{\zeta},\b{R}^{\b{\Phi}}),\]
	and use a similar notation for $\Gamma^M_i(\zeta,\b{\Phi})$.
	
	\begin{lem}
		\label{lem:lowerbound:lipz-cont for Parisi}
		Let $(\vec{\b{q}}_0,\vec{t}_0)$ and $(\vec{\b{q}}_1,\vec{t_1})$ be two pairs of sequences as in (\ref{eqn:def:x-sequence-Pan}) and (\ref{eqn:def:Q-sequence-Pan}), with associated measures and functions $(\zeta_0,\b{\Phi}^0)$ and $(\zeta_1,\b{\Phi}^1)$ given by Remark \ref{remark:associated measure and functions}. Then we have that
		\[|\cal{A}_M(\vec{\b{q}_0},\vec{t}_0)-\cal{A}_M(\vec{\b{q}_1},\vec{t_1})|\le \sum_{x\in \Om}\frac{\xi''_x(1)}{|\Om|}\int_0^1|((\Phi^0_x)_*\zeta_0)([0,q])-((\Phi^1_x)_*\zeta_1)([0,q])|dq.\]
	\end{lem}
	\begin{proof}
		This will follow if we show that for $i=1,2$
		\[|\Gamma_i^M(\zeta_0,\b{\Phi}^0)-\Gamma_i^M(\zeta_1,\b{\Phi}^1)|\le \sum_{x\in \Om}\frac{\xi''_x(1)}{2|\Om|}\int_0^1|((\Phi^0_x)_*\zeta_0)([0,q])-((\Phi^1_x)_*\zeta_1)([0,q])|dq.\label{eqn:ignore-1926}\]
		We show the case of $i=1$, with the case of $i=2$ following a similar argument. 
		
		We know by Lemma \ref{lem:perturbation:approximation of Parisi formula} and Theorem \ref{theorem:perturbation:gg gives RPC} that $\cal{A}_M(\vec{\b{q}},\vec{t})$ only depends on the measure and functions, $(\zeta,\b{\Phi})$, associated to $(\vec{\b{q}},\vec{t})$. In particular, if we have that, for fixed $\ell$ and every $x\in \Om$, the equality
		\[0=\b{q}^{0}(x)\le \dots \le \b{q}^\ell(x)=\b{q}^{\ell+1}(x)\le \dots\le \b{q}^r(x)=1,\]
		then removing the choices of $(t_{\ell},\b{q}^{\ell})$ to obtain a sequence now of length $(r-1)$ does not affect the value of $\cal{A}_M(\vec{\b{q}},\vec{t})$. In particular, by potentially adding such redundancies to both $(\vec{\b{q}_0},\vec{t}_0)$ and $(\vec{\b{q}_1},\vec{t_1})$ are sequences of the same length with $\vec{t_0}=\vec{t_1}$.
		
		In particular, the RPCs corresponding to $\zeta_0$ and $\zeta_1$ may be chosen to have the same weights $(v_\alpha)_{\alpha\in \N^r}$, only being supported on different points in $H$. Using Remark \ref{remark:alpha-h alpha summation}, both may be realized as the free energy of a Gaussian process on $\N^r\times S_M^{\Om}$. We can interpolate these process by defined $0\le t\le 1$, $x\in \Om$, and $\alpha\in \N^r$
		\[Z_x^t(\alpha)=\sqrt{t}Z_x^1(\alpha)+\sqrt{1-t}Z_x^0(\alpha),\]
		with $Z_x^i(\alpha)$ being the choice of $Z_x(\alpha)$ corresponding to $(\vec{\b{q}}_i,\vec{t}_i)$. Then if we denote by $\<*\>_t$ the Gibbs measure on $(\alpha,\b{u})\in \N^r\times S_N^{\Om}$ associated to $(\sum_{x\in \Om}(\vec{Z}_x^t,\b{u}(x)),\omega_{M,D,\b{h}}(d\b{u})v_\alpha)$, we see by applying Proposition \ref{prop:Guerra Gaussian integration} that
		\[\Gamma_1^M(\zeta_0,\b{\Phi}^0)-\Gamma_1^M(\zeta_1,\b{\Phi}^1)=\frac{-1}{2|\Om|}  \sum_{x\in \Om}\E \<\left(\xi_x'(\b{q}_0^{\alpha\wedge\alpha'}(x))-\xi_x'(\b{q}_1^{\alpha\wedge\alpha'}(x))\right)\left(\b{u}(x),\b{u}'(x)\right)_M\>_t.\label{eqn:ignore-1919}\]
		If we use that $\|\b{u}(x)\|_M=\|\b{u}'(x)\|_M=\sq{M}$, we obtain that
		\[
		\begin{split}
			&\l|\sum_{x\in \Om}\l( \xi_x'(\b{q}_0^{\alpha\wedge\beta}(x))(\b{u}(x),\b{u}'(x))- \xi_x'(\b{q}_1^{\alpha\wedge\alpha'}(x))(\b{u}(x),\b{u}'(x))\r)\r|\le \\
			& M\sum_{x\in \Om}| \xi_x'(\b{q}_0^{\alpha\wedge\alpha'}(x))-\xi_x'(\b{q}_1^{\alpha\wedge\alpha'}(x))|\le M\sum_{x\in \Om} \xi_x''(1)|\b{q}_0^{\alpha\wedge\alpha'}(x)-\b{q}_1^{\alpha\wedge\alpha'}(x)|.
		\end{split}
		\]
		In particular, if we define
		\[F_t(\alpha)=\log \int_{S_M^{\Om}} \exp \left(\sum_{x\in \Om}(\vec{Z}_x(\alpha),\b{u}(x))\omega(d\b{u})\right)\omega_{M,D,\b{h}}(d\b{u}),\]
		we see that $|\Gamma_1^M(\zeta_0,\b{\Phi}^0)-\Gamma_1^M(\zeta_1,\b{\Phi}^1)|$ is bounded above by
		\[\sum_{x\in \Om}\frac{\xi''(1)}{2|\Om|} \E \left[ \frac{\sum_{\alpha,\alpha'\in \N^r}v_{\alpha}v_{\alpha'} |\b{q}_0^{\alpha\wedge\alpha'}(x)-\b{q}_1^{\alpha\wedge\alpha'}(x))|\exp(F_t(\alpha)+F_t(\alpha'))}{\left(\sum_{\alpha\in \N^r}v_\alpha \exp(F_t(\alpha))\right)^2}\right].\label{eqn:ignore-1919-1}\]
		Using (\ref{eqn:invariance of RPC}) we see that (\ref{eqn:ignore-1919-1}) coincides with
		\[ \sum_{x\in \Om}\sum_{1\le k\le r}\frac{\xi''_x(1)}{2|\Om|}(t_k-t_{k-1}) |\b{q}_0^{k}(x)-\b{q}_1^{k}(x)|,\]
		which may easily be rearranged into the right-hand side of (\ref{eqn:ignore-1926}).
	\end{proof}
	
	With this we see that the family $\{\cal{A}_M\}_{M\ge 1}$ is continuous on its domain. In particular, we obtain Proposition \ref{prop:rpc-intro:infimum identification of B in limit} from this and Lemma \ref{lem:rpc-intro:identification of B in limit} by a standard argument.
	
	\pagebreak
	
	\section{The Ghirlanda-Guerra Identities, Synchronization, and Talagrand's Positivity Principle\label{section:RPC-gg-iden}}
	
	In further preparation for the proof of Theorem \ref{theorem:intro:bad parisi spherical}, in this section we will introduce a family of perturbations of our Hamiltonian. These perturbations are of insufficient strength to change the limiting value of any of the free energies we consider, but allow us to enforce a number of properties to asymptotically hold for the associated Gibbs measure. Depending on the property, these will asymptotically hold after averaging over the family or with non-zero probability.
	
	The property of key importance for the upper bound of Theorem \ref{theorem:spherical parisi continuum} is Talagrand's positivity principle. Roughly, this principle states that if we takes two samples from the Gibbs measure $\b{u},\b{u}'\in (S_N)^{\Om}$, then for each $x\in \Om$, we should have that $(\b{u}(x),\b{u}'(x))_N\ge -\epsilon$ with high probability when $N$ is large. A formal statement is given in Theorem \ref{theorem:talagrand's positivity} below. This property is of key importance as it makes the error terms appearing in an application of Guerra's interpolation scheme positive, which is otherwise only clear for models in which $\xi_x$ is convex for each $x\in \Om$.
	
	The lower bound however relies on a much stronger property. Roughly, this states that for specific choices of perturbation, the Gibbs measure may be approximated by a sequence of RPCs with ``synchronized" spins, at-least on the level of overlaps. Such a representation was originally discovered by Panchenko in the case of multi-species Ising spin glass models \cite{panchenkoms}, and have since appeared in a number of cases. These both generalize the case of $|\Om|=1$, where it is shown that such perturbations force the Gibbs measure to be approximated by a sequence of RPCs, again on the level of overlaps. This property was again originally demonstrated by Panchenko \cite{panchenko-Ultrametricity}.
	
	We will begin by recalling the family of perturbation Hamiltonians. Then we will give Talagrand's positivity principle. After this, we will recall the multi-species Ghirlanda-Guerra identities, and a formal version of the identification theorem associated with them, which we alluded to in the previous paragraph. Finally, we will show, for a quite general class of Hamiltonians, that there are choices of our perturbation term which lead to approximate satisfaction of the multi-species Ghirlanda-Guerra identities.
	
	For this section, we will fix a family of (random) continuous functions $H_N:S_N^{\Om}\to \R$. We now define the perturbation term, which based on the corresponding multi-species perturbation from \cite{panchenkoms} (see as well Chapter 3 or \cite{panchenko} and Appendix A of \cite{erikCS}). To begin, define for $\b{u},\b{u}'\in (S_N)^{\Om}$, let us define the overlap vector  $\b{R}_N(\b{u},\b{u}')\in [-1,1]^{\Om}$ by
	\[\b{R}_N(\b{u},\b{u}')(x)=(\b{u}(x),\b{u}'(x))_N.\]
	For any $\b{w}\in [0,1]^\Om$, we may also define the following inner-product:
	\[\b{R}_N^{\b{w}}(\b{u},\b{u}')=\frac{1}{|\Om|}\sum_{x\in \Om}\b{w}(x)\b{R}_N(\b{u},\b{u}')(x)=\frac{1}{|\Om|}\sum_{x\in \Om}\b{w}(x)(\b{u}(x),\b{u}'(x))_N.\]
	Let $\cal{W}=\{\b{w}_1,\b{w}_2,\dots \}$ be a countable dense subset of $[0,1]^\Om$, all of which we assume are non-zero. We will assume in addition that $\cal{W}$ contains each standard unit vector in $[0,1]^{\Om}$. To each such choice, we may associate (see (24) of \cite{panchenkoms}) a centered Gaussian random function, $h_{N,p,q}$, on $S_N^{\Om}$ with covariance function given by
	\[\E[ h_{N,p,q}(\b{u})h_{N,p,q}(\b{u}')]=\b{R}_N^{\b{w}_q}(\b{u},\b{u}')^p.\]
	We independently take such a choice for each $(p,q)$. Now fix in addition double-indexed sequence $\frak{X}=(x_{p,q})_{p,q\ge 1}\in [1,2]^{\N\times\N}$ and form the perturbation Hamiltonian
	\[h_N(\b{u})=\sum_{p,q\ge 1}2^{-(p+q)}x_{p,q}h_{N,p,q}(\b{u}).\]
	This is a centered Gaussian process with covariance
	\[\sum_{p,q\ge 1}x_{p,q}^24^{-(p+q)}\b{R}_N^{\b{w}_q}(\b{u},\b{u}')^p.\label{eqn:lower bound:covariance of perturbation term}\] 
	Note that in particular that as $|\b{R}_N^{\b{w}}(\b{u},\b{u}')|\le \frac{1}{|\Om|}\sum_{x\in \Om}|\b{w}(x)|\le 1$, the variance of $h_N(\b{u})$ is bounded by $\sup_{p,q}x_{p,q}^2\le 4$.
	We now introduce the perturbed Hamiltonian
	\[H^{\mathfrak{X}}_N(\b{u})=H_{N}(\b{u})+s_Nh_N(\b{u}),\label{eqn:ignore-38443}\]
	where $s_N=N^{\gamma}$ for any $1/4<\gamma<1/2$.
	
	Let us now fix any sequence of (deterministic) finite positive measures on $S_N^{\Om}$, $m_N$, with continuous bounded density with respect to the uniform measure $\omega$. We denote the partition function over $S_N^{\Om}$ with respect to $(H_N,m_N)$ as $Z_N$, and the partition function with respect to $(H_N^{\frak{X}},m_N)$ as $Z_N^{\frak{X}}$, we note that by Jensen's inequality,
	\[N^{-1}\E\log Z_N\le N^{-1}\E\log Z_N^{\frak{X}}\le N^{-1}\E\log Z_N+2s_N^2N^{-1},\label{eqn:ignore-1825}\]
	so that in particular, we have that
	\[\lim_{N\to \infty}\sup_{x\in [1,2]^{\N\times \N}}|N^{-1}\E\log Z_N-N^{-1}\E\log Z_N^{\frak{X}}|=0.\label{eqn:perturbation:perturbation has no sup effect}\]
	It will also at times be convenient to take an average over $\frak{X}=(x_{p,q})_{p,q\ge 1}\in [1,2]^{\N\times \N}$, where each $(x_{p,q})$ is thought of as an independent uniform r.v. on $[1,2]$. We will denote this average as $\E_{\mathfrak{X}}$, taken independently of all other randomness, and reserve the notation $\E$ to refer to the conditional expectation. The same application of Jensen's inequality shows that
	\[\lim_{N\to \infty}|N^{-1}\E\log Z_N-N^{-1}\E_{\mathfrak{X}}\E\log Z_N^{\frak{X}}|=0.\label{eqn:perturbation:perturbation has no average effect}\]

	We will use the notation $\<*\>_{\frak{X}}$ to denote the Gibbs measure with respect to $(H_N^{\frak{X}},m_N)$. We will need the following multi-species version of the positivity principle, which follows immediately from the deterministic Lemma 2.3 of \cite{erik}.
	\begin{theorem}
		\label{theorem:talagrand's positivity}
		Let us assume that $H_N:S_N^{\Om}\to \R$ is deterministic and satisfies
		\[\int_{S_N^{\Om}} \exp(|H_N(\b{u})|)m_N(d\b{u})<\infty.\label{eqn:ignore-1632}\]
		Fix any choice of $s_N$ such that $s_N\to \infty$ as $N\to \infty$. Then for any $\epsilon>0$, and $x\in \Om$, we have that
		\[\lim_{N\to \infty}\sup_{H_N(\b{u})}\E_{\mathfrak{X}} \<I((\b{u}(x),\b{u}'(x))_N\ge -\epsilon)\>_{\mathfrak{X}}=1,\]
		where here the supremum is taken over all such functions $H_N$ which satisfy (\ref{eqn:ignore-1632}).
	\end{theorem}

	\subsection{The Multi-Species Ghirlanda-Guerra Identities and Their Consequences}
	
	In this subsection we will recall the multi-species Ghirlanda-Guerra identities and some of their important consequences. The non-multi-species version of these identities were first discovered in the titular work by Ghirlanda and Guerra \cite{GhirlandaGuerraOG}. This was followed by extensive work by Talagrand \cite{talagrandchallenge} and Panchenko \cite{panchenko,panchenko-Ultrametricity,panchenkoUnipartite} who found that these identities both contain a given random array to be quite close to the overlaps of a Ruelle probability cascade, and asymptotically hold (possibly after a slight perturbation) for the overlap obtained by taking multiple samples from the Gibbs measure, at least in the case of the one-site model. The multi-species version of these identities were found in the case of multi-site spin glasses by Pancheko \cite{panchenkoms}, who further found they imply a surprising synchronization property which significantly constrains the possible values over the overlap array even more.
	
	\begin{defin}[Multi-Species Ghirlanda-Guerra Identities]
		\label{defin:GG-identities}
		Let us consider a random array $\b{R}\in ([-1,1]^{\N\times \N})^{\Om}$. For any $\ell\ge 1$, we may consider the truncation of this array $\b{R}^{\ell}\in ([-1,1]^{\ell \times \ell})^{\Om}$ by taking only the first $\ell$-coordinates.
		
		Then for any continuous functions $f:([-1,1]^{\ell\times \ell})^{\Om}\to \R$ and $\varphi:[-1,1]^{\Omega}\to \R$, we define the multi-species Guerra-Ghirlanda difference by
		\[
		\Delta(\ell,f,\varphi):=
		\bigg|\ell \E[f(\b{R}^{\ell})\varphi(\b{R}_{1,\ell+1})]-\E[f(\b{R}^{\ell})]\E[\varphi([\b{R}_{1,2}])]
		-\sum_{k=2}^{\ell}\E[f(\b{R}^{\ell})\varphi(\b{R}_{1,k})]\bigg|.\label{eqn:def:Guerra-Ghirlanda difference}
		\]
		If $\Delta(\ell,f,\varphi)=0$ for all possible of choices of $(\ell,f,\varphi)$, we will say that $\b{R}$ satisfies the multi-species Guerra-Ghirlanda identities.
	\end{defin}
	
	The first major consequence of these identities is the following synchronization result due to Panchenko. 
	
	\begin{theorem}[Theorem 4, \cite{panchenkoms}]
		\label{theorem:perturbation:panchenko synch}
		If a random infinite array $\b{R}\in ([-1,1]^{\N\times \N})^{\Om}$ satisfies the multi-species Guerra-Ghirlanda identities then there exists, for each $x\in \Om$, a non-decreasing continuous function $\Phi_x:[0,1]\to [0,1]$, such that a.s. for each $\ell,\ell'\ge 1$, we have that $\b{R}_{\ell,\ell'}(x)=\Phi_x(R_{\ell,\ell'})$, where here
		\[R_{\ell,\ell'}=\frac{1}{|\Om|}\sum_{x\in \Om}\b{R}_{\ell,\ell'}(x).\label{eqn:perturbation:normalized array}\]
	\end{theorem}
	
	This reduces the study of the overlap to the study of array $R=[R_{\ell,\ell'}]_{\ell,\ell'\ge 1}\in [-1,1]^{\N\times \N}$. To do this we now recall the regular Ghirlanda-Guerra identities, which state that for continuous functions $\varphi:[-1,1]\to \R$ and $f:[-1,1]^{\ell\times \ell}\to \R$, we have that
	\[\ell \E[f(R^{\ell})\varphi(R)_{1,\ell+1}]=\E[f(R^{\ell})]\E[\varphi(R_{1,2})]
	+\sum_{k=2}^{\ell}\E[f(R^{\ell})\varphi(R_{1,k})],\label{eqn:def:Guerra-Ghirlanda-classical}\]
	where $R^{\ell}$ is the truncation of $R$ as before. It is clear if $\b{R}$ satisfies the multi-species Guerra-Ghirlanda identities, then $R$ satisfies the regular Ghirlanda-Guerra identities. These identities are still quite powerful, and further determine the law of the array $R$ off the diagonal in terms of the law of a single off-diagonal element. 
	
	To recall this result, we will first need to recall the following representation theorem. We will say an array is exchangeable if for any finitely supported permutation $\sigma:\N\to \N$, we have that
	\[R\disteq R_\sigma:=[[R]_{\sigma(i),\sigma(j)}]_{i,j\ge 1}.\]

	\begin{theorem}[Dovbysh–Sudakov \cite{DovbyshSudakov}]
		\label{theorem:perturbation:dovbysh}
		Let $R\in [-1,1]^{\N\times \N}$ be a symmetric, exchangeable, and positive semi-definite array, and fix a separable infinite Hilbert space $H$. Then there is a random measure, $\mu$ on $H$, such that if $(v^1,v^2,\dots )\in H^{\N}$ are an sequence of i.i.d samples from $\mu$, we have that
		\[[R_{\ell,\ell'}]_{\ell\neq \ell'\ge 1}\disteq [(v^\ell,v^{\ell'})_H]_{\ell\neq \ell'\ge 1},\]
		where $(,)_H$ denotes the inner product on the Hilbert space $H$.
	\end{theorem}
	
	A remarkable fact is that RPCs not only satisfy the Ghirlanda-Guerra identities, but they are essentially dense in that class, which is the content of the following result.
	
	\begin{theorem}[Theorems 2.13 and 2.17 and Remark 2.1 of \cite{panchenko}]
		\label{theorem:perturbation:gg gives RPC}
		Let $\mu$ be a random measure on $H$, a.s. supported on the unit ball, and whose overlap array, $R$, satisfies (\ref{eqn:def:Guerra-Ghirlanda-classical}) for each choice of $(\ell,f,\varphi)$. Then the law of $R$ is uniquely determined by that law of $R_{12}$.
		
		Moreover, let us denote the law of $R_{12}$ as $\zeta$. Then $\zeta$ is supported on $[0,1]$, and $[R]_{\ell\neq \ell'\ge 1}$ coincides in law with the limiting law of the (off-diagonal) overlap array generated by the RPCs corresponding to any sequence of discrete probability measures $\zeta_n$ such that $\zeta_n\To \zeta$ in measure.
	\end{theorem}
	
	\subsection{Effect of the Perturbation}
	
	Now that we have introduced the Ghirlanda-Guerra identities, and highlighted their power, we will show how they can be guaranteed in the limit through the use of the perturbation term given in (\ref{eqn:ignore-38443}). To explain this we will now make the assumption that our choice of Hamiltonian satisfies
	\[\E\int_{S_N^{\Om}}\exp(|H_N(\b{u})|)m_N(d\b{u})<\infty.\label{eqn:perturbation:H'-verify simple}\]
	
	We consider an infinite sequence of vectors sampled from the Gibbs measure $\<*\>_{\mathfrak{X}}$, which we denote $(\b{u}^1,\b{u}^2,\dots )$. Associated to these we form the infinite array $\b{R}_N$ given by
	\[[\b{R}_N]_{i,j}=\b{R}_N(\b{u}^{i},\b{u}^j)\in [-1,1]^{\Om}.\]
	
	Then any choice of $\ell\ge 1$, continuous functions $f:([-1,1]^{\ell\times \ell})^{\Om}\to \R$, $\b{w}\in [0,1]^{\Om}$, and $p\ge 1$, we define
	\[
	\begin{split}
		\Delta_N^{\mathfrak{X}}(\ell,f,\b{w},p):=\bigg|&\ell \E\<f(\b{R}_N^{\ell})(\b{R}_N^{\b{w}_q}(\b{u}^1,\b{u}^{\ell+1}))^p\>_{\mathfrak{X}}-\E\<f(\b{R}_N^{\ell})\>_{\mathfrak{X}}\E\<(\b{R}_N^{\b{w}}(\b{u}^1,\b{u}^{2}))^p\>_{\mathfrak{X}}\\
		&-\sum_{k=2}^{\ell}\E \<f(\b{R}_N^{\ell})(\b{R}_N^{\b{w}}(\b{u}^1,\b{u}^{2}))^p\>_{\mathfrak{X}}\bigg|.
	\end{split}
	\]
	
	Note that this coincides with a choice of Ghirlanda-Guerra difference defined in (\ref{eqn:def:Guerra-Ghirlanda difference}). Namely, if we understand the random measure in question in Definition \ref{defin:GG-identities} as the Gibbs measure of $H_N^{\mathfrak{X}}$, and define the function $\varphi_{\b{w},p}(\b{R})=(\b{R}^{\b{w}})^p$, we have that $\Delta_N^{\mathfrak{X}}(\ell,f,\b{w},p)=\Delta(\ell,f,\varphi_{\b{w},p})$. 
	
	We now show that $\Delta_N^{\mathfrak{X}}(\ell,f,\b{w},p)$ is reasonably small after averaging with respect to $x_{p,q}\in [1,2]$. In particular, we have the following special case of the quite general Theorem A.3 of \cite{erikCS}, which is a direct adaptation of Theorem 3.2 of \cite{panchenko}. Note that as our $\Delta$ differs from theirs by a factor of $\ell$, this factor is missing from the corresponding expression (A.11) in \cite{erikCS}.
	\begin{theorem}[Theorem A.3, \cite{erikCS}]
		\label{theorem:perturbation:eriks}
		Let us fix $p,q\ge 1$. Let us denote make dependance of $Z_N^{\mathfrak{X}}$ on $x_{p,q}$ explicit by writing $Z_N^{\mathfrak{X}}(x_{p,q})$, as well as writing $\Delta_N^{\mathfrak{X}}(\ell,f,\b{w}_q,p,x_{p,q})$. We then fix choices of $x_{p',q'}\in [0,3]$ for $(p',q')\neq (p,q)$. We then define the quantity
		\[v_{N,p,q}=\sup\{\E|\log Z_N^{\mathfrak{X}}(u)-\E \log Z_N^{\mathfrak{X}}(u)|: u\in [0,3]\},\]
		and denote by $\rho=4^{-(p+q)}(\frac{1}{|\Om|}\sum_{x\in \Om}\b{w}_q(x))^p$ the variance of $h_{N,p,q}$. Then for $N$ such that $2^{-1}s_N^{-2}v_{N,p,q}/\rho<1$, we have that 
		\[\int_1^2 \Delta_N(\ell,f,\b{w}_q,p,x)dx\le 24\|f\|_{\infty}\sq{\rho}s_N^{-1}(1+\sq{v_{N,p,q}}).\]
	\end{theorem}
	
	To apply this, let us now finally set our final condition on $H_N$. We assume that it is a centered Gaussian random function with variance bounded by $N\sigma$ for some fixed $\sigma$. In this case, we may apply Theorem \ref{theorem:Borell-TIS} to verify (\ref{eqn:perturbation:H'-verify simple}). Similarly, we recall that $\E[h_{p,q}(\b{u})^2]\le 1$ so with notation Theorem \ref{theorem:perturbation:eriks}, we may apply Proposition \ref{prop:Borell-TIS-free-energy} to conclude that
	\[\P(|\log Z_N^{\mathfrak{X}}(u)-\E \log Z_N^{\mathfrak{X}}(u)|\ge \alpha)\le 2\exp(-\alpha^2/4(N\sigma+9s_N^2)),\]
	so that
	\[\E|\log Z_N^{\mathfrak{X}}(u)-\E \log Z_N^{\mathfrak{X}}(u)|\le \int_0^\infty  2\exp(-\alpha^2/4(N+9s_N^2))d\alpha=\sq{\pi 8(N\sigma+9s_N^2)}.\]
	In particular, as $s_N=N^{\gamma}$, with $\gamma<1/2$, we for sufficiently large $N$ that  $v_{N,p,q}\le \sq{\pi 9 N\sigma}\le 6\sq{N\sigma}$. It is important to note that while $v_{N,p,q}$ is random quantity depandant on the choice of $\mathfrak{X}\setminus\{x_{p,q}\}$, is is always bounded by $6\sq{N\sigma}$. In particular, as $\gamma>1/4$, we see that $2^{-1}s_N^{-2}v_{N,p,q}/\rho\to 0$. From all these observations, and noting that $\rho\le 1$, we see that Theorem \ref{theorem:perturbation:eriks} may be applied to obtain the following.
	
	\begin{lem}
		\label{lem:perturbation:general-bound}
		Fix $p,q\ge 1$. Then with notation as in Theorem \ref{theorem:perturbation:eriks}, and $N$ sufficiently large,
		\[\E_{\mathfrak{X}}\Delta_N^{\mathfrak{X}}(\ell,f,\b{w}_q,p)\le 24\|f\|_{\infty}\sq{\rho}N^{-\gamma}(1+N^{1/4}\sq{6\sigma}).\]
		In particular
		\[\lim_{N\to\infty}\E_{\mathfrak{X}}\Delta_N^{\mathfrak{X}}(\ell,f,\b{w}_q,p)=0.\]
	\end{lem}
	Putting these together we will be able to obtain the following result used crucially in the lower bound.
	
	\begin{prop}
		\label{prop:perturbation:gg-forcing-perturbation}
		Let $H_N^{\mathfrak{X}}$ be as above and the associated overlap array as $\b{R}_N^{\mathfrak{X}}\in ([-1,1]^{\N\times \N})^{\Om}$. Then we may make, for each $N\ge 1$, a choice $\mathfrak{X}^N$, such that for any subsequence $N_n$, such that $\b{R}_{N_n}^{\mathfrak{X}^{N_n}}$ converges in law to a some random array $\b{R}\in ([-1,1]^{\N\times \N})^{\Om}$, one has that $\b{R}$ satisfies the multi-species Ghirlanda-Guerra identities.
	\end{prop}
	
	\begin{proof}
		We will begin by using Lemma \ref{lem:perturbation:general-bound} to show how we may guarantee that the asymptotic Gibbs measure satisfies the multi-species Ghirlanda-Guerra identities for some value of $\mathfrak{X}$. For each $\ell\ge 1$, let us choose a countable dense subset of continuous functions $f_{\ell,k}:[-1,1]^{\ell\times \ell}\to [-1,1]$. We form the sum
		\[\Delta_{N}^{\mathfrak{X}}=\sum_{p,q,\ell,k\ge 1}\frac{\Delta_{N}^{\mathfrak{X}}(\ell,f_{\ell,k},\b{w}_q,p,x)}{2\ell 2^{\ell+k+p+q}}.\]
		We observe that as $\|f_{\ell,k}\|_{\infty}\le 1$, we have that $\Delta_N^{\mathfrak{X}}(\ell,f_{\ell,k},\b{w}_q,p,x)\le 2\ell $, so that $0\le \Delta_{N}^{\mathfrak{X}}\le 1$. In particular, by Tonelli's theorem and the dominated convergence theorem we have that
		\[\lim_{N\to \infty}\E_{\mathfrak{X}} \Delta_N^{\mathfrak{X}}=\sum_{p,q,\ell,k\ge 1}\lim_{N\to \infty}\frac{\E_{\mathfrak{X}}\Delta_N^{\mathfrak{X}}(\ell,f_{\ell,k},\b{w}_q,p)}{2\ell 2^{\ell+k+p+q}}=0.\]
		We note that by Chebyshev's inequality
		\[\E_{\mathfrak{X}}[I(\Delta_N^{\mathfrak{X}}(x)\le 2\E_{\mathfrak{X}} \Delta_N^{\mathfrak{X}}(x))]\ge 1/2.\]
		Thus for each $N\ge 1$, we may fix a choice of $\mathfrak{X}^{N}$ such that $\Delta_N^{\mathfrak{X}^N}\le 2\E_{\mathfrak{X}}\Delta_N^{\mathfrak{X}}$. In particular, for such a choice, we have that $\lim_{N\to \infty} \Delta_N^{\mathfrak{X}^N}=0$.
		
		Now we first note that the expression $\Delta(\ell,f,\varphi)$ depends continuously on the law of the array $\b{R}$. So if we denote by $\Delta_{\b{R}}(\ell,f,\varphi)$ the Ghirlanda-Guerra difference with respect to $\b{R}$, we have that
		\[\lim_{N\to \infty}|\Delta^{\mathfrak{X}^{N_n}}_N(\ell,f,\varphi)-\Delta_{\b{R}}(\ell,f,\varphi)|=0.\]
		
		In addition, the expressions $\Delta(\ell,f,\varphi)$ from Definition \ref{defin:GG-identities} are continuous in the choice of $(f,\varphi)$ (with respect to the $L^{\infty}$-norm). In particular, we note that a random array $\b{R}$ satisfies the multi-species Ghirlanda-Guerra identities if and only if
		\[\sum_{p,q,\ell,k\ge 1}\frac{\Delta(\ell,f_{\ell,k},\b{w}_q,p)}{2\ell 2^{\ell+k+p+q}}=0.\]
		Thus we see that $\b{R}^M$ satisfies the multi-species Ghirlanda-Guerra identities.
		As we have that $\lim_{N\to \infty} \Delta_N^{\mathfrak{X}^N}=0$, we conclude that $\b{R}$ must satisfy the multi-species Ghirlanda-Guerra identities.
	\end{proof}
	
	\pagebreak
	
	\section{Proof of Theorem 
		\ref{theorem:spherical parisi continuum}\label{section:prop:spherical parisi}}
	
	The focus of this section is to finally complete the proof of Theorem \ref{theorem:spherical parisi continuum}. The first subsection concerns the upper bound, which is an application of the Guerra interpolation method \cite{guerraOG} combined with Talagrand's positivity principle (Theorem \ref{theorem:talagrand's positivity}), generalizing the proof in the $|\Om|=1$ case.
	
	The next subsection serves to establish some preliminaries for the proof of the lower bound, which is more involved. To begin, we employ a variant of the cavity method known as the Aizenman–Sims–Starr scheme, with the modifications to treat the case of spherical models used by \cite{weikuo} where one takes $M$-cavities coordinates and sends $M\to \infty$ after $N\to \infty$. Much of this involves Taylor expanding the Hamiltonian in terms of the cavity coordinates, as well as approximating the sphere by certain product measures.
	
	The final subsection treats the resulting quantity, which may be understood as the difference of two expected averages over a certain Gibbs measure, and may be identified with the functional considered in Section \ref{subsection:Abstract Functional}. We then apply Proposition \ref{prop:perturbation:gg-forcing-perturbation} to force the Gibbs measure to asymptotically satisfy the multi-species Ghirlanda-Guerra identities. We are able to pass to a subsequence and approximate this lower bound with the average over the Gibbs measure instead replaced with the average over some synchronized RPC.
	
	We will also apply results of the prior section with the choice $H_N=\sum_{x\in \Om}H_{\xi_x}(\b{u}(x))$, and in particular recall the perturbation Hamiltonian $h_N$, and let 
	\[H_N^{\mathfrak{X}}(\b{u})=\sum_{x\in \Om}H_{\xi_x}(\b{u}(x))+s_Nh_N(\b{u}).\]
	Using the notation $\omega_{N,D,\b{h}}$ (\ref{def:omega-measure-D}) as before, we will denote the partition function and Gibbs measure with respect to $(H_N^{\mathfrak{X}}$, $\omega_{N,D,\b{h}}$) as $Z^{Sph,\mathfrak{X}}_N$ and $\<*\>^{{\mathfrak{X}}}$.
	
	\subsection{Proof of the Upper Bound}
	
	By Proposition \ref{prop:rpc-intro:infimum identification of B in limit}, to show the upper bound for Theorem \ref{theorem:spherical parisi continuum}, we see that it suffices to show the following result.
	
	\begin{lem}
		\label{lem:upper bound:guerra-top}
		We have that
		\[\limsup_{N\to \infty}\l(\E f_{N,\mathrm{Sph}}-\inf_{r,\vec{\b{q}},\vec{t}}\cal{A}_{N}(\vec{\b{q}},\vec{t})\r)\le 0.\]
	\end{lem}
	\begin{proof}
		By (\ref{eqn:perturbation:perturbation has no average effect}) we see it suffices to show that
		\[\limsup_{N\to \infty}\l(|\Om|^{-1}N^{-1}\E_{\mathfrak{X}}\E\log Z^{Sph,\mathfrak{X}}_N-\inf_{r,\vec{\b{q}},\vec{t}}\cal{A}_{N}(\vec{\b{q}},\vec{t})\r)\le 0.\]
		Now fix a choice of $(\vec{\b{q}},\vec{t})$ and let $G_{\zeta}$ denote the RPC corresponding to the measure $\zeta$ associated to $(\vec{\b{q}},\vec{t})$ by Remarks \ref{remark:associated to measure} and \ref{remark:associated measure and functions} with weights given by $(v_{\alpha})_{\alpha\in \N^r}$. Moreover, let $Z_x(\alpha)$ and $Y(\alpha)$ be the centered Gaussian fields on $\N^r$ by (\ref{def:RPC:Z_x}) and (\ref{def:RPC:Y}) (for existence (\ref{eqn:property of RPC:gaussians})).
		
		We want to apply Proposition \ref{prop:Guerra Gaussian integration}, with the choices
		\[V^{\mathfrak{X}}(\b{u},\alpha)=H_N^{\mathfrak{X}}(\b{u})+\sq{N}Y(\alpha),\; W^{\mathfrak{X}}(\b{u},\alpha)=\sum_{x\in \Om}(\vec{Z}_x(\alpha),\b{u}(x))+s_Nh_N(\b{u}).\]
		For this, define for $0\le t\le 1$ and $(\b{u},\alpha)\in S_N^{\Om}\times \N^r$
		\[H_t^{\mathfrak{X}}(\b{u})=\sq{1-t}V^{\mathfrak{X}}(\b{u},\alpha)+\sq{t}W^{\mathfrak{X}}(\b{u},\alpha),\]
		and denote the Gibbs measure associated to $(H_t^{\mathfrak{X}},\; \omega_{N,D,\b{h}}\otimes G_{\zeta})$ as $\<*\>_t^{\mathfrak{X}}$. Finally denote
		\[\delta(\b{u},\alpha,\b{u}',\alpha')=\E V^{\mathfrak{X}}(\b{u},\alpha)V^{\mathfrak{X}}(\b{u}',\alpha')-\E W^{\mathfrak{X}}(\b{u},\alpha) W^{\mathfrak{X}}(\b{u}',\alpha').\]
		Note that the contribution coming from $h_N(\b{u})$ (the only term dependent on the choice $\mathfrak{X}$) cancel, so that $\delta$ is independent of $\mathfrak{X}$ as the notation suggests. Then Proposition \ref{prop:Guerra Gaussian integration} states that
		\[
		\begin{split}
			&\E\log Z_N^{Sph,\mathfrak{X}}(\b{\xi})+\E \log \sum_{\al\in \N^r}v_\al e^{\sq{N}Y(\alpha)}-\E \log\l(\sum_{\al\in \N^r}v_\al\int_{S_N^\Om} \exp \bigg(V^{\mathfrak{X}}(\b{u},\alpha)\bigg)\omega_{N,D,\b{h}}(d\b{u})\r)=\\
			&\frac{1}{2}\int_0^1\bigg(\E\left\<\delta(\b{u},\alpha,\b{u},\alpha)\right\>_t^{\mathfrak{X}}-\E\left\<\delta(\b{u},\alpha,\b{u}',\alpha')\right\>_t^\mathfrak{X}\bigg)dt.\label{eqn:lem:ignore-1503}
		\end{split}
		\]
		By applying Jensen's inequality to $h_N$, we see in particular that
		\[
		\begin{split}
			&|\Om|^{-1}N^{-1}\E_{\mathfrak{X}}\E\log Z_N^{Sph,\mathfrak{X}}-\cal{A}_{N}(\vec{\b{q}},\vec{t})\le \\
			& 2|\Om|^{-1}N^{-1}s_N^2+\frac{1}{2}\E_{\mathfrak{X}}\int_0^1\bigg(\E\left\<\delta(\b{u},\alpha,\b{u},\alpha)\right\>_t^{\mathfrak{X}}-\E\left\<\delta(\b{u},\alpha,\b{u}',\alpha')\right\>_t^\mathfrak{X}\bigg)dt
		\end{split}
		\]
		As the first term on the right hand side goes to zero, we study the second term. We define for $x\in \Om$,
		\[\rho_x(u,v)=\left(v-u\right)\xi'_x(v)+\xi_x(u)-\xi_x(v).\]
		We then record
		\[
		\begin{split}
			N^{-1}\delta(\b{u},\alpha,\b{u}',\alpha')=&\sum_{x\in \Om}\left(\b{q}^{\alpha\wedge \alpha'}(x)-(\b{u}(x),\b{u}'(x))_N\right)\xi'_x(\b{q}^{\alpha\wedge\alpha'}(x))\\
			&+\xi_x((\b{u}(x),\b{u}'(x))_N)-\xi_x(\b{q}^{\alpha\wedge\alpha'})+\xi_x(0)\\
			=&\sum_{x\in \Om}\l(\rho_x(\b{q}^{\alpha\wedge\alpha'}(x),(\b{u}(x),\b{u}'(x))_N)+\xi_x(0)\r).
		\end{split}
		\]
		Thus we see that the first term on the right-hand side of (\ref{eqn:lem:ignore-1503}) coincides with $\frac{1}{2}\sum_{x\in \Om}\xi_x(0)$, which cancels a similar term in the second term. In particular, we see that
		\[
		\begin{split}
			\frac{1}{2N}\E_{\mathfrak{X}}\int_0^1\bigg(\E\left\<\delta(\b{u},\alpha,\b{u},\alpha)\right\>_t^{\mathfrak{X}}-\E\left\<\delta(\b{u},\alpha,\b{u}',\alpha')\right\>_t^\mathfrak{X}\bigg)dt=\\
			-\frac{1}{2}\int_0^1 \sum_{x\in \Om} \E_{\mathfrak{X}}\E\l\<\rho_x((\b{u}(x),\b{u}'(x))_N,\b{q}^{\alpha\wedge \alpha'}(x))\r\>_t^{\mathfrak{X}}dt.\label{eqn:ignore-321}
		\end{split}
		\]
		By Taylor's theorem, and convexity of $\xi_x$ on $[0,1]$, we see that for $u,v\in [0,1]$, $\rho_x(u,v)\ge 0$. In particular, we may conclude that there is $C$ (independent of $\epsilon$) such that $u,v\in [-\epsilon,1]$ and $x\in \Om$, we have that $\rho_x(u,v)\ge -C\epsilon$. Potentially enlarging $C$ we may also guarantee that for $u,v\in [-1,1]$, $\rho_x(u,v)\ge -C$. With these facts noted, we see that (\ref{eqn:ignore-321}) is less than
		\[\frac{C}{2}\int_0^1 \bigg(\epsilon+\E_{\mathfrak{X}}\<I((\b{u}(x),\b{u}'(x))_N\le -\epsilon)\>_t^{\mathfrak{X}}\bigg)dt\]
		Applying Theorem \ref{theorem:talagrand's positivity} and taking $\epsilon\to 0$ yields the desired result.
	\end{proof}
	
	\subsection{Proof of the Lower Bound Part I: The Aizenman-Sims-Star Scheme\label{section:lowerbound:cavity}}
	
	To being with the lower bound, note that for any sequence of positive numbers $z_{N}>0$ and $M\ge 1$, we have the following the cavity inequality (see \cite{weikuo})
	\[\liminf_{N\to \infty} N^{-1} \log z_N\ge M^{-1}\liminf_{N\to \infty} \l(\log z_{N+M}- \log z_{N}\r).\]
	Using this and the above results we see that
	\[\liminf_{N\to \infty}N^{-1}\E \log \ZNxi\ge \limsup_{M\to \infty}\liminf_{N\to \infty}\sup_{\mathfrak{X}\in [1,2]^{\N\times \N}}M^{-1}\left(\E\log Z_{N+M}^{Sph,\mathfrak{X}}-\E\log Z_N^{Sph,\mathfrak{X}}\right).\label{eqn:lower bound:enhanced-cavity method}\]
	Next, we will need to give a cavity decomposition for $H_{N+M}$ and $H_N$, which will follow from the $|\Om|=1$ case. Specifically, let us define for each $x\in \Om$, the following centered Gaussian functions on $S_N$, specified giving their covariance functions: For $u,u'\in S_N$
	\[
	\begin{split}
		\E[Z_{N,x}(u)Z_{N,x}(u')]&=\xi_x'\left((u,u')_N\right),\label{def:eqn:ZN}\\
		\E[Y_{N,x}(u)Y_{N,x}(u')]&=\theta_x((u,u')_N),\\
		\E[H_{N,M,x}(u)H_{N,M,x}(u')]&=(N+M)\left(\xi_x\left((u,u')_{N+M}\right)-\xi_x(0)\right).
	\end{split}
	\]
	We define the core Hamiltonian as
	\[\cal{H}_{N,M}(\b{u})=\frac{1}{2}\sum_{x,y\in \Om}D_{x,y}(\b{u}(x),\b{u}(y))+\sum_{x\in \Om}H_{N,M,x}(\b{u}(x))+\sq{N}\sum_{x\in \Om}\b{h}(x)\b{u}_1(x),\]
	and let $\cal{H}_{N,M}^{\mathfrak{X}}(\b{u})=\cal{H}_{N,M}(\b{u})+s_Nh_N(\b{u})$.
	We first define a function $\cal{Y}_N$ on $S_N^\Om$ by
	\[\cal{Y}_N(\b{u})=\sum_{x\in \Om}Y_{N,x}(\b{u}(x)),\]
	and then define an associated partition function on $S_N^\Om$ by
	\[\hat{Z}_{N,M}^{\mathfrak{X}}=\int_{S_{N}^{\Om}}\exp (-\cal{H}_{N,M}^{\mathfrak{X}}(\b{u})-\sq{M}\cal{Y}_N(\b{u}))\omega(d\b{u}).\]
	Next, let us further consider a sequence of $N$ i.i.d copies of $Z_{N,x}$ which we denote $\vec{Z}_{N,x}\in \R^N$. We then define a function on $(\b{u},\b{v})\in S_{N+M}^{\Om}$ (with $\b{u}\in (\R^{N})^{\Om}$ and $\b{v}\in (\R^M)^{\Om}$) by
	\[\cal{Z}_{N,M}(\b{u},\b{v})=\frac{1}{2}\sum_{x,y\in \Om}D_{x,y}(\b{v}(x),\b{v}(y))+\sum_{x\in \Om}(\vec{Z}_{N,x}(\b{u}(x)),\b{v}(x))+\sq{M}\sum_{x\in \Om}\b{h}(x)\b{v}_1(x).\]
	Then for any $\delta>0$, we define the set
	\[S_{N,M}(\delta)=\{(u,v)\in \R^N\times \R^M \cap S_{N+M}: \|v\|_M^2\le 1+\delta\}=S_{N+M}\cap \R^M\times S_{M}(0,1+\delta),\]
	and the partition function
	\[\bar{Z}_{N,M}^{\mathfrak{X},\delta}=\int_{S_{N,M}(\delta)^{\Om} } \exp(-\cal{H}_{N,M}^{\mathfrak{X}}(\b{u})-\cal{Z}_{N,M}(\b{u},\b{v}))\omega(d\b{u}d\b{v}).\]
	
	Our first result shows that for our lower bound, it will essentially suffice to deal with these partition functions instead.
	
	\begin{lem}
		\label{lem:lowerbound-guerra}
		For any $\delta>0$ we have that
		\[
		\begin{split}
			\E\log Z_{N+M}^{Sph,\mathfrak{X}}-\E\log \bar{Z}_{N,M}^{\mathfrak{X}}&\ge -\left(\frac{M^2\delta}{N}+\frac{M}{N}\right)\sum_{x\in \Om}\xi_x''(1),\\
			\bigg|\E\log Z_{N}^{Sph,\mathfrak{X}}-\E\log \hat{Z}_{N,M}^{\mathfrak{X}}\bigg|&\le \frac{M^2}{N}\sum_{x\in \Om}\xi_x''(1).
		\end{split}
		\]
	\end{lem}
	\begin{proof}
		We note that by applying a suitable rotation, we may replace the external field term in $\cal{H}_{N+M}^{\mathfrak{X}}$ as
		\[\sq{N+M}\sum_{x\in \Om}\b{h}(x)\b{u}_1(x)\mapsto \sq{N}\sum_{x\in \Om}\b{h}(x)\b{u}_1(x)+\sq{M}\sum_{x\in \Om}\b{h}(x)\b{v}_1(x),\]
		without affecting the value of $\E \log Z_{N+M}^{Sph,\mathfrak{X}}$. After this replacement is done, we denote by $Z_{N,M}^{Sph,\mathfrak{X},\delta}$ the restriction of the partition function $Z_{N+M}^{Sph,\mathfrak{X}}$ to the subset $S_{N,M}(\delta)^{\Om}\subset S_{N+M}^{\Om}$. As $Z_{N+M}^{Sph,\mathfrak{X}}\ge Z_{N,M}^{Sph,\mathfrak{X},\delta}$, we see to establish the first claim, it suffices to show that 
		\[|\E\log Z_{N,M}^{Sph,\mathfrak{X},\delta}-\E\log \bar{Z}_{N,M}^{\mathfrak{X},\delta}|\le \left(\frac{M^2\delta}{N}+\frac{M}{N}\right)\sum_{x\in \Om}\xi_x''(1).\label{eqn:ignore-2398429385}\]
		For this, note that the deterministic components of both Hamiltonians coincide. For the centered Gaussian components, we apply Corollary \ref{corr:Guerra-bound} to see that
		\[
		\begin{split}
			|\E\log Z_{N+M}^{Sph,\mathfrak{X},\delta}-\E\log \bar{Z}_{N,M}^{\mathfrak{X}}|\le&\sum_{x\in \Om}\sup_{(u,v),(u',v')\in S_{N,M}(\delta)}\bigg|(N+M)\xi_x((u,u')_{N+M}+(v,v')_{N+M})\\
			&-(N+M)\xi_x((u,u')_{N+M})-(v,v')\xi_x'((u,u')_N)\bigg|.\label{eqn:ignore-238948293}
		\end{split}
		\]
		It is easily checked using Taylor's theorem that for $(u,v), (u',v')\in S_{N+M}$
		\[
		\begin{split}
			|(N+M)\xi_x((u,u')_{N+M}+(v,v')_{N+M})-(N+M)\xi_x((u,u')_{N+M})&\\-
			(v,v')\xi_x'((u,u')_{N+M})|&\le \frac{(v,v')^2}{2(N+M)}\xi_x''(1).
		\end{split}
		\]
		Similarly one may check that
		\[|\xi_x'((u,u')_{N+M})-\xi_x'((u,u')_N)|\le \frac{M}{N}|(u,u')_{N+M}|\xi_x''(1)\le \frac{M}{N}\xi_x''(1).\]
		Combining these observations with (\ref{eqn:ignore-238948293}) yields (\ref{eqn:ignore-2398429385}), establishing the first claim. For the second claim, we first note that one may routinely show that for $u,u'\in S_N$,
		\[|N\xi_x((u,u')_N)-(N+M)\xi_x \left((u,u')_{N+M}\right)-M\theta_x((u,u')_N)-M\xi_x(0)|\le \frac{M^2}{N} \xi_x''(1).\label{eqn:ignore-23984729375}\]
		Noting that making the changes $\xi_x(r)\mapsto \xi_x(r)-\xi_x(0)$ does not affect the value of $\E\log Z_{N}^{Sph,\mathfrak{X}}$, we see that (\ref{eqn:ignore-23984729375}) and Corollary \ref{corr:Guerra-bound} demonstrates the second claim.
	\end{proof}
	
	For our next step, we want to relate the measure spaces $(S_{N+M}^{\Om},\omega_{N+M})$ and $(S_N^{\Om} \times (S_M)^{\Om}, c_{N,M}^{|\Om|}\omega_N^{\otimes \Om} \otimes \gamma_M^{\otimes \Om})$ where $\gamma_M$ denotes the standard Gaussian measure on $\R^M$ and $c_{N,M}$ denotes the constant
	\[c_{N,M}=\frac{\vol(S_{N+M})}{\vol(S_{N})}.\]  
	This will be adapted from \cite{weikuo} in the $|\Om|=1$ case, and will allow us to treat $S_{N+M}$ as a product measure when $N$ is large. For this, we recall from \cite{weikuo} the set
	\[A_{M,N}=\prod_{j=1}^{M}[-\sq{(M+N+1-j)},\sq{(M+N+1-j)}],\]
	Next, we define functions of $v\in A_{M,N}$ by
	\[
	\begin{split}
		a_{\ell}(v)&=\prod_{j=1}^{\ell-1}\sq{1+\frac{1-v^2_j}{M+N-j}},\;\;2\le \ell\le M+1,\;\;\ a_1(v)=1.\\
		\phi_{M}(v)&=(v_1a_1(v),\dots,v_Ma_M(v)),\;\; \psi_{M}(v)=a_{M+1}(v).
	\end{split}
	\]
	Finally, we define a measure on $\gamma_{N,M}$ on $A_{M,N}$ by 
	\[
	\begin{split}
		\gamma_{N,M}(dv)&=b_{M,N}\prod_{j=1}^{M}\left(1-\frac{v_j^2}{M+N+1-j}\right)^{\frac{M+N-j-2}{2}}dv,\\
		b_{M,N}&=\prod_{j=1}^{M}\frac{\Gamma((N+j)/2)}{\Gamma((N+j-1)/2)\sq{\pi(N+j)}}.
	\end{split}
	\]
	With this notation set, and writing elements of $w\in \R^{N+M}$ as $w=(u,v)$ with $u\in \R^N$ and $v\in \R^M$, we recall the following result.
	\begin{lem}[Lemma 3,\cite{weikuo}]
		For a non-negative function $f$ on $S_{N+M}$ we have that
		\[\int_{S_{N+M}}f(w)\omega(dw)=c_{N,M}\int_{A_{M,N}}\left(\int_{S_N}f(\psi_{M}(v)u,\phi_{M}(v))\omega(du)\right)\gamma_{M,N}(dv).\]
	\end{lem}
	A standard argument then shows the following.
	\begin{lem}
		\label{lem:lowerbound:disintegration}
		For a non-negative function $f$ on $S_{N+M}^{\Om}$ we have that
		\[\int_{S_{N+M}^{\Om}}f(w)\omega(d\b{w})=c_{N,M}^{|\Om|}\int_{A_{M,N}^{\Om}}\left(\int_{S_N^\Om}f(\b{\psi}_{M}(\b{v})\b{u},\b{\phi}_{M}(\b{v}))\omega(d\b{u})\right)\gamma_{M,N}(d\b{v}),\]
		where here $\b{\psi}_{M}(\b{v})(x)=\psi_{M}(\b{v}(x))$, and similarly for $\b{\phi}_M(\b{v})$, $(\b{\psi}_{M}(\b{v})\b{u})(x)=\psi_{M}(\b{v}(x))\b{u}(x)$ and finally
		\[\gamma_{M,N}(d\b{v})=\prod_{x\in \Om}\gamma_{M,N}(\b{v}(x))\]
	\end{lem}
	
	Now we set $A_{N,M}(\delta):=\phi_M^{-1}\left(S_{M}(0,1+\delta)\right)\cap A_{N,M}(\delta)$. Then by Lemma \ref{lem:lowerbound:disintegration} we have
	\[\bar{Z}_{N,M}^{\mathfrak{X},\delta}=c_{N,M}^{|\Om|}\int_{S_N^{\Om} \times  A_{N,M}(\delta)^{\Om}} \exp(-\cal{H}_{N,M}^{\mathfrak{X}}(\b{\psi}_{M}(\b{v})\b{u}))-\cal{Z}_{N,M}(\b{\psi}(\b{v})\b{u},\b{\phi}(\b{v}))\omega(d\b{u})\gamma_{M,N}(d\b{v}).\]
	In the next lemma, we compare this to the following partition function
	\[\bar{Z}_{N,M}^{\mathfrak{X},\delta,I}=c_{N,M}^{|\Om|}\int_{S_N^{\Om} \times (S_M^{\delta})^{\Om}} \exp(-\cal{H}_{N,M}^{\mathfrak{X}}(\b{u})-\cal{Z}_{N,M}(\b{u},\b{v}))\omega(d\b{u})\gamma_{M,N}(d\b{v}).\]
	\begin{lem}
		\label{lem:lowerbound:removing Fn}
		Let us choose $M\ge 10\delta>0$. Then we have that 
		\[\liminf_{N\to \infty}M^{-1}\left(\E \log \bar{Z}_{N,M}^{\mathfrak{X},2\delta}-\E \log Z_{N,M}^{\mathfrak{X},\delta}\right)\ge -4\delta\sum_{x\in \Om}\left(\sum_{y\in \Om}|D_{xy}|+|\b{h}(x)|+\xi'_x(1)\right).\]
	\end{lem}
	\begin{proof}
		To begin we will need an estimate. By construction, for any $v\in A_{M,N}$
		\[\|\phi_M(v)\|^2+N\psi_M(v)^2=N+M.\label{eqn:ignore-1020203}\]
		Note moreover that $1\le a_{\ell}(v)\le a_{M+1}(v)$, so that $\|v\|^2\le \|\phi_M(v)\|^2\le \psi_M(v)^2\|v\|^2$. Combined with (\ref{eqn:ignore-1020203}), this implies that if $v\in S^\delta_M\cap A_{M,N}$
		\[0\le N(\psi_M(v)^2-1)\le \psi_M(v)^2\|v\|^2-M\le \psi_M(v)^2M(\delta+1)-M.\]
		Thus if we assume that $M/\delta,N/M>10$, one derives routinely that
		\[0\le N\left(\psi_M(v)^2-1\right)\le 2\delta M.\]
		From this we conclude that for $v,v'\in S_{N}^\delta$, we have that
		\[
		\begin{split}
			N|(\phi_M(v),\phi_M(v'))-(v,v')|&\le 4\delta N^{-1}M^2,\;\;\; N|\psi_M(v')\psi_M(v)-1|\le 2\delta N^{-1}M,\label{eqn:ignore-luna-11}\\
			N|\psi_M(v)-1|&\le 2\delta N^{-1}M.
		\end{split}
		\]
		
		Now we will fix $M,\delta>0$, assume that $N\ge 10M$. We now restrict this integration in $\bar{Z}_{N,M}^{\mathfrak{X},2\delta}$. For this note by the above estimate, for sufficiently large $N$ we have that $\phi_M(S_{M}^\delta)\subseteq S_{M}(0,1+2\delta)$. Moreover, as $100 M\delta\le N$, we have that $S_M^\delta\subseteq A_{N,M}$. Thus if we define 
		\[\bar{Z}_{N,M}^{\mathfrak{X},\delta,res}=c_{N,M}^{|\Om|}\int_{S_N^{\Om} \times S_M(\delta)^{\Om}} \exp(-\cal{H}_{N,M}^{\mathfrak{X}}(\b{\psi}_{M}(\b{v})\b{u}))-\cal{Z}_{N,M}(\b{\psi}(\b{v})\b{u},\b{\phi}(\b{v}))\omega(d\b{u})\gamma_{M,N}(d\b{v}),\]
		it suffices for us to show that
		\[\liminf_{N\to \infty}M^{-1}\left(\E \log \bar{Z}_{N,M}^{\mathfrak{X},\delta,res}-\E \log Z_{N,M}^{\mathfrak{X},\delta}\right)\ge -4\delta\sum_{x\in \Om}\left(\sum_{y\in \Om}|D_{xy}|+|\b{h}(x)|+\xi'_x(1)\right) \label{eqn:ignore-123934}\]
		
		Returning to (\ref{eqn:ignore-123934}), we note that to bound the difference, we must we must deal with the effect of both the deterministic and random components of
		\[\cal{H}_{N,M}^{\mathfrak{X}}(\b{\psi}_{M}(\b{v})\b{u})+\cal{Z}_{N,M}(\b{\psi}_{M}(\b{v})\b{u},\b{\phi}_{M}(\b{v})) \text{ and }\cal{H}_{N,M}^{\mathfrak{X}}(\b{u})+\cal{Z}_{N,M}(\b{u},\b{v}).\label{eqn:lowerbound:compare-in-dis}\]
		To begin we will deal with the deterministic pieces, for which we just bound the differences. For this, note that for $N$ sufficiently large and $(\b{u},\b{v}),(\b{u}',\b{v}')\in (S_{N}\times S_M^{\delta})^{\Om}$
		\[
		\begin{split}
			\bigg|\sum_{x,y\in \Om}D_{x,y}(\psi_{M}(\b{v}(x))\b{u}(x),\psi_{M}(\b{v'}(x))\b{u'}(x))-\sum_{x,y\in \Om}D_{x,y}(\b{u}(x),\b{u'}(x))\bigg|\le \\
			\sum_{x,y\in \Om}|D_{x,y}||(\psi_{M}(\b{v}(x))\psi_{M}(\b{v'}(x))-1)(\b{u}(x),\b{u'}(x))|\le 2\delta M\sum_{x,y\in \Om}|D_{x,y}|.
		\end{split}
		\]
		A similar argument shows that
		\[\bigg|\sum_{x,y\in \Om}D_{x,y}(\phi_M(\b{v}(x)),\phi_{M}(\b{v'}(x)))-\sum_{x,y\in \Om}D_{x,y}(\b{v}(x),\b{v'}(x))\bigg|\le 4N^{-1}M^2\delta \sum_{x,y\in \Om}|D_{x,y}|.\]
		Finally, to contend with the external field, we note that
		\[\l|\sq{N}\b{h}(x)\b{\psi}_M(\b{v}(x))\b{u}_1(x)-\sq{N}\b{h}(x)\b{u}_1(x)\r|\le |\b{h}(x)||\le N|\b{\psi}_M(\b{v}(x))-1|\le 2\delta M|\b{h}(x)|.\]
		This completes the analysis of the deterministic terms. The random components will be dealt with by employing Corollary \ref{corr:Guerra-bound}. To do so, we note that the covariance of $\cal{H}_{N,M}(\b{u})+\cal{Z}_{N,M}(\b{u},\b{v})$ at points $(\b{u},\b{v})$ and $(\b{u}',\b{v}')$ is given by
		\[\sum_{x\in \Om}\bigg((N+M)\xi_x((\b{u}(x),\b{u}'(x))_{N+M})+\xi'_x((\b{u}(x),\b{u}'(x))_N)(\b{v}(x),\b{v}'(x))\bigg).\]
		In particular, the difference between the variance functions of the Hamiltonians in (\ref{eqn:lowerbound:compare-in-dis}),  excluding the perturbation terms, may be bounded above by
		\[\sum_{x\in \Om}\left(3\delta M\xi_x'(1)+3\delta N^{-1}M(1+\delta)+4M^2N^{-1}\xi_x''(1)\delta(1+\delta)\right).\]
		Finally for the perturbation term, we recall that the covariance of $h_N$ is given by (\ref{eqn:lower bound:covariance of perturbation term}). We note that for any $\b{w}\in [0,1]^{\Om}$, and $(\b{u},\b{v})\in S_N^{\Om}\times (S_N^\delta)^{\Om}$
		\[|\b{R}^{\b{w}}_{N+M}(\b{\psi}_{M}(\b{v})\b{u},\b{\phi}_M(\b{v}))-\b{R}^{\b{w}}_{N+M}(\b{u},\b{v})|\le \left(\frac{1}{|\Om|}\sum_{x\in \Om}\b{w}(x)\right)(4M^2N^{-1}(1+\delta)+2MN^{-1}\delta).\]
		Using this, it is easily shown that the perturbation term is negligible as well.
	\end{proof}
	
	We next make the replacement $\gamma_{M,N}\mapsto \gamma_M$ and introduce 	\[\bar{Z}_{N,M}^{\mathfrak{X},\delta,II}=c_{N,M}^{|\Om|}\int_{S_N^{\Om} \times (S_M^{\delta})^{\Om}} \exp(-\cal{H}_{N,M}^{\mathfrak{X}}(\b{u})-\cal{Z}_{N,M}(\b{u},\b{v}))\omega(d\b{u})\gamma_{M}(d\b{v}).\]
	We then have the following result.
	\begin{lem}
		For any $M\ge 1$ and $\delta>0$ we have that
		\[\liminf_{N\to \infty} \E \log \frac{\bar{Z}_{N,M}^{\mathfrak{X},\delta,II}}{\bar{Z}_{N,M}^{\mathfrak{X},\delta,I}}\ge 0.\]
	\end{lem}
	\begin{proof}
		The only change between these integrals is the replacement $\gamma_{M,N}\mapsto \gamma_M$. Both of these measures have a density, and are products of the measures in the $|\Om|=1$ case. As Lemma 5 of \cite{weikuo} shows this ratio is bounded below by a constant going to $1$, from which this result immediately follows.
	\end{proof}
	
	Now finally, we will relate $\bar{Z}_{N,M}^{\mathfrak{X},\delta,II}$ to our final product partition function
	\[\tilde{Z}_{N,M}^{\mathfrak{X}}=\int_{S_N^{\Om}\times S_M^\Om}\exp(-\cal{H}_{N,M}^{\mathfrak{X}}(\b{u})-\cal{Z}_{N,M}(\b{u},\b{v})\omega(d\b{u})\omega(d\b{v}).\]
	\begin{lem}
		For fixed $M\ge 1$ and $\delta>0$, we have that
		\[\E \log \bar{Z}_{N,M}^{\mathfrak{X},\delta,II}-\E \log \tilde{Z}_{N,M}^{\mathfrak{X}}\ge -\frac{M\delta}{2}\sum_{x,y\in \Om}|D_{x,y}|+ |\Om|\log \frac{c_{N,M} \gamma_M(S^\delta_M)}{\vol(S_M)},\]
		
	\end{lem}
	\begin{proof}
		Consider the mapping $\hat{\b{v}}(x)=\|\b{v}(x)\|^{-1}\b{v}(x)$ defined on $(S_M^\delta)^{\Om}$. We will try to replacement the factor of $\cal{Z}_{N,M}(\b{u},\b{v})$ with $\cal{Z}_{N,M}(\b{u},\hat{\b{v}})$ in $\bar{Z}_{N,M}^{\mathfrak{X},\delta,II}$. For the first term, note that for $\b{v},\b{v}'\in (S_M^\delta)^{\Om}$, we have that
		\[\bigg|\sum_{x,y\in \Om}D_{x,y}(\b{v},\b{v}')-\sum_{x,y\in \Om}D_{x,y}(\hat{\b{v}},\hat{\b{v}}')\bigg|\le M\delta \sum_{x,y\in \Om}|D_{x,y}|.\]
		The remaining terms in $\cal{Z}_{N,M}(\b{u},\b{v})$ are linear in $\b{v}$. To deal with them we recall that for any vector $v\in \R^M$, the function
		\[\lam \mapsto \int_{S_M}\exp(\lam(v,u))\omega(du)=\int_{S_M}\cosh(\lam(v,u))\omega(du),\]
		is increasing in $\lam>0$ (see Lemma 1 of \cite{weikuo}). In particular, we see that $\log(Z_{N,M}^{\mathfrak{X},\delta})$ is bounded below by
		\[-\frac{M\delta}{2}\sum_{x,y\in \Om}|D_{x,y}|+\log \left(c_{N,M}^{|\Om|}\int_{S_N^\Om\times (S_M^{\delta})^\Om}\exp(-\cal{H}_{N,M}^{\mathfrak{X}}(\b{u})-\cal{Z}_{N,M}(\b{u},\hat{\b{v}}))\omega(d\b{u})\gamma_M(d\b{v})\right).\]
		By the co-area formula the final term may be written as
		\[|\Om|\log \frac{c_{N,M} \gamma_M(S^\delta_M)}{\vol(S_M)}+\log(\tilde{Z}_{N,M}^{\mathfrak{X}}).\]
		Taking expectations then completes the proof.
	\end{proof}
	By Stirling's approximation we see that \[\lim_{M\to \infty}M^{-1}\log\gamma_M(S_M^{\delta})=0, \;\; \lim_{M\to \infty}\lim_{N\to \infty}M^{-1}\log c_{N,M}/\vol(S_M)=0.\]
	In addition, all of the above bounds are independent of $(x_{p,q})_{p,q\ge 1}\in [1,2]^{\N\times \N}$, and go to zero if the limits $N\to \infty$, $M\to \infty$, and then $\delta\to 0$ are taken. Thus in total, we have shown the following result.
	\begin{lem}
		\label{lem:lowerbound:main-lem}
		We have that 
		\[\liminf_{N\to \infty}|\Om|^{-1}N^{-1}\E \log Z_{N,Sph}\ge \limsup_{M\to \infty}\liminf_{N\to \infty}\sup_{x\in [1,2]^{\N\times \N}}|\Om|^{-1}M^{-1}\bigg(\E \log \tilde{Z}_{N,M}^{\mathfrak{X}}-\E\log \hat{Z}_{N,M}^{\mathfrak{X}}\bigg).\label{eqn:lem:lowerbound:main-lem}\]
	\end{lem}
	
	\subsection{Proof of the Lower Bound Part II: Synchronization and Perturbation}
	
	The next step is to understand the quantity
	\[|\Om|^{-1}M^{-1}\bigg(\E \log \tilde{Z}_{N,M}^{\mathfrak{X}}-\E\log \hat{Z}_{N,M}^{\mathfrak{X}}\bigg).\label{eqn:lowerbound:2:first bound}\]
	Let us denote the Gibbs measure on $S_N^{\Om}$ with respect to $(\cal{H}_{N,M}^{\mathfrak{X}},\omega_{N})$  as $\<*\>_{N,M}^{\mathfrak{X}}$. Subtracting a factor of the partition function of $(\cal{H}_{N,M}^{\mathfrak{X}},\omega_{N})$ from both terms shows that (\ref{eqn:lowerbound:2:first bound}) is equal to
	\[|\Om|^{-1}M^{-1}\l(\E \log \l\<\int_{S_M^{\Om}}\exp(\cal{Z}_{N,M}(\b{u},\b{v}))\omega(d\b{v})\r\>_{N,M}^{\mathfrak{X}}-\E\log \l\<\exp(\sqrt{M}\cal{Y}_N(\b{u}))\r\>_{N,M}^{\mathfrak{X}}\r).\label{eqn:lower bound:II:main-term}\]
	
	Thus the analysis of the right-hand side of Lemma \ref{lem:lowerbound:main-lem} comes down to understanding the Gibbs measure $\<*\>_{N,M}^{\mathfrak{X}}$.
	
	This is where the importance of the perturbation term comes in, as we may apply Proposition \ref{prop:perturbation:gg-forcing-perturbation} to make, a sequences of choices of $\mathfrak{X}^{N,M}$ which force the results of Section \ref{section:RPC-gg-iden} to hold. We summarize these results in the following corollary.
	
	\begin{corr}
		\label{corr:lowerbound:mmgg convergence}
		Let us denote by $\b{R}_{N,M}^{\mathfrak{X}}\in ([-1,1]^{\N\times \N})^{\Om}$ the infinite overlap array associated to an infinite sequence of i.i.d samples from the Gibbs measure $\<*\>_{N,M}^{\mathfrak{X}}$. There is a choice of $\mathfrak{X}^{N,M}$, for each $N\ge M$, such that for any subsequence $N_n$ of $N$, such that $\b{R}_{N_n,M}^{\mathfrak{X}^{N_n,M}}$ converges in law to some $\b{R}$, one has that $\b{R}$ satisfies the multi-species Ghirlanda-Guerra identities.
	\end{corr}
	
	We are now ready to establish the lower bound, and complete the proof of Theorem \ref{theorem:spherical parisi continuum}.
	\begin{proof}[Proof of Theorem \ref{theorem:spherical parisi continuum}]
		As noted above, the upper bound follows from Proposition \ref{prop:rpc-intro:infimum identification of B in limit} and Lemma \ref{lem:upper bound:guerra-top} above.
		
		To establish the lower bound, let us take the set-up of Corollary \ref{corr:lowerbound:mmgg convergence}. Denote the quantity (\ref{eqn:lower bound:II:main-term}) with the choice $\cal{X}^{N,M}$ coming from Corollary \ref{corr:lowerbound:mmgg convergence} as $A_{N,M}$. Then the above results show that
		\[\liminf_{N\to \infty}|\Om|^{-1}N^{-1}\E \log \ZNxi\ge \limsup_{M\to \infty} \liminf_{N\to \infty} A_{N,M}.\]
		
		We now recast $A_{N,M}$ in terms of the abstract functional introduced in Section \ref{subsection:Abstract Functional}. For this, fix a family of isometric embeddings $(\R^N)^{\Om}\subseteq H$, which for notational convenience we will regard as genuine inclusions. Denote the orthogonal projection onto this subspace as $\Pi_N$. We define a continuous function $\b{R}_N^H:H\times H\to [-1,1]^{\Om}$ by letting
		$\b{R}_N'(\b{u},\b{u}')(x)=\max(\min((\Pi_N\b{u}(x),\Pi_N\b{u}'(x))_N,1),-1)$. Importantly, this extends the function $\b{R}_N:S_N^{\Om}\times S_N^{\Om}\to [-1,1]^{\Om}$ considered above. Let us write the induced extension of $\<*\>^{\mathfrak{X}}_{N,M}$ from $S_N^{\Om}$ to $H$ as $G^H_{N,M}$. It is then clear that
		\[A_{N,M}=\Gamma^M_{\omega_{M,D,\b{h}},\b{\xi}}(G^H_{N,M},\b{R}^H_N).\]
		Now choose $M\ge 1$ and a subsequence $N_n$ of $N$, such that
		\[\liminf_{N\to \infty}A_{N,M}=\lim_{n\to \infty}A_{N_n,M},\]
		and such that $\b{R}_{N_n,M}^{\mathfrak{X}^{N_n,M}}$ converges in some law to some $\b{R}$. Let $R\in [-1,1]^{\N\times \N}$ denote the array given by $R_{\ell,\ell'}=\frac{1}{|\Om|}\sum_{x\in \Om}\b{R}_{\ell,\ell'}$. By Theorem \ref{theorem:perturbation:panchenko synch}, we see that there is a collection of non-decreasing continuous functions $\Phi_x:[0,1]\to [0,1]$, such that a.s. for all $\ell,\ell'\ge 1$,
		\[\b{R}_{\ell,\ell'}(x)=\Phi_x([R]_{\ell,\ell'}).\]
		Next, let us choose as a sequence of discrete measures, $\zeta_n$, supported on $[0,1]$ such that $\zeta_{n}\To R_{12}$ in law. By applying Theorems \ref{theorem:perturbation:dovbysh} and \ref{theorem:perturbation:gg gives RPC}, we see that if we denote by $R_{\zeta_n}$ for the array sampled from the RPC corresponding to $\zeta_n$, then as $n\to \infty$ we have in law \[[R_{\zeta_n}]_{\ell\neq \ell'\ge 1}\To R_{\ell\neq \ell'\ge 1}.\]
		Moreover, note that we must have that $R_{\ell,\ell}=1$ as this is true for each finite $N$, which combined with $[R_{\zeta_n}]_{\ell,\ell}=1$, implies that in law $R_{\zeta_n}\To R^M$ as $n\to \infty$. In particular, both $\b{\Phi}(R_{\zeta_n})$ and $\b{R}_{N_n,M}^{\mathfrak{X}^{N_n,M}}$ converge to the same law as $n\to \infty$. In particular, extending $\b{\Phi}$ to a function $\b{R}^{\b{\Phi}}:H\times H\to [-1,1]^{\Om}$ as in (\ref{eqn:rpc-1:R^Phi}) we see by Lemma \ref{lem:perturbation:approximation of Parisi formula} that
		\[\lim_{n\to \infty} |\Gamma^M_{\omega_{M,D,\b{h}},\b{\xi}}(G_{N_n,M}^H,\b{R}_N^{H})-\Gamma^M_{\omega_{M,D,\b{h}},\b{\xi}}(G_{\zeta_n},\b{R}^{\b{\Phi}})|=0.\]
		Recalling the definition of $\cal{A}_{M}$, we thus see that
		\[\liminf_{N\to \infty}A_{N,M}\ge \inf_{r,\vec{\b{q}},\vec{t}}\cal{A}_M(\vec{\b{q}},\vec{t}).\]
		Combined with Proposition \ref{prop:rpc-intro:infimum identification of B in limit}, this establishes the lower bound and so completes the proof.
	\end{proof}
	
	With this, we have completed the proof of Theorem \ref{theorem:spherical parisi continuum}, and by extension, the proofs of Theorems \ref{theorem:intro:bad main euclidean} and \ref{theorem:intro:bad parisi spherical} as well.
	
	\pagebreak
	
	\begin{appendices}
		
	\section{Concentration and Interpolation Results for Gaussian Functions \label{section:appendix:gaussian}}
		
		This appendix consists of some results about Gaussian random functions we need throughout the text. The purpose is to recall and derive these in the generality required above, as well as to provide a convenient reference for the main text. To begin, we recall the Borell-Tsirelson-Ibragimov-Sudakov (commonly abbreviated to Borell-TIS) inequality (see Theorem 2.1.1 of \cite{taylor})
		
		\begin{theorem}
			\label{theorem:Borell-TIS}[Borell-TIS]
			Let $h$ be some a.s. bounded, centered continuous Gaussian function on some topological space $M$, such that $\sup_{s\in M}\E[h(s)^2]\le \sigma^2$. Then $\sup_{s\in M}h(s)$ possesses a finite first moment, and for any $c>0$
			\[\P\left( |\sup_{s\in M}h(s)-\E \sup_{s\in M}h(s)|>c \right)\le 2\exp(-c^2/4\sigma^2).\]
			Identical statements hold for $\sup_{s\in M}|h(s)|$.
		\end{theorem}
		
		Our second claim follows simply from the concentration of Lipshitz statistics for Gaussians, and is a generalization of Theorem 1.2 of \cite{panchenko}, which treats the case where $M$ is a finite set. It may also be thought of as a positive temperature version of Theorem \ref{theorem:Borell-TIS}.
		\begin{prop}
			\label{prop:Borell-TIS-free-energy}
			Let $h$ be some continuous centered Gaussian function on some compact metric space $T$, with finite positive Borel measure $m$, and such that $\sup_{s\in T}\E[h(s)^2]\le \sigma^2$. Let us define the partition function $Z=\int_{T}\exp(h(s))m(ds)$. Then $\log Z$ possesses a finite first moment and for any $c>0$
			\[\P\left(|\log Z-\E \log Z|>c\right)\le 2\exp(-c^2/4\sigma^2).\]
		\end{prop}
		\begin{proof} 
			Without loss of generality, we may assume that $m$ is a probability measure.
			Employing the Karhunen–Loève expansion (see \cite{KL}), we see that there exists a sequence of random variables $(Z_k)_{k\ge 1}$, which are i.i.d. centered Gaussians, as well as a deterministic sequence $\varphi_k\in L^2(T,m)$, such that we may write a.s. $h(s)=\sum_{k=1}^{\infty}Z_k\varphi_k(s)$ and where $\sum_{k=1}^{\infty}\varphi_k(s)^2=\E[h(s)^2]\le \sigma^2$. By the Cauchy-Schwarz inequality
			\[
			\begin{split}
				&\log\int_{T}\exp\left(\sum_{k=1}^{\infty}v_k\varphi_k(s)\right)m(ds)\\
				&\le \log\int_{T}\exp\left(\sigma\l(\sum_{k=1}^{\infty}v_k^2\r)^{1/2}\right)m(ds)\le \sigma\left(\sum_{k=1}^{\infty}v_k^2\right)^{1/2}.
			\end{split}
			\]
			In particular, we see that $(v_k)_{k\ge 1}\mapsto \log\int_{T}\exp\left(\sum_{k=1}^{\infty}v_k\varphi_k(s)\right)m(ds)$ is $\ell^2$-Lipshitz with constant less than $\sigma$.
			
			We recall that a standard Gaussian random variable satisfies the log-Sobolev inequality with constant $1$, by the Bakry-Emery criterion \cite{ledoux}. In particular, so does the sequence $(Z_k)_{k\ge 1}$. Thus by Herbst's argument \cite{ledoux} we have for any $\alpha>0$ that
			\[\log\E Z^{\alpha}\le \alpha\E\log Z+\alpha^2 \sigma^2/2.\]
			Thus the claim now follows from Markov's inequality.
		\end{proof} 
		
		Next we will need a slightly variant of this result which removes the compactness assumption.
		
		\begin{corr}
			\label{cor:Borell-TIS-free-energy-noncompact}
			Let $(h,T,m)$ be as in Proposition \ref{prop:Borell-TIS-free-energy} except that the assumption that $T$ is compact is weakened to the assumption $T$ is $\sigma$-compact. Then if we have that $\E \log Z\le \log \E Z<\infty$, we have that
			\[\P\left(|\log Z-\E \log Z|>c\right)\le 2\exp(-c^2/4\sigma^2).\]
			Moreover, let us fix an increasing sequence of compact $T_n\subseteq T$ such that $\bigcup_{n\ge 1}^{\infty}T_n=T$, and denote $Z_n=\int_{T_n} \exp(h(s))m(ds)$. Then we have that
			\[
			\label{eqn:lemma:free energy non-compact exhaustion by compact}
			\lim_{n\to \infty} \E \log Z_n=\E \log Z.
			\]
		\end{corr}
		\begin{proof}
			By Jensen's inequality we have that $\E \log Z\le \log \E Z<\infty$. We next show that $\E |\log Z|<\infty$. As $\E \log Z < \infty$, we see that it suffices to show that $\E \log Z I(Z\ge 1) <\infty$. For this, we may assume without loss of generality that $\P(Z\ge 1)>0$. Then
			\begin{align}
				\E \log Z I(Z\ge 1)&=\E[\log Z|Z\ge 1]\P(Z\ge 1)\le \log \l(\E[Z|Z\ge 1]\r)\P(Z\ge 1)\\
				&=\log \l(\frac{\E[ZI(Z\ge 1)]}{\P(Z\ge 1)}\r)\P(Z\ge 1)\le \log \E Z-\P(Z\ge 1)\log \P(Z\ge 1).
			\end{align}
			As $-x\log(x)$ is bounded for $x\in (0,1)$, we see $\E \log Z I(Z\ge 1) <\infty$ so $\E |\log Z|<\infty$.
			
			As $\E |\log Z|<\infty$ we see that $\log Z$ is a.s. finite, and by the monotone convergence theorem that $\lim_{n\to \infty}\log Z_n=\log Z$. By the dominated convergence theorem, we thus obtain (\ref{eqn:lemma:free energy non-compact exhaustion by compact}). From this we have a.s. that
			\[\lim_{n\to \infty}|\log Z_n-\E \log Z_n|=|\log Z-\E \log Z|.\]
			Thus we see that a.s.
			\[
			\liminf_{n\to \infty}I(|\log Z_n-\E \log Z_n|>\alpha)\ge I(|\log Z-\E \log Z|\ge \alpha).
			\]
			As $\P(|\log Z-\E \log Z|=\alpha)=0$, we see by Fatou's lemma that
			\[\P(\log Z-\E \log Z|\ge \alpha)\le \liminf_{n\to \infty}\P(|\log Z_n-\E \log Z_n|>\alpha)\le 2\exp(-c^2/4\sigma^2).\]
		\end{proof}
		
		Next we will need the following consequence of Gaussian integration by parts. The only difficulty is  up to the verification of convergence conditions to allow the interchange of integration and differentiation.
		
		\begin{prop}
			\label{prop:Guerra Gaussian integration}
			Let $V$ and $W$ two independent centered Gaussian processes, defined on some topological space $A$ with finite positive Borel measure $m$, with covariance functions given by $C_V$ and $C_W$, respectively. Let us assume that both $C_V$ and $C_W$ are bounded. Define the interpolating function
			\[V_t(x)=\sq{(1-t)}V(x)+\sq{t}W(x).\]
			Then we have that
			\[\frac{d}{dt}\left(\E\log\int_Ae^{V_t(x)}m(dx)\right)=\frac{1}{2}\bigg(\E\left\<C_W(x,x)-C_V(x,x)\right\>_t-\E\left\<C_W(x,x')-C_V(x,x')\right\>_t\bigg),\]
			where here $\<*\>_t$ denotes the Gibbs measure with respect to $(V_t,m)$.
		\end{prop}
		\begin{proof}
			Let $C$ be a constant bounding both the variance of $V$ and $W$. We may assume that $m$ is a probability measure. Then by Jensen's inequality we have that
			\[0\le \E \log\int_Ae^{V_t(x)}m(dx)\le \log\E \int_Ae^{V_t(x)}m(dx)\le C/2\]
			so that $\log\int_Ae^{V_t(x)}m(dx)$ has a first moment and is a.s. finite. Next, note the inequality
			\[\l|\frac{d}{dt}e^{V_t(x)}\r|\le \frac{2}{\sq{t(1-t)}}\cosh(V(x))\cosh(W(x)).\]
			Thus a similar application of Jensen's inequality shows that $|\frac{d}{dt}e^{V_t(x)}|$ is a.s. uniformly integrable with respect to $(A,m)$, so by the Leibniz rule we a.s. have that
			\[\frac{d}{dt}\log\int_Ae^{V_t(x)}m(dx)=\frac{\int_A\frac{d}{dt}V_t(x)e^{V_t(x)}m(dx)}{\int_Ae^{V_t(x)}m(dx)}=\left\<\frac{d}{dt}V_t(x)\right\>_t.\]
			An easy computation using functional Gaussian integration by parts yields that
			\[\E \left\<\frac{d}{dt}V_t(x)\right\>_t=\frac{1}{2}\bigg(\E\left\<C_W(x,x)-C_V(x,x)\right\>_t-\E\left\<C_W(x,x')-C_V(x,x')\right\>_t\bigg).\]
			The final verification requires us to show that 
			\[\frac{d}{dt}\E \log\int_Ae^{V_t(x)}m(dx)=\E \frac{d}{dt}\log\int_Ae^{V_t(x)}m(dx)=\E\left\<\frac{d}{dt}V_t(x)\right\>_t.\]
			Again we will apply the Leibniz rule, for which we verify that $|\left\<\frac{d}{dt}V_t(x)\right\>_t|$ has a uniform $1$st moment with respect to $t$. For this, we note that 
			\[\left|\left\<\frac{d}{dt}V_t(x)\right\>_t\right|\le \sqrt{t}\<|W(x)|\>_t+\sq{1-t}\<|V(x)|\>_t.\]
			We show that a uniform bound on $\<|W(x)|\>_t$, the other being identical. For this, we note that employing again functional Gaussian integration by parts we have that 
			\[
			\begin{split}
				\E\<|W(x)|\>_t^2\le \E\<W(x)^2\>_t=\frac{1}{2}\bigg(&\E\<C_W(x,x)\>_t+\sq{t}\E \< C_W(x,x)W(x)\>_t\\
				&-\sq{t}\E \< C_W(x,y)W(x)\>_t\bigg)\le 
				2C(1+\E\<|W(x)|\>_t).
			\end{split}
			\]
			This implies uniform finiteness of $\E\<|W(x)|\>_t$ by the quadratic formula.
		\end{proof}
		
		As an elementary consequence of Proposition \ref{prop:Guerra Gaussian integration} we may obtain continuity of both the free energy of the spherical models in the choice of mixing function.
		
		\begin{corr}
			\label{corr:overview:continuity of free energy in xi}
			For any choice of valid mixing functions $\b{\xi}^0$ and $\b{\xi}^1$, and let us denote the coefficients with respect to $\b{\xi}^i$ as $\b{\beta}_p^i(x)$. We have that
			\[\lim_{N\to \infty}|\Om|^{-1}N^{-1}|\E \log \ZNxi(\b{\xi}^0)-\E \log \ZNxi(\b{\xi}^1)|\le \frac{1}{|\Om|}\sum_{x\in \Om,p\ge 0}\l|\b{\beta}^0_p(x)^2-\b{\beta}^1_p(x)^2\r|.\]
		\end{corr}
		
		Finally, another incredibly useful corollary of this result is the following.
		
		\begin{corr}
			\label{corr:Guerra-bound}
			Let $V$, $W$, and $(A,m)$ be as in Proposition \ref{prop:Guerra Gaussian integration}. Then we have that 
			\[\bigg|\E\log\int_Ae^{V(x)}m(dx)-\E\log\int_Ae^{W(x)}m(dx)\bigg|\le \sup_{x,x'\in A}|C_V(x,x')-C_W(x,x')|.\]
		\end{corr}
		
		A more complicated consequence is the following result, which follows from a mixture of using Corollary \ref{corr:Guerra-bound} to deal with the Gaussian component, Taylor's theorem for the deterministic part, and then the co-area formula. For the context, let us fix some $\epsilon>0$, some finite index set $\Om$, and some $\b{q}\in (2\epsilon,\infty)^{\Om}$. We now give a simple result which allows us to compare the effect of an $\epsilon$-thickening on the partition function.
		
		\begin{prop}
			\label{prop:Guerra-disintegration-new-appendix}
			Let $(S,\mu)$ denote some finite Borel measure space, and furthermore, let us fix a choice of $N\ge 1$, deterministic differentiable function $g:S_N^{\epsilon}(\b{q})\to \R$ and a family $(V_x)_{x\in \Om}$ of a centered continuous Gaussian functions $V_x:S_N^{\epsilon}(\b{q}(x))\times S\to \R$ such that for $(u,\alpha),(u',\alpha')\in S_N^{\epsilon}(\b{q}(x))\times S$ we have that
			\[\E V_x(u,\alpha)V_x(u',\alpha')=Nf_{x,\alpha,\alpha'}(\|u\|^2_N,\|u'\|^2_N,(u,u')_N),\]
			where $f_{x,\alpha,\alpha'}$ are some fixed functions. Let us define the partition functions:
			\[Z=\int_{S_N(\b{q})\times S} e^{\sum_{x\in \Om}V_x(\b{u}(x),\alpha)+g(\b{u})}\omega(d\b{u})\mu(d\alpha),\;\; Z^{\epsilon}=\int_{S_N^{\epsilon}(\b{q})\times S} e^{\sum_{x\in \Om}V_x(\b{u}(x),\alpha)+g(\b{u})}d\b{u}\mu(d\alpha).\]
			Then we have that
			\[
			\begin{split}
				N^{-1}|\E \log Z-\E \log Z^{\epsilon}|\le&
				\sqrt{3}\epsilon\sum_{x\in \Om}\left(\sup_{\alpha,\alpha'\in S}\|\|\D f_{x,\alpha,\alpha'}\|\|_{L^\infty}\right)+N^{-1/2}\|\|\D g\|\|_{L^{\infty}}\sqrt{|\Om|\epsilon}\\
				&
				+\sum_{x\in \Om}\left(\frac{\epsilon}{2\b{q}(x)}+\frac{1}{2N}|\log(N\epsilon/\b{q}(x))|+\frac{1}{2N}|\log(N^{-1}\b{q}(x))|\right). \label{eqn:appendix:lem:e-thick}
			\end{split}
			\]
		\end{prop}
		\begin{proof}
			We begin by introducing the normalization function $S_N^{\epsilon}(\b{q})\to S_N(\b{q})$ given for $x\in \Om$ by
			\[\b{u}_{\b{q}}(x)=\frac{\sqrt{\b{q}(x)}}{\|\b{u}(x)\|_N}\b{u}(x),\]
			and for convenience we define  $V(\b{u},\alpha)=\sum_{x\in \Omega}V_{x}(\b{u}(x),\alpha)$. With this notation, we bound our quantity as the sum of three differences,
			\[N^{-1}|\E \log Z-\E \log Z^{\epsilon}|\le I+II+III,\]
			which are given by
			\[
			\begin{split}
				I&=N^{-1}\l|\E \log \int_{S_N^{\epsilon}(\b{q})\times S} e^{V(\b{u},\alpha)+g(\b{u})}d\b{u}\mu(d\alpha)-\E\log \int_{S_N^{\epsilon}(\b{q})\times S} e^{V(\b{u}_{\b{q}},\alpha)+g(\b{u})}d\b{u}\mu(d\alpha)\r|,\\
				II&=N^{-1}\l|\E \log \int_{S_N^{\epsilon}(\b{q})\times S} e^{V(\b{u}_{\b{q}},\alpha)+g(\b{u})}d\b{u}\mu(d\alpha)-\E\log \int_{S_N^{\epsilon}(\b{q})\times S} e^{V(\b{u}_{\b{q}},\alpha)+g(\b{u}_{\b{q}})}d\b{u}\mu(d\alpha)\r|,\\
				III&=N^{-1}\l|\E\log \int_{S_N^{\epsilon}(\b{q})\times S} e^{V(\b{u}_{\b{q}},\alpha)+g(\b{u}_{\b{q}})}d\b{u}\mu(d\alpha)-\E\log \int_{S_N(\b{q})\times S} e^{V(\b{u}_{\b{q}},\alpha)+g(\b{u}_{\b{q}})}\omega(d\b{u})\mu(d\alpha)\r|.
			\end{split}
			\]
			We begin by treating $I$. By Corollary \ref{corr:Guerra-bound}, we have that
			\[
			\begin{split}
				I\le \sup_{(\b{u},\alpha),(\b{u}',\alpha')\in S_N^{\epsilon}(\b{q})\times S}\bigg|\sum_{x\in \Om}&\bigg(f_{x,\alpha,\alpha'}\left(\|\b{u}(x)\|^2_N,\|\b{u}'(x)\|^2_N,(\b{u}(x),\b{u}'(x))_N\right)\\
				&-f_{x,\alpha,\alpha'}\left(\b{q}(x),\b{q}(x),(\b{u}_{\b{q}}(x),\b{u}'_{\b{q}}(x))_N\right)\bigg)\bigg|.
			\end{split}
			\]
			This may further be bounded by
			\[\sup_{\alpha,\alpha'\in S}\sum_{x\in \Om}\left(\sup_{ r,r'\in \left[\sqrt{\b{q}(x)},\sq{\b{q}(x)+\epsilon}\right],\rho\in [-1,1]}\bigg|f_{x,\alpha,\alpha'} \left(r^2,(r')^2, rr'\rho \right)-f_{x,\alpha,\alpha'} \left(\b{q}(x),\b{q}(x),\b{q}(x)\rho\right)\bigg|\right).\]
			We note that 
			\[\left\|\l(r^2-\b{q}(x),(r')^2-\b{q}(x),rr'\rho-\b{q}(x)\rho\r)\right\|\le \sqrt{3}\epsilon.\]
			So applying Taylor's theorem, we see that
			\[I\le \sqrt{3}\epsilon\sum_{x\in \Om}\sup_{\alpha,\alpha'\in S}\|\|\D f_{x,\alpha,\alpha'}\|\|_{L^\infty}.\]
			For $II$, we use the bound
			\[II\le N^{-1}\sup_{\b{u}\in S_N^{\epsilon}(\b{q})}|g(\b{u})-g(\b{u}_{\b{q}})|.\]
			Combining this with Taylor's theorem we conclude that
			\[II\le N^{-1/2}\|\|\D g\|\|_{L^{\infty}}\sup_{\b{u}\in S_N^{\epsilon}(\b{q})}\sqrt{\sum_{x\in \Om}\|\b{u}(x)-\b{q}(x)\|^2_N}\le N^{-1/2}\|\|\D g\|\|_{L^{\infty}}\sqrt{|\Om|\epsilon}.\]
			For $III$, by expanding the first integral using the co-area formula with respect to $\b{u}\mapsto \b{u}_{\b{q}}$ we have that
			\[III=N^{-1}\left|\log \left(\prod_{x\in \Om}\sq{N\b{q}(x)}\int_{1}^{\sq{1+\epsilon/\b{q}(x)}}r^{N-1}dr\right)\right|.\]
			We note that
			\[\sq{N\b{q}(x)}\int_{1}^{\sq{1+\epsilon/\b{q}(x)}}r^{N-1}dr=\sq{N^{-1}\b{q}(x)}\left((1+\epsilon/\b{q}(x))^{N/2}-1\right).\]
			Using that for $x>0$, we have that
			\[Nx\le \left((1+x)^{N}-1\right)\le \exp(Nx),\]
			we see that
			\[\frac{1}{2}\log(N\epsilon/\b{q}(x))\le \log\left(\sq{N^{-1}\b{q}(x)}\left((1+\epsilon/\b{q}(x))^{N/2}-1\right)\right)\le \frac{N\epsilon}{2\b{q}(x)}+\frac{1}{2}\log(N^{-1}\b{q}(x))\]
			and so we obtain the bound
			\[III\le \sum_{x\in \Om}\left(\frac{\epsilon}{2\b{q}(x)}+\frac{1}{2N}|\log(N\epsilon/\b{q}(x))|+\frac{1}{2N}|\log(N^{-1}\b{q}(x))|\right).\]
			Combining all of these bounds completes the proof.
		\end{proof}
		
		\pagebreak
		
		\section{Analytic Properties of the Functionals $\cal{A}$ and $\cal{B}$ \label{section:appendix:B appendix}}
		
		In this appendix we will prove some results concerning the functionals $\cal{B}$ and $\cal{A}$ introduced in (\ref{eqn:def:hat-A}) (with $\beta=1$) and (\ref{eqn:def:A true def}) above. We first begin with a result that shows that the infimum of $\cal{B}$ over $\cal{Y}$ and its subset $\cal{Y}_{fin}$ coincide.
		
		\begin{prop}
			\label{prop:A and B continuity}
			We have that
			\[
			\inf_{(\zeta,\b{\Phi})\in \mathscr{Y}_{fin}}\cal{B}(\zeta,\b{\Phi})=\inf_{(\zeta,\b{\Phi})\in \mathscr{Y}}\cal{B}(\zeta,\b{\Phi}).
			\label{eqn:A continuity result}
			\]
		\end{prop}
		
		The main tool we use to prove this is the following explicit bound.
		
		\begin{lem}
			\label{lem:appendix:continuity for the same phi}
			Let us fix some small $c,c'>0$. Then there is some large $C:=C(c',c)>0$ such that the following holds: Let $\zeta_0$ and $\zeta_1$ be probability measures on $[0,1]$ such that $\zeta_i([0,1-c])=1$, and choose some monotone increasing $\b{\Phi}:[0,1]\to [0,1]^{\Om}$ such that $\Phi_x(1-c)\ge c'$ and $(\zeta^i,\b{\Phi})\in \mathscr{Y}$. Then we have that
			\[
			|\cal{B}(\zeta_0,\b{\Phi})-\cal{B}(\zeta_1,\b{\Phi})|\le C\sup_{s\in [0,1]} \l|\zeta_0([0,s])-\zeta_1([0,s])\r|.\label{eqn:ignore-93494}
			\]
		\end{lem}
		\begin{proof}
			For convenience, let us define $\epsilon:=\sup_{s\in [0,1]} \l|\zeta_0([0,s])-\zeta_1([0,s])\r|$. To treat one term of $\cal{B}$, we see that
			\[
			\begin{split}
				\l|\int_{0}^1\zeta_0([0,s])\xi'_x(\Phi_x(s))\Phi_x'(s)ds-\int_0^1\zeta_1([0,s])\xi'_x(\Phi_x(s))\Phi_x'(s)ds\r|\le&\\
				\int_{0}^1\l|\zeta_0([0,s])\xi'_x(\Phi_x(s))\Phi_x'(s)-\zeta_1([0,s])\xi'_x(\Phi_x(s))\Phi_x'(s)\r|ds\le&\\
				\epsilon \int_0^1 \xi'_x(\Phi_x(s))\Phi'_x(s)ds=\epsilon(\xi'_x(1)-\xi'_x(\b{\Phi}_x(0)))\le & \epsilon\xi_x(1).
			\end{split}
			\label{eqn:ignore-239473987}
			\]
			Further, let us write $\b{\delta}^i$ for the function $\b{\delta}$ corresponding to the choice $(\zeta_i,\b{\Phi})$. Then as in (\ref{eqn:ignore-239473987}) we see that
			\[|\delta_x^0(u)-\delta_x^1(u)|\le \epsilon.\]
			We also that we may take $q_*=1-c'$ for both endpoints, and we have that $\delta_x^i(1-c)=1-\Phi_x(1-c)\ge c'$. Thus we have that $\b{\delta}:[0,1]\to [c',1]^{\Om}$, and as $\b{K}^D$ is uniformly differentiable on $[c',1]^{\Om}$, for some large $C_0:=C_0(c')>0$ we have that 
			\[\sup_{s\in [0,1],x\in \Om}|K^D_x(\b{\delta}^0(s))-K^D_x(\b{\delta}^1(s))|\le C_0\epsilon.\]
			Similarly we may find some $C_1>0$ such that
			\[\|(D+\b{K}^D(\b{\delta}^0(0)))^{-1}-(D+\b{K}^D(\b{\delta}^1(0)))^{-1}\|_{F}\le C_1\epsilon.\]
			Together, these show that
			\[
			|\cal{B}(\zeta_0,\b{\Phi})-\cal{B}(\zeta_1,\b{\Phi})|\le \frac{\epsilon}{2}\left(\frac{1}{|\Om|}\sum_{x\in \Om}\xi_x(1)+C_0+\frac{C_1}{|\Om|}\sum_{x,y\in \Om}|\b{h}(x)\b{h}(y)|\right),
			\]
			which completes the proof.
		\end{proof}
		
		\begin{proof}[Proof of Proposition \ref{prop:A and B continuity}]
			Let us take some $(\zeta,\b{\Phi})\in \mathscr{Y}$. Note that this pair satisfies the conditions of Lemma \ref{lem:appendix:continuity for the same phi} for some $c,c'>0$. Moreover, note that the right-hand side of (\ref{eqn:ignore-93494}) is proportional to the Wasserstein metric, and so it is well known that we may find a sequence (for example by taking suitably small quantiles approximation of $\zeta$) of probability measures $\zeta_n$, such that $\zeta_n([0,1-c])=\zeta([0,1-c])=0$, and such that
			\[\limsup_{n\to \infty}|\sup_{s\in [0,1]}|\zeta([0,s])-\zeta_n([0,s])|=0,\]
			So that in particular, we have that
			\[\limsup_{n\to \infty}|\cal{B}(\zeta,\b{\Phi})-\cal{B}(\zeta_n,\b{\Phi})|=0.\]
			Thus we have that $\cal{B}(\zeta,\b{\Phi})\ge  \inf_{(\zeta,\b{\Phi})\in \mathscr{Y}_{fin}}\cal{B}(\zeta,\b{\Phi})$, which suffices to complete the proof.
		\end{proof}
		
		The next result we need shows that we have that $\cal{B}(\zeta,\b{Phi})=\cal{B}(\zeta,\b{\Phi}')$ when $\b{\Phi}(s)=\b{\Phi}'(s)$ for $s\in \supp(\zeta)$. For this we need a technical lemma.
		
		\begin{lem}
			\label{lem:appendix:basic integration by parts}
			Let $\phi:[0,1]\to [0,1]$ be a non-decreasing and absolutely continuous function, $\mu$ a finite measure on $[0,1]$, and $f:[0,1]\to \R$ integrable. Then we have that
			\[\int_0^1 \mu([0,x])f(\phi(x))\phi'(x)dx=\int_{\phi(0)}^{\phi (1)}\phi_*(\mu)([0,x])f(x)dx.\]
		\end{lem}
		\begin{proof}
			By definition $\phi_*(\mu)([0,x])=\mu(\phi^{-1}([0,x]))$, and as this is monotone, and thus integrable, we see by employing a change of variables with respect to function $\phi$, we have that
			\[
			\int_{\phi(0)}^{\phi (1)}\mu(\phi^{-1}([0,x]))f(x)dx=\int_{0}^{1}\mu(\phi(\phi^{-1}([0,x])))f(\phi(x))\phi'(x)dx.\label{eqn:ignore-21398348}
			\]
			To compare this to the desired integral, we note that if $\phi$ is strictly increasing in a neighborhood of $x$, we have that 
			\[\phi(\phi^{-1}([0,x]))=\{y\in [0,1]:\phi(y)\le \phi(x)\}=[0,x].\]
			In particular, if $\phi'(x)$ exists and is positive, this holds. However, as $\phi'$ must be either positive or zero, and exists a.e, we see that a.e. we have that
			\[\mu(\phi(\phi^{-1}([0,x])))\phi'(x)=\mu([0,x])\phi'(x).\]
			Using this to replace the integrand of (\ref{eqn:ignore-21398348}) completes the proof.
		\end{proof}
		
		We can now give the second lemma needed to prove Proposition \ref{prop:A and B continuity}.
		
		\begin{prop}
			\label{prop:general:parisi doesn't depend away from support}
			Let us fix $(\zeta,\b{\Phi}),(\zeta,\b{\Phi}')\in\mathscr{Y}$. Let us assume that for any $s\in \supp(\zeta)$, we have that $\b{\Phi}(s)=\b{\Phi}'(s)$. Then 
			\[\cal{B}(\zeta,\b{\Phi})=\cal{B}(\zeta,\b{\Phi}').\label{eqn:diagonal:ignore-1842}\]
		\end{prop}
		\begin{proof}
			To begin we note that our assumptions imply that $(\b{\Phi}_*)(\zeta)\disteq (\b{\Phi}'_*)(\zeta)$. In particular by applying Lemma \ref{lem:appendix:basic integration by parts} we see that
			\[
			\begin{split}
				\int_0^1\zeta([0,q])\xi'_x(\Phi_x(q))\Phi_x(dq)&=\int_0^1(\Phi_x)_*(\zeta)([0,q])\xi'_x(q)dq,\\
				\delta_x(s)&=\int_{\Phi_x(s)}^1(\Phi_x)_*(\zeta)([0,s])ds.\label{eqn:diagonal:ignore-1916}
			\end{split}
			\]
			In particular, if we write $\b{\delta}$ and $\b{\delta}'$ for the functions corresponding to $\b{\Phi}$ and $\b{\Phi}'$, then $\b{\delta}=\b{\delta}'$. With these observations, the only term that is not clearly invariant under $\b{\Phi}\mapsto \b{\Phi}'$ is
			\[\frac{1}{2|\Om|}\sum_{x\in \Om}\int_0^{q_*}\b{K}_x^D(\b{\delta}(s))\Phi_x'(s)ds.\]
			It is clear that the integrands coincide on $\supp(\zeta)$, so let us consider the integral over some open connected component of $\supp(\zeta)^c\cap (0,q_*)$, say $(a,b)$. We note that $\b{\Phi}(s)=\b{\Phi}'(s)$ for $s\in \{a,b\}$. Let us write $\zeta_0=\zeta([0,a])=\zeta([0,b))>0$ for simplicity. Then for $s\in (a,b)$,
			\[\delta_x(s)=\delta_x(b)+\zeta_0(\Phi_x(b)-\Phi_x(s)),\]
			so that recalling $|\Om|\D \Lambda^D=\b{K}^D$, we see that
			\[\frac{1}{|\Om|}\sum_{x\in \Om}\int_a^{b}\b{K}_x^D(\b{\delta}(s))\Phi_x'(s)ds=-\frac{1}{\zeta_0}\left(\Lambda^D(\b{\delta}(b))-\Lambda^D(\b{\delta}(a))\right).\]
			As $\b{\delta}=\b{\delta}'$, we see this is invariant, which completes the proof of (\ref{eqn:diagonal:ignore-1842}).
		\end{proof}
		
		Before proceeding, we note that the analogous result for $\cal{A}$ follows from a similar proof, which is omitted.
		
		\begin{prop}
			\label{prop:general: for B, parisi doesn't depend away from support}
			Let us fix $(\zeta,\b{\Phi}),(\zeta,\b{\Phi}')\in\mathscr{Y}$. Let us assume that for any $s\in \supp(\zeta)$, we have that $\b{\Phi}(s)=\b{\Phi}'(s)$. Then 
			\[\inf_{\b{b}}\cal{A}(\zeta,\b{\Phi},\b{b})=\inf_{\b{b}}\cal{A}(\zeta,\b{\Phi}',\b{b}).\label{eqn:diagonal:ignore-1842-b}\]
		\end{prop}
		
		Finally we end this section with a continuity result for $\cal{B}$, which is similar to Corollary \ref{corr:overview:continuity of free energy in xi} above.
		
		\begin{prop}
			\label{prop:overview:continuity of parisi in xi}
			For any choice of valid mixing functions $\b{\xi}^0$ and $\b{\xi}^1$ let us denote the coefficients of $\xi^i_x$ as $\b{\beta}^i_p(x)$, and let us write $\cal{B}_{\b{\xi}}$ to indicate the choice of mixing function in $\cal{B}$. Then we have that
			\[\sup_{(\zeta,\b{\Phi})\in \mathscr{Y}}|\cal{B}_{\b{\xi}^0}(\zeta,\b{\Phi})-\cal{B}_{\b{\xi}^1}(\zeta,\b{\Phi})|\le \frac{1}{2|\Om|}\sum_{x\in \Om,p\ge 0}\l|\b{\beta}^0_p(x)^2-\b{\beta}^1_p(x)^2\r|.\]
			In particular,
			\[\left|\inf_{(\zeta,\b{\Phi})\in \mathscr{Y}}\cal{B}_{\b{\xi}^0}(\zeta,\b{\Phi})-\inf_{(\zeta,\b{\Phi})\in \mathscr{Y}}\cal{B}_{\b{\xi}^1}(\zeta,\b{\Phi})\right|\le \frac{1}{2|\Om|}\sum_{x\in \Om,p\ge 0}\l|\b{\beta}^0_p(x)^2-\b{\beta}^1_p(x)^2\r|.\]
		\end{prop}
		\begin{proof}
			We may assume that $\beta^i_p\ge 0$ for each $(i,p)$. We compute then that
			\[
			\begin{split}
				|\cal{B}_{\b{\xi}^0}(\zeta,\b{\Phi})-\cal{B}_{\b{\xi}^1}(\zeta,\b{\Phi})|= &\\
				\frac{1}{2|\Om|}\sum_{x\in \Om}\l|\int_0^1\zeta([0,u])\l((\xi^0_x)'(\Phi_x(u))-(\xi^1_x)'(\Phi_x(u))\r)\Phi'_x(u)du\r|\le & \\
				\frac{1}{2|\Om|}\sum_{x\in \Om}\int_0^1|(\xi^0_x)'(u)-(\xi^1_x)'(u)|du
				\le &\\
				\frac{1}{2|\Om|}\sum_{x\in \Om,p\ge  0}\l|\b{\beta}^0_p(x)^2-\b{\beta}^1_p(x)^2\r|.
			\end{split}
			\]
		\end{proof}
		
		\pagebreak
		
		\section{Properties of Functions $\bold{K}^D$ and $\Lambda^D$ \label{section:appendix:functional}}
		
		Let us fix a positive semi-definite matrix $D\in \R^{\Om\times \Om}$, where $\Omega$ is some finite set, and recall the functions $\b{K}^D$ and $\Lambda^D$ defined by (\ref{eqn:def:KD}) and (\ref{eqn:def:Lambda}). The purpose of this appendix will be to prove the analytic results we need about these functions above.
		
		To begin, let use the notation $\odot$ to denote the Hadamard product (i.e. $[A\odot B]_{ij}=A_{ij}B_{ij}$). By the Schur Product Theorem, this defines an operation that takes two positive definite matrices to another positive definite matrix. Moreover, for $n$-by-$n$ positive semi-definite $A$ and $B$, we have that 
		\[s_1(A\odot B)\le s_1(A)s_1(B),\;\;\;\; s_{|\Om|}(A\odot B)\ge s_{|\Om|}(A)s_{|\Om|}(B).\]
		We will additionally employ the shorthand $A^{\odot 2}=A\odot A$. The following proposition comprises the facts about $\Lambda^D$ and $\b{K}^{D}$ we will require.
		
		\begin{prop}
			\label{prop:appendix:K and Lambda}
			$\Lambda^D$ is a strictly concave function on $(0,\infty)^{\Om}$, occurring as the Legendre transform of the function strictly concave function $\frac{1}{|\Om|}\log \det(D+\b{u})$. The function $\b{K}^D$ is a well-defined homeomorphism of $(0,\infty)^{\Om}$ onto $\{\b{u}\in \R^{\Om}:D+\b{u}>0\}$. Moreover, $|\Om|\D \Lambda^D(\b{u})=\b{K}^D(\b{u})$ and
			\[\D \b{K}^D(\b{u})=-\bigg(\big((D+\b{K}^D(\b{u}))^{-1}\big)^{\odot 2}\bigg)^{-1}.\label{eqn:appendix:DK-computation}\]
			Finally, both $\|\b{K}^D(\b{u})\|$ and $-\Lambda^D(\b{u})$ grow to $\infty$ as $\b{u}\to \partial (0,\infty)^{\Om}$.
		\end{prop}
		\begin{proof}
			In the course of this proof we denote $\cal{X}=\{\b{u}\in \R^{\Om}:D+\b{u}>0\}$. We first compute that for $x,y\in \Om$, we have that
			\[\D_x\log \det(D+\b{u})=[(D+\b{u})^{-1}]_{xx},\;\;\; \D_{xy}^2\log \det(D+\b{u})=-[(D+\b{u})^{-1}]_{xy}^2.\]
			In particular $\D^2\log \det(D+\b{u})=-[(D+\b{u})^{-1}]^{\odot 2}$, so for $\b{u}\in \cal{X}$, we see that $\log \det(D+\b{u})$ is strictly concave. This and our computation of $\D\log \det(D+\b{u})$ makes it clear that $\Lambda^D(\b{u})$ coincides with the Legendre transform of $\frac{1}{|\Om|}\log \det(D+\b{u})$ when $\b{K}^D$ is well-defined.
			
			We next show that $\log \det(D+\b{u})$ is essentially smooth in the sense of \cite{convexanalysis}: if $\b{u}_n$ is a sequence in $\cal{X}$, such that $\b{u}_n$ converges to some $\b{u}\in \partial\cal{X}$, we have that $\|\D\log \det(D+\b{u}_n)\|\to \infty$. As all entries of $\D\log \det(D+\b{u}_n)$ are positive, we may equivalently show that
			\[\frac{1}{|\Om|}\sum_{x\in \Om}\D_x\log \det(D+\b{u}_n)=\tr((D+\b{u}_n)^{-1})\to \infty.\]
			For this note that $\partial\cal{X}$ consists of non-invertible positive semi-definite matrices, so in particular, $\det(D+\b{u}_n)\to \det(D+\b{u})=0$, so that $s_{|\Om|}(D+\b{u}_n)\to 0$. In particular, as $\tr((D+\b{u}_n)^{-1})\ge s_{|\Om|}(D+\b{u}_n)^{-1}$, we conclude that $\tr((D+\b{u}_n)^{-1})\to \infty$ as desired. Note this also shows that $\log \det(D+\b{u}_n)\to -\infty$. In particular, this shows that for each $\alpha\in \R$, the level set $\{\b{v}:\log \det(D+\b{v})\ge \alpha\}$ is closed, so that $\log \det(D+\b{u})$ is a closed concave function. 
			
			Finally, as $\log \det(D+\b{v})$ is strictly concave, we see that $\log \det(D+\b{v})$ is a concave function of Legendre-type in the sense of \cite{convexanalysis}. In particular, we may apply Theorem 26.5 of \cite{convexanalysis} which shows a few things. First it shows that the Legendre and Fenchel transforms of $\log \det(D+\b{u})$ coincide. Next, it shows that $\b{K}^D$ is a well-defined homeomorphism from the domain of the Fenchel conjugate of $\log \det(D+\b{u})$, i.e.
			\[\cal{Y}:=\bigg\{\b{u}\in \R^{\Om}:\inf_{\b{v}\in \cal{X}}\left(\sum_{x\in \Om}\b{v}(x)\b{u}(x)-\log \det(D+\b{v})\right)>-\infty\bigg\},\]
			onto $\cal{X}$. On this domain, the Legendre transform of $\frac{1}{|\Om|}\log \det(D+\b{v})$, which coincides with $\Lambda^D$, is itself an essentially smooth and closed strictly concave function. Finally, we have that $|\Om|\D \Lambda^D(\b{u})=\b{K}^D(\b{u})$ and $|\Om|\D^2 \Lambda^D(\b{u})=\D\b{K}^D(\b{u})=(\D^2 \log(D+\b{K}^D(\b{u})))^{-1}$. In particular, if we verify that $\cal{Y}=(0,\infty)^{\Om}$, all of the remaining claims immediately follow from these statements. 
			
			For this note that as $[(D+\b{u})^{-1}]_{xx}>0$ for $\b{u}\in \cal{X}$ and $x\in \Om$, we have that $\cal{Y}\subseteq (0,\infty)^{\Om}$. On the other hand, for $\b{v}\in \cal{X}$, we have that 
			\[\log\det(D+\b{v})\le \sum_{x\in \Om}\log(s_1(D)+\b{v}(x)),\]
			so that
			\[\inf_{\b{v}\in \cal{X}}\sum_{x\in \Om}\b{v}(x)\b{u}(x)-\log(D+\b{v})\ge \inf_{\b{v}\in \cal{X}}\sum_{x\in \Om}(\b{v}(x)\b{u}(x)-\log(s_1(D)+\b{v}(x)))>-\infty.\]
			as the function $x\mapsto ax-\log(b+x)$ is bounded below for any $a,b>0$. This shows that $(0,\infty)^{\Om}\subseteq\cal{Y}$, so $(0,\infty)^{\Om}=\cal{Y}$,  completing the proof.
		\end{proof}
		
		We will also need the following more technical result, which shows that the Jacobian of $\b{K}^D$ is dominated by its entries along the diagonal when one approaches a point on the boundary of its domain.
		
		\begin{theorem}
			\label{theorem:appendix:K lemma hard}
			Fix $K>0$. Then there is small $\epsilon>0$ such that for any $\b{u}\in (0,K)^{\Om}$ and $x\in \Om$, we have that $\D_x \b{K}^D_x(\b{u})\le - \epsilon \b{u}(x)^{-2}$. Moreover, for $y\in \Om\setminus \{x\}$ we have that $|\D_x \b{K}^D_y(\b{u})|\le \epsilon^{-1}$.
		\end{theorem}
		
		Before proving Theorem \ref{theorem:appendix:K lemma hard}, we will first prove two technical results required for its proof. The first studies the somewhat obscure matrix transform which occurs in the Jacobian of $\b{K}^D$.
		
		\begin{lem}
			\label{lem:appendix: technical lemma}
			Fix a symmetric matrix $M\in \R^{\Om\times \Om}$, such that $M_{xx}=0$ for $x\in \Om$. Additionally, fix small $\epsilon>0$. Then for $\b{u}\in [\epsilon,\infty)^\Om$, let $M(\b{u})\in \R^{\Om\times\Om}$ be the matrix defined by $M(\b{u})_{xy}=\b{u}(x)^{-1/2}\b{u}(y)^{-1/2}M_{xy}$. Consider the function $F:[\epsilon,\infty)^\Om\to \R^{\Om\times \Om}$ given by
			\[F_M(\b{u})=\bigg(\big((I+M(\b{u}))^{-1}\big)^{\odot 2}\bigg)^{-1}.\]
			Then there is large $C>0$, such that for any $\b{u}\in [\epsilon,\infty)^\Om$, with $I+M(\b{u})\ge \epsilon I$, we have for any $x,y\in \Om$, with $x\neq y$ that
			\[|[F_M(\b{u})]_{xy}|\le C \b{u}(x)^{-1}\b{u}(y)^{-1} \text{ and  } [F_{M}(\b{u})]_{xx}\ge C^{-1}.\label{eqn:ignore-1123}\]
		\end{lem}
		
		\begin{proof}
			In the course of this proof, $C>0$ will be a large constant, allowed to change line by line, but always independent of $\b{u}$. For convenience, we will write $F_M(\b{u})=F(\b{u})$ and $n=|\Om|$.
			
			To begin we note that 
			\[[F(\b{u})]_{xx}\ge s_n(F(\b{u}))=(s_1((I+M(\b{u}))^{-1})^{\odot 2})^{-1}\ge s_n(I+M(\b{u}))^2\ge \epsilon^2,\]
			which demonstrates the second claim of (\ref{eqn:ignore-1123}). We now focus our attention on the first claim of (\ref{eqn:ignore-1123}). To begin, let us notate the terms of the square $\{x,y\}$-by-$\Omega\setminus\{x,y\}$ block decomposition of $M$ as
			\[M=\begin{bmatrix}
				A & B^t\\
				B & D
			\end{bmatrix}.\]
			We will use the notations $A(\b{u})$, $B(\b{u})$, and $D(\b{u})$ to denote the corresponding submatrices of $M(\b{u})$. We will now compute $(I+M(\b{u}))^{-1}$ in block form using the Schur complement formula, but for this, we will need to establish some notation. Namely let:
			\[Z(\b{u}(x),\b{u}(y))=
			\begin{bmatrix}
				\b{u}(x)^{-1}& (\b{u}(x)\b{u}(y))^{-1/2}\\
				(\b{u}(x)\b{u}(y))^{-1/2} & \b{u}(y)^{-1}\\
			\end{bmatrix}, \; L(\b{u})=A-B^t(I+D(\b{u}))^{-1}B.
			\]
			\[R(\b{u})=(I+Z(\b{u}(x),\b{u}(y))\odot L(\b{u}))^{-1}, \; \; \; S(\b{u})=-(I+D(\b{u}))^{-1}B(\b{u})R(\b{u}),\]
			\[T(\b{u})=(I+D(\b{u}))^{-1}+(I+D(\b{u}))^{-1}B(\b{u})R(\b{u})B(\b{u})^t(I+D(\b{u}))^{-1}.\]
			Then we have the following square $\{x,y\}$-by-$\Omega\setminus\{x,y\}$ block decomposition:
			\[(I+M(\b{u}))^{-1}=\begin{bmatrix}
				R(\b{u}) & S(\b{u})^t\\
				S(\b{u}) & T(\b{u})
			\end{bmatrix}.\]
			We now define the following $\{x,y\}$-by-$\{x,y\}$ matrix:
			\[Q(\b{u})=R(\b{u})^{\odot 2}-(S(\b{u})^{\odot 2})^t(T(\b{u})^{\odot 2})^{-1}(S(\b{u})^{\odot 2}).\label{eqn:misc:Q-def}\]
			By employing Schur's complement formula again, we see that the $\{x,y\}$-by-$\{x,y\}$ block of $F(\b{u})$ is given by $Q(\b{u})^{-1}$. From this we see that \[[F(\b{u})]_{xy}=\det(Q(\b{u}))^{-1}[Q(\b{u})]_{xy},\] 
			which shows that the first claim of (\ref{eqn:ignore-1123}) follows from the following two claims:
			\[|[Q(\b{u})]_{xy}|\le C\b{u}(x)^{-1}\b{u}(y)^{-1},\;\;\; \det(Q(\b{u}))\ge C^{-1}.\label{eqn:appendix:lemma:reduced}\]
			We first show the second claim, which is equivalent to showing that $\det(Q(\b{u})^{-1})\le C$. We observe the crude bound: $\det(Q(\b{u})^{-1})\le [F(\b{u})]_{xx}[F(\b{u})]_{yy}\le s_1(F(\b{u}))^2$. We compute that 
			\[s_1(F(\b{u}))=s_n\bigg(\big((I+M(\b{u}))^{-1}\big)^{\odot 2}\bigg)^{-1}\le s_n\big((I+M(\b{u}))^{-1}\big)^{-2}=s_1(I+M(\b{u}))^2.\]
			Combined with the trivial bound $s_1(I+M(\b{u}))\le n\tr(I+M(\b{u}))=n$, establishes the second claim of (\ref{eqn:appendix:lemma:reduced}). For the first claim, we observing the form (\ref{eqn:misc:Q-def}), than it suffices to show the following claims
			\[
			\begin{split}
				[R(\b{u})]_{xy}&\le C\b{u}(x)^{-1/2}\b{u}(y)^{-1/2},\\
				[(S(\b{u})^{\odot 2})^t(T(\b{u})^{\odot 2})^{-1}(S(\b{u})^{\odot 2})]_{xy}&\le C\b{u}(x)^{-1}\b{u}(y)^{-1}.\label{eqn:appendix:lemma:reduced 2}
			\end{split}
			\]
			We first focus on the first claim. As $R(\b{u})$ is a submatrix of $(I+M(\b{u}))^{-1}$, a similar argument shows that $\det(R(\b{u}))\ge C^{-1}$. Moreover, we may explicitly write $R(\b{u})^{-1}$ as
			\[R(\b{u})^{-1}=\begin{bmatrix}
				1+\b{u}(y)^{-1}L_{yy}(\b{u})& \b{u}(x)^{-1/2}\b{u}(y)^{-1/2}L_{xy}(\b{u})\\
				\b{u}(x)^{-1/2}\b{u}(y)^{-1/2}L_{xy}(\b{u})	& 1+\b{u}(x)^{-1}L_{xx}(\b{u})
			\end{bmatrix}.\]
			It is clear that the entries of $L(\b{u})$ are uniformly bounded, which together with the bound $\det(R(\b{u}))\ge C^{-1}$ shows the first claim of (\ref{eqn:appendix:lemma:reduced 2}).
			
			To show the second claim, we will first study $(T(\b{u})^{\odot 2})^{-1}$. For this, note that the second term in the definition of $T(\b{u})$ is positive definite, so that
			\[s_1((T(\b{u})^{\odot 2})^{-1})\le s_n(T(\b{u}))^{-2}\le s_n((I+D(\b{u}))^{-1})^{-2}= s_1(I+D(\b{u}))^{2},\]
			which we showed before was bounded. In particular, this shows that the entries of $(T(\b{u})^{\odot 2})^{-1}$ are uniformly bounded. We now note that if we define \[W(\b{u})=-(I+D(\b{u}))^{-1}B(\b{u}),\]
			it is easily checked that $\|W(\b{u})e_x\|\le C\b{u}(x)^{-1/2}$ and $\|W(\b{u})e_y\|\le C\b{u}(y)^{-1/2}$. On the other hand, we have that
			\[S(\b{u})e_1=[R(\b{u})]_{xx}W(\b{u})e_x+[R(\b{u})]_{xy}W(\b{u})e_y.\]
			So using that $|[R(\b{u})]_{xy}|\le C\b{u}(x)^{-1/2}\b{u}(y)^{-1/2}$, we see that $\|S(\b{u})e_x\|\le C\b{u}(x)^{-1/2}$. Similarly we see that $\|S(\b{u})e_y\|\le C\b{u}(y)^{-1/2}$. Combining this with the boundedness of $(T(\b{u})^{\odot 2})^{-1}$, we conclude the second claim of (\ref{eqn:appendix:lemma:reduced 2}).
		\end{proof}
		
		We also need the following bound, which will allow us to control the size of diagonal terms of a positive definite matrix $A$, purely in terms of the behavior of the diagonal entries of $A^{-1}$ and the off-diagonal entries of $A$.
		
		\begin{lem}
			\label{lem:appendix:sum lemma}
			Fix positive definite $A\in \R^{\Om\times \Om}$. Then
			\[\sum_{x\in \Om}A_{xx}[A^{-1}]_{xx}\le |\Om|+\l(\sum_{x,y\in \Om, x\neq y}A_{xy}^2\r)\left(|\Om|\tr(A^{-1})\right)^2.\]
		\end{lem}
		\begin{proof}
			It is sufficient to show that for all fixed $z\in \Om$
			\[A_{zz}[A^{-1}]_{zz}\le 1+\l(\sum_{x\in \Om\setminus\{z\}}A_{xz}^2 \r)\left(|\Om|\tr(A^{-1})\right)^2.\]
			Let us denote the terms in the block decomposition of $A$ with square $\{z\}$ and $\Om\setminus\{z\}$
			\[A=\begin{bmatrix}
				A_{zz}& A_z^t\\
				A_z& A^z
			\end{bmatrix}.\label{eqn:ignore-551}\]
			By the Shur's complement formula we have that $[A^{-1}]_{zz}=(A_{zz}-A_z^t(A^z)^{-1}A_z)^{-1}$, so using the formula $x(x-y)^{-1}=1+y(x-y)^{-1}$, we may write
			\[A_{zz}[A^{-1}]_{zz}=1+A_z^t(A^z)^{-1}A_z(A_{zz}-A_z^t(A^z)^{-1}A_z)^{-1}.\label{eqn:appendix:ignore-454}\]
			Using that $s_1((A^z)^{-1})\le s_1(A^{-1})\le |\Om|\tr(A^{-1})$, we see that
			\[A_z^t(A^z)^{-1}A_z\le |\Om|\tr(A^{-1})\sum_{x\in \Om\setminus\{z\}}A_{xz}^2.\]
			Combined with the bound $(A_{zz}-A_z^t(A^z)^{-1}A_z)^{-1}=[A^{-1}]_{zz}\le |\Om|\tr(A^{-1})$, we see that (\ref{eqn:appendix:ignore-454}) yields the desired result.
		\end{proof}
		
		\begin{proof}[Proof of Theorem \ref{theorem:appendix:K lemma hard}]
			For $\b{u}\in (0,\infty)^{\Om}$, we define a symmetric matrix $G_{\b{u}}\in\R^{\Om\times \Om}$, given for $z,w\in \Om$ by 
			\[[G_{\b{u}}]_{zw}=(D_{zz}+K_z(\b{u}))^{-1/2}(D_{ww}+K_w(\b{u}))^{-1/2}D_{zw}(1-\delta_{zw}).\label{eqn:appendix:ignore-621}\]
			To explain the importance of this matrix, let us further define $G\in \R^{\Om\times \Om}$ by $[G]_{zw}=D_{zw}(1-\delta_{zw})$, and $\b{d}\in \R^{\Om}$ by $\b{d}(z)=D_{zz}$. Then in the notation of Lemma \ref{lem:appendix: technical lemma}, we have that $G_{\b{u}}=F_G(\b{d}+\b{K}^D(\b{u}))$. We also note that by Proposition \ref{prop:appendix:K and Lambda}, and again in the notation of Lemma \ref{lem:appendix: technical lemma},
			\[\D_z K_w(\b{u})=-(D_{zz}+K_z(\b{u}))(D_{ww}+K_w(\b{u}))F_{G}(\b{d}+\b{K}^D(\b{u})).\label{eqn:appendix:ignore-1344}\]
			We seek to apply Lemma \ref{lem:appendix: technical lemma} to (\ref{eqn:appendix:ignore-1344}), so we begin by verifying its conditions. The first statement we will show is that for any $\b{u}\in (0,K)^{\Om}$, we that $\b{d}+\b{K}^D(\b{u})\in (K^{-1},\infty)^{\Om}$.
			
			For this take $z\in \Om$, and let us notate the terms in the block decomposition of $D$ with square $\{z\}$-by-$\Om\setminus\{z\}$ blocks as
			\[D=\begin{bmatrix}
				D_{zz}& D_z^t\\
				D_z& D^z
			\end{bmatrix}.\]
			Similarly, let $\b{K}^z\in \R^{\Om\setminus \{z\}}$ denote the vector obtained from $\b{K}$ by omitting the entry $K_z$. We note by Schur's complement formula that
			\[([(D+\b{K}^D(\b{u}))^{-1}]_{zz})^{-1}=D_{zz}+K_z(\b{u})-D_z^t(D^z+\b{K}^z(\b{u}))^{-1}D_z,\]
			This implies, in particular, that
			\[\b{u}(z)^{-1}=([(D+\b{K}^D(\b{u}))^{-1}]_{zz})^{-1}\le D_{zz}+K_z(\b{u}).\label{eqn:appendix:u minus lower bound}\]
			This implies the desired claim that $\b{d}+\b{K}^D(\b{u})\in (K^{-1},\infty)^{\Om}$ if $\b{u}\in (0,K)^{\Om}$. Next we will show that there is $\epsilon>0$ such that $I+G_{\b{u}}\ge \epsilon I$. We first note that as $I+G_{\b{u}}$ is congruent to $D+\b{K}^D(u)$, it is positive definite, leaving us to provide a lower bound for $s_{|\Om|}(I+G_{\b{u}})$. We note that
			\[s_{|\Om|}(I+G_{\b{u}})^{-1}=s_{1}((I+G_{\b{u}})^{-1})\le |\Om|\tr((I+G_{\b{u}})^{-1})\le \sum_{z\in \Om}\b{u}(z)[D_{zz}+K_z(\b{u})].\]
			Recalling that $\b{u}(z)=[(D+\b{K}^D(\b{u}))^{-1}]_{zz}$, we see by Lemma \ref{lem:appendix:sum lemma}, that the right-hand side is further less than
			\[|\Om|^2\tr((D+\b{K}^D(\b{u}))^{-1})^2\sum_{z,w\in \Om,z\neq w}D_{zw}^2\le |\Om|^2K^2\sum_{z,w\in \Om,z\neq w}D_{zw}^2.\]
			In particular, this shows that \[s_{|\Om|}(I+G_{\b{u}})\ge \l(|\Om|^2K^2\sum_{z,w\in \Om,z\neq w}D_{zw}^2\r)^{-1}.\] 
			Thus we may now apply Lemma \ref{lem:appendix: technical lemma} to conclude there is large $C>0$ such that for $\b{u}\in (0,K)^{\Om}$, we have that
			\[|[F_{G}(\b{d}+\b{K}^D(\b{u}))]_{xy}|\le C(D_{xx}+K^D_x(\b{u}))^{-1}(D_{yy}+K^D_y(\b{u}))^{-1},\;\;\;\; [F_{G}(\b{d}+\b{K}^D(\b{u}))]_{xx}\ge C^{-1}.\]
			These statements, combined with (\ref{eqn:appendix:ignore-1344}) and (\ref{eqn:appendix:u minus lower bound}), yield the desired claims.
		\end{proof}
		
	\end{appendices}
	
	\pagebreak

\end{document}